\documentclass[11pt, reqno]{amsart}
\usepackage{etex}
\usepackage{textcmds} 
\usepackage{amsmath, amssymb, amsfonts, amstext, verbatim, amsthm, mathrsfs}
\usepackage{microtype}
\usepackage[all]{xy}
\usepackage[modulo]{lineno}
\usepackage[usenames]{color}
\usepackage{aliascnt}
\usepackage{enumitem}
\usepackage{xspace}
\usepackage{amsfonts}
\usepackage[centertags]{amsmath}
\usepackage{rotating}
\usepackage[margin=3.5cm]{geometry}
\usepackage{dsfont}
\usepackage{bm}
\usepackage{color}
\usepackage{subfigure}
\usepackage{amsmath}
\usepackage{array}
\usepackage[all]{xy}
\usepackage{euscript}
\usepackage[T1]{fontenc}
\usepackage{mathbbol}
\usepackage{nccmath}
\usepackage{marginnote}
\usepackage{stmaryrd}
\usepackage{youngtab}
\usepackage{tikz,pgfplots}
\usepackage{blkarray}
\usepackage{ulem}

\usepackage{graphics,graphicx}  

\let\oldtocsection=\tocsection
\let\oldtocsubsection=\tocsubsection

\renewcommand{\tocsection}[2]{\hspace{0em}\oldtocsection{#1}{#2}}
\renewcommand{\tocsubsection}[2]{\hspace{1em}\oldtocsubsection{#1}{#2}}

\usepackage[backref,colorlinks=true,linkcolor=black,citecolor=blue,urlcolor=blue,citebordercolor={0 0 1},urlbordercolor={0 0 1},linkbordercolor={0 0 1}]{hyperref} 

\tikzset{node distance=3cm, auto}

\makeatletter
\def\@secnumfont{\bfseries}
\def\section{\@startsection{section}{1}%
  \z@{.7\linespacing\@plus\linespacing}{.5\linespacing}%
  {\normalfont\Large\bfseries}}
\def\subsection{\@startsection{subsection}{2}%
  \z@{.5\linespacing\@plus.7\linespacing}{-.5em}%
  {\normalfont\large\bfseries}}
  \def\subsubsection{\@startsection{subsubsection}{3}%
  \z@{.5\linespacing\@plus.7\linespacing}{-.5em}%
  {\normalfont\bfseries}}
\makeatother

\newtheorem{thm}{Theorem}[subsection]
\newtheorem{lemma}[thm]{Lemma}
\newtheorem{prop}[thm]{Proposition}
\newtheorem{cor}[thm]{Corollary}

\newtheorem{rmk}[thm]{Remark}

\newtheorem{definition}[thm]{Definition}

\newtheorem{example}[thm]{Example}
\theoremstyle{remark}
\numberwithin{equation}{subsection} 

\numberwithin{figure}{section}
\numberwithin{table}{subsection}

\newcommand{\eend}{{\rm end}}

\newcommand{\oCP}{{\overline{{\C}P}}\!\,}
\newcommand{\id}{{\rm id}}

\newcommand{\Gg}{\mathcal{G}}

\newenvironment{itemlist}
   { \begin{list} {$\bullet$}
         { \setlength{\topsep}{.5ex}  \setlength{\itemsep}{.5ex} \setlength{\leftmargin}{2.5ex} } }
   { \end{list} }

\newcommand{\bbm}{{\bf{m}}}
\newcommand{\bw}{{\bf{w}}}

\newcommand{\bx}{{\bf{x}}}

\newcommand{\bE}{{\bf{E}}}
\newcommand{\bB}{{\bf{B}}}
\newcommand{\bP}{{\bf{P}}}

   \newcommand{\PSL}{{\rm PSL}}

\newcommand{\NI}{{\noindent}}

\newcommand{\Ss}{{\mathcal S}}

\newcommand{\ov}{\overline}
\newcommand{\al}{{\alpha}}

\newcommand{\be}{{\beta}}

\newcommand{\om}{{\omega}}
\newcommand{\de}{{\delta}}

\newcommand{\ka}{{\kappa}}
\newcommand{\la}{{\lambda}}

\newcommand{\si}{{\sigma}}

\newcommand{\MS}{{\medskip}}

\newcommand{\er}{{\Diamond}}
\newcommand{\Z}{\mathbb{Z}}

\newcommand{\R}{\mathbb{R}}
\newcommand{\Q}{\mathbb{Q}}
\newcommand{\C}{\mathbb{C}}

\newcommand{\eps}{\varepsilon}
\newcommand{\Ff}{\mathcal{F}}

\newcommand{\pt}{pt}

\newcommand{\acc}{\mathrm{acc}}
\newcommand{\sembeds}{\stackrel{s}{\hookrightarrow}}
\newcommand{\se}{\stackrel{s}{\hookrightarrow}}

\makeatletter
\newcommand{\dashover}[2][\mathop]{#1{\mathpalette\df@over{{\dashfill}{#2}}}}
\newcommand{\fillover}[2][\mathop]{#1{\mathpalette\df@over{{\solidfill}{#2}}}}
\newcommand{\df@over}[2]{\df@@over#1#2}
\newcommand\df@@over[3]{%
  \vbox{
    \offinterlineskip
    \ialign{##\cr
      #2{#1}\cr
      \noalign{\kern1pt}
      $\m@th#1#3$\cr
    }
  }%
}
\newcommand{\dashfill}[1]{%
  \kern-.5pt
  \xleaders\hbox{\kern.5pt\vrule height.4pt width \dash@width{#1}\kern.5pt}\hfill
  \kern-.5pt
}
\newcommand{\dash@width}[1]{%
  \ifx#1\displaystyle
    2pt
  \else
    \ifx#1\textstyle
      1.5pt
    \else
      \ifx#1\scriptstyle
        1.25pt
      \else
        \ifx#1\scriptscriptstyle
          1pt
        \fi
      \fi
    \fi
  \fi
}
\newcommand{\solidfill}[1]{\leaders\hrule\hfill}
\makeatother

\title{Staircase symmetries  in Hirzebruch surfaces}
\author{Nicki Magill}
\address{Mathematics Department, Cornell University}
\email{nm627@cornell.edu}
\thanks{NSF Graduate Research Grant DGE-1650441
}
\author{Dusa McDuff}
\address{Mathematics Department, Barnard College}
\email{dusa@math.columbia.edu}
\keywords{symplectic embeddings in four dimensions, symplectic capacity function, Diophantine equation}
\subjclass{53D05,11D99}
\date{December 12, 2021}

\begin{document}

\begin{abstract} This paper continues the investigation of staircases in the family of  Hirzebruch surfaces formed by 
blowing up  the projective plane with weight $b$, that was started in Bertozzi, Holm et al. in arXiv:2010.08567.  
We explain the symmetries underlying the structure of the set of $b$ that admit staircases, and show how the properties of these symmetries arise from a governing Diophantine equation.    We also greatly simplify the techniques  needed to show that a family of steps does form a staircase by using arithmetic properties of the accumulation function. There should be analogous results  about both staircases and  mutations
for the other rational toric domains considered, for example, by Cristofaro-Gardiner et al. in arXiv:2004.07829 and 
by Casals--Vianna in arXiv:2004.13232.
\end{abstract}

\maketitle
\tableofcontents
\section{Introduction}

\subsection{Overview}\label{ss:over}

This paper continues the investigation of the ellipsoidal embedding capacity function for the family of Hirzebruch 
surfaces $H_b$ that was begun in \cite{ICERM}.  Here $(H_b,\om)$ is the one-point blowup $\C P^2(1)\# \oCP^2(b)$ of the complex projective plane with line class $L$ of size $1$ and exceptional divisor $E_0$ of size $b$.
The capacity function $c_{X}: [1,\infty) \to \R$ for a general four-dimensional target manifold ($X,\la)$  is defined by
$$
c_{X}(z):=\inf\left\{\lambda\ \Big|\ E(1,z)\sembeds \lambda X\right\},$$
where $z\ge 1$ 
is a real variable,  $\lambda X: = (X,\lambda\omega)$,  an ellipsoid $E(c,d)\subset \mathbb{C}^2$ is the set
$$
E(c,d)=\left\{(\zeta_1,\zeta_2)\in\mathbb{C}^2 \ \big|\  \pi \left(  \frac{|\zeta_1|^2}{c}+\frac{|\zeta_2|^2}{d} \right)<1 \right\},
$$
and we write $E\se \la X$ if there is a symplectic embedding of $E$ into $\la X$.  

It is straightforward to see that
$c_{H_b}(z)$ is bounded below by the volume constraint function $V_b(z) = \sqrt{\frac{z}{1-b^2}}$, where $1-b^2$ is the appropriately normalized volume of $H_b$. Further,
 the function $z\mapsto c_{H_b}(z)$ is piecewise linear when not identically equal to the volume constraint curve. When its graph has infinitely many nonsmooth points (or steps) lying above the volume curve, we say that $c_{H_b}$ has an {\bf infinite staircase}.   
\MS

It was proven in \cite{AADT} that when $b= 1/3$, i.e. in the case when $H_b$ is monotone, the function $c_{H_b}$ admits a staircase with outer steps at points $z=x_k/x_{k-1}$ that satisfy the recursion $x_{k+1} = 6x_k - x_{k-1}$ and accumulate at the fixed point $3+2\sqrt2$ of this recursion.\footnote
{
This staircase actually consists of three intertwining strands that each satisfy this recursion; for details see Example~\ref{ex:13}. Further,  conjecturally $b=1/3$ is
 the only rational value of $b$ at which $H_b$ admits a staircase.}
  Another key result from this paper is \cite[Thm.1.8]{AADT}, 
  stating that if $c_{H_b}$ has an infinite staircase, then its accumulation point is at the point 
$z=\acc(b)$, the unique solution $>1$ of the following quadratic equation involving $b$:
\begin{align}\label{eq:accb0}
z^2-\left( \frac{\left(  3-b\right)^2}{1-b^2}-2 \right)z+1=0
\end{align}
Further,  if $H_b$ has an infinite staircase, then at $z=\acc(b)$ the ellipsoid embedding function must 
equal the volume:
\begin{align}\label{eq:cHb}
c_{H_b}(\acc(b))=\sqrt{\frac{\acc(b)}{1-b^2}} = : V_b(\acc(b)).
\end{align}
We say that $H_b$ is {\bf unobstructed} if $c_{H_b}(\acc(b)) = V_b(\acc(b)$.  Thus if $H_b$ has a staircase, it is unobstructed. However the converse does not hold:  \cite[Thm.6]{ICERM}  shows that, although  $H_{1/5}$ is unobstructed, it has no staircase.
As shown in Fig.~\ref{fig:101}, the function $b\mapsto \acc(b)$ decreases for $b\in [0,1/3)$, with minimum value $a_{\min}: = \acc(1/3) = 3+2\sqrt2$, and then increases.

\begin{center}
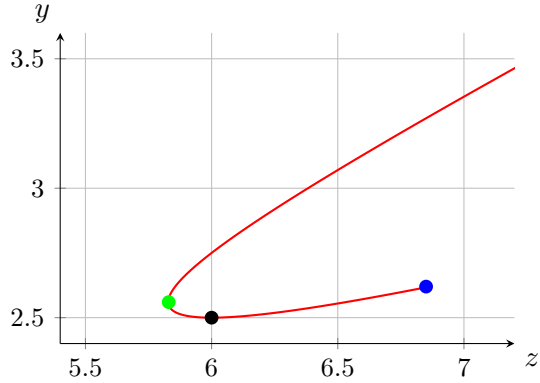
\begin{figure}[h]
\begin{tikzpicture}
\begin{axis}[
	axis lines = middle,
	xtick = {5.5,6,6.5,7},
	ytick = {2.5,3,3.5},
	tick label style = {font=\small},
	xlabel = $z$,
	ylabel = $y$,
	xlabel style = {below right},
	ylabel style = {above left},
	xmin=5.4,
	xmax=7.2,
	ymin=2.4,
	ymax=3.6,
	grid=major,
	width=3in,
	height=2.25in]
\addplot [red, thick,
	domain = 0:0.7,
	samples = 120
]({ (((3-x)*(3-x)/(2*(1-x^2))-1)+sqrt( ((3-x)*(3-x)/(2*(1-x^2))-1)* ((3-x)*(3-x)/(2*(1-x^2))-1) -1 ))},{sqrt(( (((3-x)*(3-x)/(2*(1-x^2))-1)+sqrt( ((3-x)*(3-x)/(2*(1-x^2))-1)* ((3-x)*(3-x)/(2*(1-x^2))-1) -1 )))/(1-x^2))});
\addplot [black, only marks, very thick, mark=*] coordinates{(6,2.5)};
\addplot [blue, only marks, very thick, mark=*] coordinates{(6.85,2.62)};
\addplot [green, only marks, very thick, mark=*] coordinates{(5.83,2.56)};
\end{axis}
\end{tikzpicture}
\caption{ This shows the location of the accumulation point $(z,y)=(\acc(b), V_b(\acc(b)))$ for $0\le b< 1$.
The blue point with $b=0$ is at $(\tau^4, \tau^2)$ and is the accumulation point for the Fibonacci stairs.  The green point  
with  $z = 3+2\sqrt2,  b=1/3$  
is the accumulation point for the 
stairs in  $H_{1/3}$,   and is the minimum of the function $b\mapsto\acc(b)$.  
The black point  with $z=6, b=1/5$  is the place where $V_b(\acc(b))$ takes its minimum. 
} \label{fig:101} 
\end{figure}
\end{center}

It turns out that the nature of the accumulation function 
plays a crucial role in our discussion.  For example, as shown in
Lemma~\ref{lem:rat}, the properties of the  pairs of rational numbers $(b,z)$ with $z=\acc(b)$  are a key to the symmetries of the problem.  Other important consequences are collected in \S\ref{ss:qperf}.
Note also that the case $b=0$ (that is, the case of $\C P^2$) was fully analyzed in \cite{ball}.
Here there is a staircase, which is called the Fibonacci stairs because its numerics are governed by the Fibonacci numbers; see   Fig.~\ref{fig:101}.  The new staircases that we have found for general $H_b$ are all analogs of that one.
However, as we explain in Example~\ref{ex:13}, it is perhaps better to consider our staircases to be offshoots of the $1/3$-staircase.

Obstructions to embedding ellipsoids come from certain exceptional divisors in blowups of the target manifold.  When $X = H_b$,  these divisors  live in $\C P^2 \# (N+1){\oCP}^2$ and their homology classes $\bE$  have the form
\begin{align*}
& dL - mE_0 - \sum m_i E_i = : (d,m,\bbm), \quad\mbox{ where}\\
& \  \bbm: = (m_1,\dots,m_N), \ m_1\ge m_2\ge\dots \ge m_N.
\end{align*}
In the most relevant such classes, the tuple of coefficients  $\bbm$  consists of the (integral) weight expansion 
 of a center point $p/q$ (see \eqref{eq:WT}); correspondingly we say that $\bE$ is {\bf perfect} and write $\bE: = (d,m,p,q)$.
For $z$ near $p/q$, the  embedding obstruction is given by 
\begin{align}\label{eq:obstr0}
\mu_{\bE,b} (z) = \begin{cases} \frac{qz}{d-mb}&\mbox { if } z\le p/q\\
\frac p {d-mb} &\mbox { if } z\ge p/q; \end{cases}
\end{align}
in particular, it has an outer corner (or step) at $z=p/q$.  Since, as explained in \cite[\S2.1]{ICERM},
$c_{H_b}(z)$ is the maximum over all exceptional classes  $\bE$ of the obstruction functions $ \mu_{\bE,b} (z)$,  
 given  $\bE = (d,m,p,q)$ as above, we must have  $c_{H_b}(z)\ge \mu_{\bE,b} (z)$ for  $z\approx p/q$.
We say that the  function $\mu_{\bE,b}$  is
\begin{itemize}\item[-]  {\bf obstructive} at $z$  if $\mu_{\bE,b} (z)> V_b(z)$, and
\item[-] 
{\bf live} at $z$ if $c_{H_b}(z)= \mu_{\bE,b} (z)> V_b(z)$.
\end{itemize}
Further, we call $\bE$ a {\bf center-blocking class} if, for one  of the two elements $b \in \acc^{-1}(p/q)$,
the function $z\mapsto \mu_{\bE,b}(z)$ is obstructive at the center $z=p/q$, since in this case
it follows from \eqref{eq:cHb} that the corresponding surface $H_b$ has no staircase.  
As explained in  \cite[Lem.38]{ICERM}, it follows by  continuity that for every center-blocking class $\bE$  there is an open interval $J_\bE\subset [0,1)$ that contains the appropriate\footnote
{
If $m/d> 1/3$ then this will be the larger element in $\acc^{-1}(p/q)$, while if $m/d< 1/3$ it will be the smaller one; see \cite[Def.37]{ICERM}.  By Lemma~\ref{lem:13no}, there is no quasi-perfect class with $m/d = 1/3$.  If there is no possibility of confusion, we often simply call these classes blocking classes.}  point $\acc^{-1}(p/q)$  and is a component of the set of
 $b$-values that are blocked by $\bE$.
\MS

The paper \cite{ICERM} found three families of center-blocking classes $\bB^U_n, \bB^L_n,\bB^E_n$ for $n\ge 0$, together with six associated sequences of staircases.   
Each  of these blocking classes $\bB$  has 
 an associated maximal blocked  interval $J_\bB = (\be_{\bB,\ell}, \be_{\bB,u})\subset (0,1)$ consisting of points $b$ that cannot admit a staircase because $\mu_{\bB, b}(\acc(b)) > V_b(\acc(b))$. However, it turns out
 that there are staircases at both  endpoints of these intervals, which gives  three staircase families $\Ss^U, \Ss^L, \Ss^E$ with staircases indexed by $n\ge 0$ and  $\ell$ or $u$, where the steps of staircases labelled $\ell$ (for \lq lower') ascend, while those labelled $u$ (for \lq upper') descend.  The Fibonacci stairs appear as $\Ss^L_{\ell,0}$.
\MS

 It was noted in \cite[Cor.60]{ICERM} that  the centers $p/q$ of the blocking and step classes  for the family $\Ss^U$ are related to those of $\Ss^L,\Ss^E$ by a fractional linear transformation, that we denote either $p/q\mapsto (ap+bq)/(cp+dq)$ or $(p,q)\mapsto (ap+bq, cp+dq)$.  In particular, the two families $\Ss^U, \Ss^E$ are related by the {\bf shift}
\begin{align}\label{eq:S}
S: (p,q)\mapsto (6p-q, p)
\end{align}
that implements the recursion underlying the staircase at $1/3$;
while  the 
two families $\Ss^U, \Ss^L$ are related by the {\bf reflection}
\begin{align}\label{eq:R}
R: (p,q)\mapsto (6p-35q, p-6q), 
\end{align}
that fixes the point $7$ and takes $\infty$ to $6$.

\MS

Our main result verifies a conjecture in \cite{ICERM}, and can be informally stated as follows. (For more detail, see Theorems~\ref{thm:Gg} and \ref{thm:Gglive}.)

\begin{thm}   For each $i\ge 1$ there are staircase families $(S^i)^{\sharp}(\Ss^U)$ and $(S^i)^{\sharp}(\Ss^L)=(S^iR)^{\sharp}(\Ss^U)$, where the $i$-fold shift $S^i$ and reflection $R$ act on the centers of the blocking classes and staircase steps as above.
\end{thm}

Moreover, we will see in Proposition~\ref{prop:seed} below that each staircase family is generated  by its blocking classes together with two \lq seed' classes, a fact that makes it much easier to establish the
effect of the symmetries on the staircase families. 
as we explain below, the coefficients $d,m$ of all the relevant classes are given in terms of $p,q$ by a general  formula~\eqref{eq:formdm0}.

\MS

\begin{rmk}\label{rmk:Ush}\rm
  In the setting considered by Usher~\cite{usher}, the target manifold $P(1,b)$ is the polydisc $B^2(1) \times B^2(b)$, or equivalently the product of two spheres of areas $1,b$. He finds a doubly indexed family of staircases $\Ss_{n,k} =
  \bigl(\bE_{i,n,k}\bigr)_{i\ge 0}$, where $i$  indexes the staircase steps, $n\ge 0$ indexes the intrinsic recursion 
  $x_{i+1,n,k} = \nu_n x_{i,n,k} - x_{i-1,n,k}$ satisfied by the parameters  of the perfect classes $\bE_{i,n,k}, i\ge 0,$ in $\Ss_{n,k}$, and $k$ indexes a symmetry generated by so-called  \lq Brahmagupta moves' that generate  infinitely many  families of staircases  from a basic family $(\Ss_{n,0})_{n\ge 0}$ .
  More precisely, in \cite[\S2.2.1]{usher}, Usher finds a
  way to encode the parameters  of the relevant perfect classes $\bE$
   by means of a triple $(x,\de,\eps)$ of integers that satisfy the Diophantine equation
  $x^2 - 2\de^2 = 2-\eps^2$.  Here the value of $\eps$ is related to the recursion variable $n$, and, if this is fixed, he shows that a (very!) classically known  maneuver that goes from one solution of $x^2 - 2\de^2 = N$ to another can be implemented in such a way that it preserves the set of perfect classes.   Usher expressed this manoever in arithmetic terms (multiplication by a unit  in a number field).  However, as we explain in Remark~\ref{rmk:Ush2} below, when expressed in terms of the coordinates $p,q$,  Usher's basic symmetry 
  is the same as ours, namely the transformation $(p,q)\mapsto (6p-q,p)$. 
  
Usher's setting is simpler than ours in that the  function $b\mapsto\acc(b)$ that specifies the accumulation point of any staircase for $P(1,b) $ is injective rather than two-to-one.\footnote
{
This statement is oversimplified in that one could well argue that the analog of our family $H_b, b\in [0,1)$ is the family
$P(1,b), b> 0$ with involution $b\mapsto 1/b$.  However, $P(1,b)$ is symplectomorphic to a rescaling of $P(1,1/b)$, so that $c_{P(1,b)}$ is a rescaling of $c_{P(1,1/b)}$ and $\acc(b) = \acc(1/b)$. In our case, if $\acc^{-1}(z) =  
\{b^+, b^-\}$ the two functions $c_{H_{b^+}}, c_{H_{b^-}}$ can be very different, one with a staircase, and one without: see Fig.~\ref{fig:symm}.}
  Also   
the symmetry between the two classes $[\pt \times S^2]$ and $[S^2 \times \pt]$ allows the arithmetic properties of a general quasi-perfect class to be encoded by means of variables that satisfy the equation $x^2 - 2\de^2 = N$, while the corresponding equation in our setting is  $x^2 - 8y^2 = k^2$ (see Lemma~\ref{lem:rat}).   Nevertheless, the two situations are very similar.
\hfill$\er$
\end{rmk}

The work presented here leads to many interesting questions.  Here are some of them.

\MS

$\bullet$  The picture developed here seems to make up the first level of an iterative \lq\lq fractal'' kind of structure
for the Hirzebruch surfaces  $H_b$.  
One might consider the family of blocking classes  $\bB^U_n, n\ge 0$, extended by two seeds as in Proposition~\ref{prop:seed}, to be the backbone of the first level of this structure. This level also includes  the associated staircase classes.  We prove in Proposition~\ref{prop:block0} that  all the staircase classes are also center-blocking classes.  Further, numerical evidence suggests that there are  staircases whose steps have centers with $4$-periodic continued fractions, indeed it seems with any even period.
What seems to be the case is  that each pair of adjacent ascending/descending staircases at level one shares a first step, and that this first step is a blocking class with associated $4$-periodic staircases.  Thus the backbone of the second level should consist of these shared steps, with associated $4$-periodic staircases generated by appropriate seeds at level one.  For more details, see \cite{MMW}.

$\bullet$  It also would be very interesting to analyse Usher's results using the current framework, to see if there are analogs of blocking classes, seeds and staircase families.
 One might be able to build a bridge between the two cases by thinking of a polydisc as a degenerate two-point blowup of $\C P^2$, and  then looking at
the ellipsoidal capacity  function for the family of two-fold blowups of $\C P^2$ that join the two cases.  This will also be the subject of future work.

$\bullet$  
The recursive patterns behind the staircases for rational target manifolds $X$ 
are related to almost toric structures and the transformations called mutations that appear for example in \cite{AADT}
and  Casals--Vianna~\cite{CV}.   It would be very interesting to know how the symmetries discussed here appear in those contexts. 
\MS

\NI {\bf Acknowledgements.}\   We thank  Chao Li for help with Diophantine equations,  and Tara Holm, Peter Sarnak  and Morgan Weiler for  useful discussions and comments.  The first author also thanks Tara Holm in her capacity as research advisor    for introducing her to the subject and for 
 support and encouragement along the way.

\subsection{Main results}\label{ss:main}
  We now describe our main results in more detail.

\MS

In what follows it is important to distinguish purely numerical properties (such as those in \eqref{eq:dioph0})
from geometric properties that are needed to guarantee that a class $\bE$ gives a live obstruction $\mu_{\bE,b}(z)$ at  relevant values of $b,z$.  
As above we represent a class $\bE = dL - mE_0 - \sum_i m_i E_i$, where $(m_1,m_2,\dots)$ is the weight expansion of $p/q$, by the tuple $ (d,m,p,q) $, and say that $\bE$ is {\bf quasi-perfect} if and only if  the Diophantine conditions
\begin{align}\label{eq:dioph0}
3d=p+q+m,\quad d^2 - m^2 = pq-1
\end{align}
hold. (As explained in \cite[\S2.1]{ICERM}, these are equivalent to the conditions $c_1(\bE) = 1,\  \bE\cdot\bE = -1$, where $c_1$ is the first Chern class of $H_b$.)  
We say that a quasi-perfect class $\bE $ is {\bf perfect}  if it is represented by an exceptional curve, which holds if and only if it reduces correctly by Cremona moves (see Lemma~\ref{lem:Cr1}). 
Although a quasi-perfect class is obstructive at its center $z=p/q$ for  $b = m/d$ by \cite[Lem.15]{ICERM}, the fact that  this obstruction is live at this $(b,z)$
(and hence coincides with the capacity function in some range) follows from \lq positivity of intersections', namely the fact that  the intersection number of two
different exceptional divisors is always non-negative (see \cite[Prop.21]{ICERM}).  Hence in order to show that a given surface $H_b$ actually has a staircase, we need to prove that the relevant staircase classes are exceptional classes.
However, a large part of the following discussion is purely numerical.  Notice also that 
because quasi-perfect classes can be obstructive, it makes sense to consider {\bf quasi-perfect  blocking classes}, i.e. tuples $(d,m,p,q)$  such that $\mu_{\bB,b}(p/q)> V_b(p/q)$  for the appropriate $b\in \acc^{-1}(p/q)$.

\MS

The coefficients $(d_\ka,m_\ka,p_\ka,q_\ka)$ of the
step classes of the staircases we consider always satisfy a recursion of the form 
\begin{align}\label{eq:recur00}
x_{\ka+1} = \nu x_\ka - x_{\ka-1},\qquad \ka\ge \ka_0,
\end{align}
 for suitable recursion parameter $\nu$ and initial value $\ka_0$\footnote{Not all infinite staircases satisfy this recursion. The $b=1/3$ staircase is given by a nonhomogenous recursion. Further, it is not known that all staircases must be given by some recursion.}.  Hence, given $\nu$, each sequence of parameters $(x_\ka)_{\ka\ge i}$ (for $x = p,q,d,m$) is determined by two initial values $x_{\ka_0},x_{\ka_0+1}$ that are called {\bf seeds}.   It turns out that 
the relevant classes $(d,m,p,q)$ can be extended 
to tuples
\begin{align}
(d,m,p,q, t,\eps),\quad t> 0,\  \eps \in \{\pm1\},
\end{align}
where the integer  $t$  is a function of $p,q$. Further the tuple
$(p,q,t,\eps)$ determines
the degree variables $(d,m)$ by the following  formula  
\begin{align}\label{eq:formdm0}
 d: = \tfrac 18( 3(p + q) + \eps t), \quad  m: =  \tfrac 18((p + q) + 3 \eps t).
\end{align}
Hence $|d-3m| = t$, and $\eps =1$ if and only if $m/d> 1/3$. This point of view is explained in \S\ref{ss:qperf}.
\MS

It is straightforward to check  that both $S$ and $R$  preserve $t$ and hence $\eps(d-3m)$, while they both change the sign of $\eps$. Further, for classes that give obstructions when $b>1/3$ we have  $\eps=1$, while $\eps = -1$ if the classes are relevant for $b< 1/3$; see Example~\ref{ex:v2}.\MS

Our first main result is that all the numerical data of a staircase family such as $\Ss^U$ is determined by a family of quasi-perfect classes $(\bB^n)_{n\ge 0}$ together with two seeds $\bE_{\ell,seed}, \bE_{u,seed}$.  To explain this, we introduce the following language.

\begin{itemlist}\item
A {\bf pre-staircase} $\Ss$ is a sequence of tuples $\bE_\ka: = \big((d_\ka,m_\ka,p_\ka,q_\ka, t_\ka,\eps)\bigr)_{\ka\ge 0}$  that is defined recursively with recursion parameter $\nu$, and that satisfy \eqref{eq:formdm0}.    
 Given such a sequence, the limits $a_\infty: = \lim p_\ka/q_\ka$ and $b_\infty: = \lim m_\ka/d_\ka$ always exist by 
 Corollary~\ref{cor:monot}.  
 We say that 
 $\Ss$ is {\bf perfect} if all the classes $\bE_\ka$ are perfect, and that $\Ss$ is {\bf live} if  the obstructions
 $\mu_{\bE_\ka, b_\infty}$ are live at $p_\ka/q_\ka$ for all sufficiently large $\ka$ and with $b$ equal to the limiting value $b_\infty$.   
We will refer to a {\bf staircase} as a live pre-staircase.\footnote
 {
 We do not insist that the classes in a staircase are perfect; however in all known cases they are perfect. Indeed the only  way that we know of to prove  that a staircase is live  is first to show that it is perfect and then to show there are no \lq\lq overshadowing classes''. See the beginning of \S\ref{sec:live} for more details.} Thus if $\Ss$ is live, $H_{b_\infty}$ has a staircase. This is a slight abuse of notation as not all staircases follow the recursive structure of a pre-staircase, but all staircases considered in this paper are indeed pre-staircases. 
\item  A pre-staircase $\Ss$ is said to be {\bf associated to a quasi-perfect class} $$
\bB = (d_\bB, m_\bB, p_\bB, q_\bB)
$$ if the following linear relation is satisfied by its step coefficients
\begin{align}\label{eq:linrel}\begin{array}{lll}
&(3m_\bB-d_\bB)\, d_\ka = (m_\bB - q_\bB)\, p_\ka +m_\bB\, q_\ka & \quad \mbox{ if }\; \Ss \; \mbox{ ascends}\\ \notag
&(3m_\bB-d_\bB)\, d_\ka =   m_\bB\, p_\ka - (p_\bB-m_\bB)\, q_\ka & \quad \mbox{ if }\; \Ss \; \mbox{ descends.}
\end{array}
\end{align}
If in addition the pre-staircase is perfect, then $H_{b_\infty}$ is unobstructed, i.e. $c_{H_b}(\acc(b)) = V_b(\acc(b))$, and
 it is shown in \cite[Thm.52]{ICERM} that the  limits $(b_\infty, a_\infty)$ are the parameters $(b,z)$ of  the appropriate endpoint of the blocked $b$-interval $J_\bB$.\footnote
{
This follows very easily from the calculation in \eqref{eq:volacc} below.}
  Moreover, if  there is both an ascending and a descending perfect pre-staircase associated to $\bB$ then 
$\bB$ is a perfect blocking class.   This means in particular that $\bB$  is obstructive for the $b$-value corresponding to its center.

\item A {\bf pre-staircase family} $\Ff$ consists of a family of quasi-perfect classes $\bigl(\bB^\Ff_n\bigr)_{n\ge 0}$ (called {\bf pre-blocking classes}) together with ascending pre-staircases $\Ss^\Ff_{\ell,n}, n\ne 1$, and 
descending pre-staircases $\Ss^\Ff_{u,n}, n\ne 0,$  where $\Ss^\Ff_{\bullet,n}$ is associated with $\bB^\Ff_n$ for $\bullet = \ell,u$.  The family $\Ff$  is said to be {\bf perfect} if all the classes in $\Ff$ are perfect, and {\bf live} if all the pre-staircases $\Ss^\Ff_{\bullet,n}, \bullet = \ell,u$, are live.  
\item A  pre-staircase family is called a {\bf staircase family}  if it is 
live.  Then the pre-blocking classes are perfect by \cite[Thm.52]{ICERM}, and in all cases encountered here  the step classes are also perfect.\end{itemlist} 
\MS

Finally, we make the following definition.  Notice that pre-staircase families that are obtained from $\Ss^U$ by applying the shift $S^i$ have pre-blocking classes whose centers ascend, while those that 
are obtained from $\Ss^L$ by applying $S^i$ have pre-blocking classes whose centers descend.  It turns out that in all cases  the adjacent pre-blocking class can be considered as part of the appropriate  staircase; for example in $\Ss^U$ the blocking class $\bB^U_{n-1}$ can be considered as a step in the ascending staircase $\Ss^U_{\ell,n}$ while
$\bB^U_{n+1}$ can be considered as a step in the descending staircase $\Ss^U_{\ell,n}$.  Clearly the numbering  of this adjacent blocking class depends on whether the centers of these classes ascend or descend as $n$ increases.

\begin{definition}\label{def:prestf}
A pre-staircase family $\Ff$ is said to be {\bf generated by the quasi-perfect classes $\bB_n,n\ge0,$ and seeds} $\bE_{\ell,seed},
\bE_{u,seed}$ if the $\bB_n = (d_n,m_n,p_n,q_n,t_n,\eps)$ are its pre-blocking classes, and, for for all $n$ and $\bullet \in \{\ell,u\}$ 
the steps in the pre-staircase  $\Ss^\Ff_{\bullet,n}$ have recursion parameter $\nu=t_n$ and seeds
\begin{itemize}\item[-]
$\bE_{\ell,seed}, \bB_{n-1}$ for $\bullet = \ell$ 
and $\bE_{u,seed}, \bB_{n+1}$ for $\bullet = u$, if  the $\bB_n$ ascend;
\item[-] $\bE_{\ell,seed}, \bB_{n+1}$ for $\bullet = \ell$ 
and $\bE_{u,seed}, \bB_{n-1}$ for $\bullet = u$, if  the $\bB_n$ descend.
\end{itemize}
\end{definition}

In Propositions~\ref{prop:Uu} and~\ref{prop:Ll} below, we establish the following compact description of the numerical information that determines each of our staircase families.  This information will make it much easier 
to understand the effect of the symmetries.

\begin{prop} \label{prop:seed} The staircase families $\Ss^U,\Ss^L$ are generated by their blocking classes together with two seeds.
\end{prop}

\begin{rmk}\rm
As we explain in Example~\ref{ex:13}, all the classes (both blocking classes  and seeds) that generate  $\Ss^U$ are directly related to the classes that generate the staircase at $b=1/3$.  Since the  staircase classes in $H_{1/3}$   satisfy the recursion~\eqref{eq:recur00} with $\nu=6$ given by  $S$, one can see that the whole structure of the staircases so far discovered in the family of manifolds $H_b$ stems from that of the staircase at $b=1/3$. \hfill$\er$
\end{rmk}

We now explain the action of the family\footnote
{
$\Gg$ is not a semigroup because  $S^i  RS^j = S^{i-j}R \in \Gg$ only if $i\ge j$; see Lemma~\ref{lem:Phi0}.}
of transformations 
\begin{align}
\Gg: = \{S^i  R^\de\ \big| \ \de\in \{0,1\}, i\ge 0\},   
\end{align}
where $S$ is the shift $p/q\mapsto  (6p-q)/ p$ and $R$ is the reflection $z\mapsto (6z-35)/(z-6)$.
The sequence of numbers 
$$
v_1 = \infty\,  = 1/0,\ \ v_2 = S(v_1)=6/1,\  \dots,\  v_i = S^{i-1}v_1,\dots
$$
has limit $
a_{\min}: = 3 + 2\sqrt2$, which is the accumulation point of the $b=1/3$ staircase.   With $w_1: = 7$ and $w_i = S^{i-1}(w_1)$,
we have the following interleaving family of numbers:
\begin{align}\label{eq:vw}
v_1 = \infty >w_1 = 7>v_2 = 6 > w_2 =   41/7 > \cdots > v_i > w_i > v_{i+1} >\cdots\ \to a_{\min}.
\end{align}
The symmetry $S$ acts as a shift:  $S(v_i) = v_{i+1}, S(w_i) = w_{i+1}$, while $R$ and its composition with powers of $S$ are reflections. In particular $R_{v_i}: = S^{i-2}RS^{-i+1}$ is a reflection that fixes $v_i$ and interchanges $w_{i+1}$ with $w_i$.  As we show in \eqref{eq:acc-1}, for each $i$
the two $b$-values that correspond  to the point $z=v_i$ are rational and have simple linear expressions in terms of the numerator and denominator $p_i,q_i$ of $v_i = p_i/q_i$.  Hence one might expect them to be relevant to the problem.
Note that (except for one or two initial terms) the  staircase  steps in $\Ss^U$ all  lie in the interval $(7,\infty)= (w_1,v_1)$, where $v_i,w_i$ are as in \eqref{eq:vw}.  Further,
$S\bigl((w_1,v_1)\bigr) = (w_2,v_2)$ while $R\bigl((w_1,v_1)\bigr) = (v_2,w_1)$.  We showed in \cite{ICERM} that 
$R$ takes the blocking classes and staircase steps of $\Ss^U$ to $\Ss^L$ and that $S$ takes $\Ss^U$ to $\Ss^E$.

\begin{figure}\label{fig:symm}
\vspace{-1 in}
\centerline{\includegraphics[width=7in]{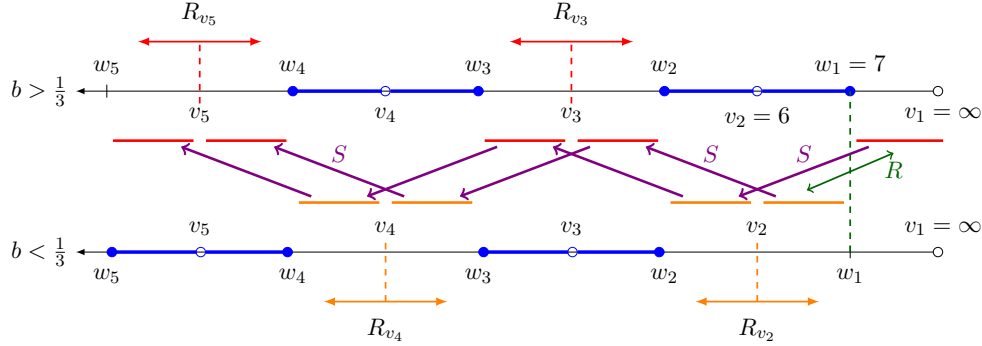}}
\vspace{-6 in}
\caption{The centers of the staircase family $\Ss^U$ (resp. $\Ss^L$) live in the rightmost red (resp. orange) interval. 
Connecting these, in green, we have the reflection $R$ about $w_1=7$ which sends the $\Ss^U$ to the $\Ss^L$ family.  The shift $S$, purple, preserves orientation and takes staircase families with $b$ value below (resp. above) $1/3$ to staircases with $b$ value above (resp. below) $1/3.$   By applying $S$ iteratively  to $\Ss^U$ and $\Ss^L$, one obtains staircase families 
in the other red or orange half-intervals. In each of these half-intervals, there is a family of ascending and descending staircases whose accumulation points converge to some  $v_i$.  For each $i \geq 1$, each interval $(w_i,w_{i-1})$ admits a reflection $R_{v_i}$ that fixes its center point $v_i$ and interchanges the staircases in that interval.  Finally, the blue intervals are blocked by a family of principal blocking classes with centers at the $v_i$.}
\end{figure}

The next theorem gives a numerical description of the image of the staircase families $\Ss^U$ by $T\in \Gg$.
This extends  \cite[Cor.60]{ICERM}  where it is shown that $S, R$ take the centers of the classes in $\Ss^U$  to those of $\Ss^E, \Ss^L$ respectively.  This extended action is illustrated in Fig.\ref{fig:symm}.  Note that $S$ preserves the direction (ascending or descending) of the centers of a family of classes, while $R$ reverses it.

\begin{thm}\label{thm:Gg}  For each $T \in \Gg$ there is a corresponding pre-staircase family $T^\sharp(\Ss^U)$, consisting of tuples $(d,m,p,q,t,\eps)$ where 
\begin{itemize}\item[-]  $(p,q)$ are the image by $T$ of the corresponding tuple in $\Ss^U$,
\item[-]  $t$ is unchanged and $\eps$  transforms by the factor $(-1)^{i + \de}$, where $T= S^iR^\de$,
\item[-]  $(d,m)$ are determined by $(p,q,t,\eps)$  according to \eqref{eq:formdm0}. In particular, $\eps(d-3m)$ remains unchanged.
\end{itemize}
This pre-staircase family is generated in the sense of Definition~\ref{def:prestf} by the images under $T$ of the seeds and blocking classes of $\Ss^U$.
Further, 
\begin{itemlist}\item[{\rm (i)}] If $T= S^i$  for some $i>0$, then the centers of the pre-blocking classes of $T^\sharp(\Ss^U)$ increase with $n$, and
the pre-staircase steps lie in the interval $(w_{i+1}, v_{i+1})$ and have $\eps =  (-1)^i$.
\item[{\rm (ii)}] If $T= S^iR$  for some $i\ge 0$, then the centers of the pre-blocking classes decrease with $n$, and
the pre-staircase steps lie in the interval $(v_{i+1}, w_{i})$ and have $\eps =  (-1)^{i+1}$. \end{itemlist}
 In particular,
 $R^\sharp(\Ss^U) = \Ss^L$ while $S^\sharp(\Ss^U) = \Ss^E$.
\end{thm} 

For the proof see Proposition~\ref{prop:Gg}.

\begin{rmk} \label{rmk:Gg} \rm (i)  When $b>1/3$, the accumulation points  of the staircases  lie in the intervals $(w_{2i+1}, w_{2i})$, while, for each $i$,  the interval $[w_{2i}, w_{2i-1}]$ is blocked by the {\bf principal blocking class} $\bB^P_{v_{2i}}: = (S^{2i-2})^\sharp(\bB^U_0)$ with center $v_{2i}$.  Similarly, when $b<1/3$, the accumulation points  of the staircases  lie in the intervals $(w_{2i+1}, w_{2i})$, while, for each $i$,  the interval $[w_{2i+2}, w_{2i+1}]$ is blocked by the  principal blocking class $\bB^P_{v_{2i+1}}: =  (S^{2i-1})^\sharp(\bB^U_0)$ with center $v_{2i+1}$. 
See Figs.\ref{fig:symm} and~\ref{fig:reflect} and  Corollary~\ref{cor:block}.

\MS
\NI (ii)  As already noted, the action of a general symmetry  $T= S^i R^\de$ on the degree variable $(d,m)$ of the blocking classes  is determined by its action on $(p,q)$ together with the equation \eqref{eq:formdm0}.  
We will see in Lemma~\ref{lem:degree}  that the action of $T$ on the $(d,m)$-coordinates of the family of blocking classes $\bB^U_n$ is 
given by a $2\times 2$ integral matrix  $T_{\bB}^*$.    Because, as noted after \eqref{eq:formdm0},   the linear function $d-3m$
is invariant  by $T_{\bB}^*$ modulo sign,
this matrix has eigenvector $(3,1)^{\vee}$, with  eigenvalue $(-1)^{i+\de} \det(T_{\bB}^*)$.  In general, the other eigenvector (with eigenvalue $(-1)^{i+\de}$) has no obvious interpretation.

However, there are two cases in which it does.  Indeed if $T$ is the reflection $R_{v_i}$  that fixes $v_i$, interchanging $w_i$ with $w_{i+1}$, then  it has two lifts to an action on degree depending on whether we take $b>1/3$ or $b<1/3$.   As we show in \S\ref{ss:symmact}, it turns out that only one of these actions on degree has order two, though they both have clear geometric interpretations.   For example, if $i=2$ so that $v_2=6$, then $R_{v_2} = SR$ interchanges the centers of the blocking classes and seeds of the staircase families $\Ss^L$ and $\Ss^E$.
The lift  $(R_{v_2})_{\bB}^{*}$ that acts on the degree components of the blocking classes has order two and eigenvector $(5,1)^{\vee}$, since $1/5 = \acc_L^{-1}(6)$ is the limit of the ratio of the degree components. On the other hand,
 the matrix $(R_{v_2})^*_{\bP} = (SR)_\bB^*$ takes the blocking classes of the staircase family $\Ss^U$ to those of $SR(\Ss^U)$, fixing the shared principal blocking class $B^U_0$.  Thus it  has eigenvector $(3,2)^{\vee}$, but, as we calculate in Example~\ref{ex:v2}, does not have order two.
  For more details about the action of a general element $T\in \Gg$ on the degree components of the blocking classes, see Proposition~\ref{prop:degree}. \hfill$\er$
\end{rmk}

Our second main result is that all the new staircases are live.

\begin{thm}\label{thm:Gglive}   For each $T\in \Gg$  the pre-staircase family $T^\sharp(\Ss^U)$ is live, and hence is a staircase family.
\end{thm} 

The proof that all the classes involved are perfect is given in Proposition~\ref{prop:perf}.  It is greatly eased by the discovery that their $(d,m)$ components satisfy \eqref{eq:formdm0}.   The proof that the classes are live is given in \S\ref{sec:live}.  Although it is based on the methods developed in \cite{ICERM} that we explain at the beginning of 
\S\ref{sec:live},  we have significantly simplified
the proof by  using new arithmetic arguments. 
Indeed, in Proposition~\ref{prop:liveGen}, we establish a simple, widely applicable criterion for an arbitrary perfect pre-staircase family to be live. 

\MS

The paper \cite{ICERM}  considered the following subsets of the $b$-parameter space $[0,1)$:
\begin{align*}
{\it Stair}: &= \{b \ | \ H_b \mbox{ has a staircase} \} \subset [0,1) ,\;\;\mbox{ and }\\
{\it Block}: & = \bigcup \bigl\{ J_{\bB} \ \big| \ \bB  \mbox{ is a blocking class} \bigr\} \subset [0,1).
\end{align*}
Each blocking class $\bB^U_n$ defines an interval $J_{U,n}\subset (1/3,1)$ of $b$-values that have no staircase because $\mu_{\bB^U_n,b}(\acc(b))> V_b(\acc(b))$.   Further the end points of these intervals lie in {\it Stair}.
Similarly, the blocking classes in the staircase family $\Ss^L= R^\sharp(\Ss^U)$ define blocked intervals $J_{L,u}\subset (0,1/3)$, whose endpoints lie in {\it Stair}.

There is an induced action of the symmetries in $\Gg$ on the $b$-variable $[0,1)$.  This is easiest to describe 
for the shift $S$ since this gives an injection $[a_{\min}, \infty)\to [a_{\min}, \infty)$ that fixes $a_{\min}: = \acc(1/3)=3+2\sqrt{2}$.  
Hence, if we write $$
\acc_L:[1/3,1)\to [a_{\min}, \infty), \qquad \acc_U:[0,1/3]\to [a_{\min}, \infty)
$$
 for the appropriate restriction of the function $b\mapsto \acc(b)$ in \eqref{eq:accb0}, for each $k\ge 1$ we can define $(S^k)^*$ by:
\begin{align}\label{eq:actS}
(S^k)^*(b) = \begin{cases} \acc_L^{-1} \circ S^k\circ \acc_L&\mbox{ if } b\le 1/3,\\
\acc_U^{-1} \circ S^k\circ \acc_U&\mbox{ if } b\ge 1/3.
\end{cases}
\end{align}
Since $(S^k)^*$ is conjugate to $S^k$, the assignment $k\mapsto (S^k)^*$ is a  homomorphism; i.e.
$(S^k)^*\circ (S^m)^* = (S^{k+m})^*, k,m\ge 0$.

The reflection symmetries $S^k R$ are not defined on the whole $z$-interval $[a_{\min}, \infty)$ and so do not extend to a global action on $b$-variable.   Instead, in view of Remark~\ref{rmk:Gg}~(ii), it is most natural to 
restrict the reflection $R_{v_i}: = S^{2i-1}R$ that fixes $v_i$ to the appropriate $b$-interval corresponding to
 $ (v_{i+1},v_{i-1})\supset (w_i,w_{i-1})$. Thus we define
 \begin{align}\label{eq:actR}
(R_{v_i})^*(b) = \begin{cases} \acc_L^{-1} \circ R_{v_i}\big|_{(v_{i+1},v_{i-1})}\circ \acc_L&\mbox{ if  $i$  is even},\\
\acc_U^{-1} \circ R_{v_i}\big|_{(v_{i+1},v_{i-1})}\circ \acc_U&\mbox{ if  $i$  is odd}.
\end{cases}
\end{align}

The following is an immediate consequence of Theorem~\ref{thm:Gglive} because the endpoints of the blocked $b$-intervals $J_{U,n}, J_{L,n},$ are taken by the function $b\mapsto \acc(b)$ to the accumulation points of the corresponding staircases, upon which $\Gg$ acts geometrically.

\begin{cor}\label{cor:Gg} Let $J_{U,n}$ (resp. $J_{L,n}$) be the $b$-interval blocked by $\bB^U_n$ (resp. $\bB^L_n$).  For $J = J_{U,n}, n \ge 0,$ or $J_{L,n}, n\ge 1$ and each $k\ge 0$,  the interval $(S^k)^*(J) $  is a component of {\it Block}, and its endpoints are in {\it Stair}.  Moreover for each $i\ge 2$ the reflection $(R_{v_i})^*$
permutes those intervals $(S^k)^*(J) $ that lie entirely in  $(v_{i+1}, v_{i-1})$.
\end{cor}

\begin{rmk}\rm (i)
We conjecture that the action of the elements $T\in \Gg$ described above preserves the sets  {\it Stair} and {\it Block}.  This would hold if, for example,  every interval in 
{\it Block} was defined by a single blocking class
with two associated staircases, and if these, plus the staircases at $0,1/3$, were the only staircases.  However, the proof of a such a result seems out of reach at present.
\MS

\NI (ii)  Note that the map $(R_{v_i})^*$ has order two since it is defined to be the conjugate of a reflection.  The statements in Remark~\ref{rmk:Gg}~(ii) about the transformations $(R_{v_i})^\sharp$ are rather different, since here we are concerned with the action of $R_{v_i}$ on the degree variables $(d,m)$ of the relevant blocking classes, and not on the conjugate to its action on the $z$-variable. \hfill$\er$
\end{rmk}

\section{The accumulation function and its symmetries}\label{sec:symm}

In this section, we first discuss the arithmetic properties of the symmetries and related topics.
In \S\ref{ss:qperf} we first give an alternative way to understand the coordinates $(d,m,p,q)$ of a quasi-perfect class, and then show that all such classes are center-blocking.
This second result relies on the particular form of the accumulation function $b\mapsto \acc(b)$. 
Finally, we
 describe the
staircase families $\Ss^U$ and $\Ss^L$ and
 the staircase at $b=1/3$ in the language used in \cite{ICERM}.

\subsection{The fundamental recursion}

We have two sets of variables, the $z$ variable on the domain $E(1,z)$ and the $b$ variable on the target.  They are related by the equation
\begin{align}\label{eq:accb}
z^2-\left( \frac{\left(  3-b\right)^2}{1-b^2}-2 \right)z+1=0
\end{align}
Since $b\in [0,1)$, one can check that this equation has two positive solutions that we denote $a, \frac 1a$, where $a: = \acc(b) >1$.  As illustrated in Fig.~\ref{fig:101},  this function is in general two-to-one with a unique minimum $\acc^{-1}(3 + 2\sqrt2) = 1/3$.
We denote by\footnote
{
Here, and elsewhere $L$ (or $\ell$)  denotes \lq\lq lower''  while $U$ (or $u$) denotes \lq\lq upper'' .}
$$
\acc_L^{-1}: \bigl(3 + 2\sqrt 2, \frac{7 + 3\sqrt 5}{2}\bigr] \; \to \; [0, 1/3),\qquad 
 \acc_U^{-1}: (3 + 2\sqrt 2, \infty)\;\to\; (1/3,1)
$$
the corresponding inverses to the function $b\mapsto \acc(b)$.   Thus, with 
$\tau: = a + \frac 1a + 2 =  \frac{\left(  3-b\right)^2}{1-b^2} \ge 8,$ we have
\begin{align*}
\acc_L^{-1}(a)  = \frac{ 3 - \sqrt{\tau^2 -8\tau}}{\tau+1},\qquad 
& \acc_U^{-1}(a)  = \frac{ 3 + \sqrt{\tau^2 -8\tau}}{\tau+1}
\end{align*}

The minimum $\tau=8$ is attained when $b=1/3$. The corresponding equation
 $z^2 - 6z + 1=0$ is therefore very special.  It arises by taking the limit $\lim {x_{\ka+1}}/{x_\ka}$ of the recursion 
$x_{\ka+1} = 6x_\ka - x_{\ka-1}$  that seems to play a central role in this staircase problem.
For example, the steps of the staircase at $b= 1/3$ satisfy this recursion.  Moreover, as the following lemma shows, 
certain properties of the function $b\mapsto \acc(b)$ are invariant with respect to  this recursion.\MS

\begin{lemma}\label{lem:rat} {\rm (i)}  If $\acc(b) = p/q$ then
 \begin{align}\label{eq:bsi}
 b = \frac{3pq \pm (p+q)\sqrt\si}{p^2+q^2 + 3pq},\quad \mbox{ where } \si: = (p-3q)^2 - 8q^2.  
 \end{align}
In particular,
the numbers $b$ and $\acc(b)$ are both rational if and only if   $(p-3q)^2 - 8q^2 = k^2$ for some  integer $k\ge 1$. 
\MS

\NI {\rm (ii)} The quantity $\si(p,q): = (p-3q)^2 - 8q^2 = p^2 + q^2 - 6pq$ is invariant under the transformation
$(p,q)\mapsto (6p-q,p)$.
\MS

\NI {\rm (iii)}  In particular, because $(p,q) = (6,1)$ is a solution of $\si(p,q)=1$, any successive pair in the sequence
$$
(y_1,y_2,y_3,y_4, \dots) = (1,6,35,204,\dots)
$$
 gives another solution.  Further, these are the only solutions  for $\si=1$.

\MS

\NI {\rm (iv)} If $p=y_i, q=y_{i-1}$ for some $i>1$, we have
\begin{align}\label{eq:acc-1}
\acc_U^{-1}(\frac pq) = \frac{p+q+3}{3p + 3q + 1}, \qquad   
\acc_L^{-1}(\frac pq) = \frac{p+q-3}{3p + 3q - 1} 
\end{align}  
so that
\begin{align*}  
\acc_U^{-1}(6)& = 5/11,\;\; \acc_U^{-1}(35/6) = 11/31,\;\; \acc_U^{-1}(204/35) = 121/359, \cdots\;\;  \searrow\; 1/3\\
\acc_L^{-1}(6)& = 1/5,\;\; \acc_L^{-1}(35/6) = 19/61,\;\; \acc_L^{-1}(204/35) = 59/179, \cdots \;\; \nearrow\; 1/3.
\end{align*}
\end{lemma}
\begin{proof} Let $a = \frac pq>1$ so that $a = \acc(b)$ where $b = \frac{3\pm \sqrt{\tau^2 - 8\tau}}{\tau + 1}$, and $\tau = \frac pq + \frac qp + 2$.  
Since  $$
\tau(\tau-8) = \frac{(p^2 + q^2 + 2pq)(p^2 + q^2 - 6pq)}{p^2q^2},
$$
and $ p^2 + q^2 - 6pq = (p-3q)^2 - 8q^2,
$
this formula for $b$ simplifies to that in \eqref{eq:bsi}. The rest of (i) is clear.
Next, note that (ii) holds 
because  $(6p-q)^2 + p^2 -6(6p-q)p = p^2 + q^2 - 6pq$. 
To prove (iii), note that given any solution $(p,q)$  with $p>q$,  one can use the reverse iteration $(p,q)\mapsto (q, 6q-p)$
to reduce to a solution with $p>q>0$ and $6q-p\le 0$.  But the only such solution is $(6,1)$. Finally, to see that the formula in (iv) for $\acc_U^{-1}(p/q)$ is the same as that in  \eqref{eq:bsi} we must check that 
$$
(3pq + p+q) (3p + 3q +1) =  (3pq + p^2 + q^2)(p+q+3) = (9pq + 1)(p+q+3).
$$
One can check that the third order terms on both sides are the same, and that the rest of the identity holds because $p^2 + q^2 = 6pq + 1$.
The proof for  $\acc_L^{-1}(p/q)$ is similar.
Thus $\acc_U^{-1}(y_i/y_{i-1})$ decreases with limit $1/3$, while 
$\acc_L^{-1}(y_i/y_{i-1})$ increases with limit $1/3$.
  \end{proof}
  
  \begin{rmk}\label{rmk:Pell} \rm   Since $\si(p,q) = \si(q,p)$, $(p,q)$ is a solution if and only if $(q,p)$ is, we always assume that $p>q$ so that the entries in the pairs  $(p,q), \si(p,q), \si(\si(p,q)), \dots$ increase;  with the other convention they would decrease. Notice also that the Pell numbers $0,1,2,5,12, 29, 70,\dots$ that 
  form such a basic element in the polydisc case considered in \cite{FM,usher} are closely related to the sequence 
  $0,1,6,35,\dots$ that is fundamental here; indeed 
   the numbers  $2y_i, i\ge 0,$ are precisely the even-placed Pell numbers.
 \hfill$\er$ \end{rmk}

Let $S:  =\left(\begin{array}{cc} 6&-1\\1&0\end{array}\right)$ be the \lq shift' matrix that implements the recursion
\begin{align}\label{eq:recur0}
x_{\ka+1} = 6x_\ka - x_{\ka-1},\quad S\left(\!\!\!\begin{array}{c} x_{\ka}\\x_{\ka-1}\end{array}\!\!\!\right) = \left(\!\!\!\begin{array}{c} x_{\ka+1}\\x_{\ka}\end{array}\!\!\!\right),
\end{align}
where  the matrix $A: = \left(\begin{array}{cc} a&b\\c&d\end{array}\right)$ acts on the $z$ variables by the fractional linear transformation
\begin{align}\label{eq:fractlin} z\mapsto \frac{az+b}{cz+d}=: Az,
\end{align}
 taking (by extension) $\infty$ to $\frac ac$.
Starting with $y_0=0, y_1 = 1$ we get the sequence 
\begin{align*}
&y_0=0,\;\; y_1=1,\;\; y_2=6,\;\; y_3=35,\;\; y_4=204,\;\; y_5=1189,\;\; y_6 = 6930, \dots,\\  \notag
&\frac{35}6 = [5;1,5],\;\; \frac{204}{35} = [5;1,4,1,5],\;\; \frac{1189}{204} = [5;1,4,1,4,1,5], \dots,
\end{align*}
with general term $[5;\{1,4\}^k,1,5]$; see Lemma~\ref{lem:symmCF}.
If we define 
\begin{align}\label{eq:Rdef}
R: = \left(\begin{array}{cc} 6&-35\\1&-6\end{array}\right) = \left(\begin{array}{cc} 
y_2&-y_3\\ y_1&-y_2\end{array}\right),
\end{align}
then  $R^2 = \id, \det R = -1$, and we will see that the {\bf reflection} $R$ and {\bf shift}  $S$ generate symmetries of our problem.

It is convenient to consider the following decreasing sequences of points in the interval $(3 + 2\sqrt 2, \infty)$:
\begin{align}\label{eq:wv0}
&v_1: = \infty,\;\; v_2: = 6,\;\; v_3: = \frac {35}6 , \;\;\;\;  v_j: = \frac{y_j}{y_{j-1}}\\ \notag
&w_1 = 7,\;\; w_2: = \frac{41}7,\;\; \;\;  w_k = \frac {y_{k+1}+ y_{k}}{y_{k}+ y_{k-1}}.
\end{align}
These sequences interweave:
\begin{align*} 
3 + 2\sqrt 2 < \cdots <  w_k < v_k < w_{k-1} <\cdots < w_2 < v_2 < w_1< v_1 = \infty.
\end{align*}

\begin{lemma}\label{lem:Phi0}  Let   $v_j, w_j, S,\ R$  be as above.  Then:
\begin{itemize}  \item[{\rm (i)}] the following matrix identities hold:
\begin{align*}
&SR  = RS^{-1}, \quad S^{-1}R  = RS, \quad  R\circ R = 
{\rm Id}. 
\end{align*}
\item[{\rm (ii)}]  The matrix  $S^k =  \left(\begin{array}{cc} y_{k+1}&- y_{k}\\y_k&-y_{k-1}\end{array}\right)$
has determinant $1$, i.e.
\begin{align}\label{eq:detSk} 
y_{k}^2 = y_{k+1}y_{k-1} +1 = 6y_{k+1}y_k - y_{k+1}^2 + 1\quad \forall k\ge 1.
\end{align}
\item[{\rm (iii)}] With action as in \eqref{eq:fractlin}, $S(v_j) = v_{j+1} $ and $S(w_j) = w_{j+1} $, for $j\ge 1$;
\item[{\rm (iv)}]  The matrices $S,R$ generate the subgroup of ${\rm PSL}(2,\Z)$ that fixes the quadratic form $p^2 - 6pq + q^2$.
\end{itemize}
\end{lemma}
\begin{proof}   The proof of (i)-(iii) is straightforward.  In particular the formula for $S^k$ holds because $S$ implements the recursion, 
and we also have $\det(S^k) = (\det(S))^k = 1$. Further, one can check that  $A\in {\rm SL}(2,\Z)$ preserves the form $p^2 - 6pq + q^2$ if and only if
$$
A = \begin{pmatrix} a&b\\c&d\end{pmatrix}, \quad\mbox{where } \ c = -b, d=6b+a.
$$
Hence, if $\det(A) = 1$ then we have $a^2 + 6ab + b^2=1$, which implies by Lemma~\ref{lem:rat}~(iii) that 
when $a> -b > 0$ we must have 
$(a, b) =  (y_{k+1}, -y_k)$ for some $k$, so that $A = S^k$ for some $k\ge 1$. It follows similarly that 
the only other matrices $A$ that preserve the form and have  $\det(A) = 1$ have the form $S^k$ for some $k\le0.$
Further if $\det(A) = -1$ then because the matrix $RA$ has determinant $1$ and preserves the quadratic form, we must have $A = S^kR$ for some $k\in \Z$.   Thus (iv) holds.
\end{proof}

\begin{cor}\label{cor:ident}\begin{itemize}\item[{\rm (i)}]  For each $i\ge 1$, the restriction of
\begin{align}\label{eq:Rvi}
R_{v_i} : = S^{i-2} R S^{-(i-1)} = S^{2i-3}R =\left(\begin{array}{cc} y_{2i-1}&-y_{2i}\\y_{2i-2}&-y_{2i-1}\end{array}\right)
\end{align}
 to 
the interval $(v_{2i-1}, v_1)$ is a reflection that fixes $v_i$ and interchanges the points $w_{i+k}, w_{i-k-1}$, and
$v_{i+k}, v_{i-k}$, for $0\le k\le i-1$. 
\item[{\rm (ii)}] The restriction of $R $ to the interval $(v_2,v_1) = (6,\infty)$ is a reflection that fixes $w_1=7$.
\end{itemize}
\end{cor}

We end this subsection with a brief discussion of the weight expansion and continued fractions.
The (integral) {\bf weight expansion} of a rational number $p/q\ge 1$ is a recursively defined, nonincreasing sequence of integers
$W(p/q) = (W_1,W_2,\dots)$ defined as follows:
\begin{align}\label{eq:WT}
\bullet &\; \mbox{ $W_1 = q$ and $W_n\ge W_{n+1}$ for all $n$,}\\ \notag
\bullet &\; \mbox{ if $W_i> W_{i+1} = \dots = W_n$ (where we set $W_0: = p$), then}
\end{align}
\vspace{-.3in}
\begin{align*}
W_{n+1} = 
\begin{cases} 
 W_{i+1} & \mbox{ if }\;\; W_{i+1} + \dots +W_{n+1} = (n-i+1)W_{i+1}  < W_i,\\
W_i - (n-i)W_{i+1} & \mbox{ otherwise}
\end{cases}
\end{align*}
Thus $W(p/q)$  starts with $\lfloor p/q\rfloor$ copies of $q$ (where $\lfloor p/q\rfloor$ is the largest integer $\le p/q$),
and ends with some number $\ge 2$ of copies of $1$.  One can check that 
\begin{align*}
\sum_{i\ge 1} W_i = p+q-1,\qquad \sum_{i\ge 1} W_i^2 = pq.
\end{align*}
Using this, it is straightforward to check that the equations \eqref{eq:dioph0} for a quasi-perfect tuple $(d,m,p,q)$ 
imply that the corresponding class $\bE = dL- mE_0 -\sum_i W_i E_i$ satisfies the  conditions
$c_1(\bE)= 1, \bE\cdot\bE = -1$, as claimed earlier.
Moreover, 
the multiplicities $\ell_0,\dots,\ell_k$ of the entries in $W(p/q)$ are the coefficients of the continued fraction expansion of $p/q$.  Thus, if the distinct weights are $X_0:=q> X_1>\dots> X_k = 1$ and we write
$$
W(p/q) = \bigl( X_0^{\times \ell_0}, X_1^{\times \ell_1},\dots, 1^{\times \ell_k}\bigr),
$$
we have
$$
p/q =[\ell_0; \ell_1,\dots,\ell_k] : = \ell_0 + \frac {1}{\ell_1 + \frac{1}{\ell_2 + \dots + \frac{1}{\ell_k}}}.
$$
As we see from Theorems~\ref{thm:Uu},~\ref{thm:Ll} below, the   centers of the staircase steps have very regular continued fraction expansions that, as we now show, behave well under the symmetries.

\begin{lemma}\label{lem:symmCF} {\rm (i)}  The shift $S = \left(\begin{array}{cc} 6&-1\\
1&0\end{array}\right)$ has the following effect on continued fraction expansions, where, for $x\ge 1$, $CF(x)$ denotes the continued fraction of $x$.
\begin{itemize}  
\item
 If $z = p/q = [5+k; CF(x)]$ for some $k\ge0, x\ge 1$, then
$$
Sz =  \frac{6p-q}{p} = [5;1,4+k,CF(x)]. 
$$
In particular, if $z>6$ then $k\ge 1$ and  the third entry in $CF(Sz)$ is at least $5$, while
 if $z<6$ then $k=0$ and $x>1$ and we have 
$$
S\bigl([5;CF(x)]\bigr) = [5;1,4,CF(x].
$$
\end{itemize}

\NI  {\rm (ii)}  The reflection $R = \left(\begin{array}{cc} 6&-35\\
1&-6\end{array}\right)$ has the following effect on continued fraction expansions:
\begin{itemize}  
\item
 If $z = p/q = [6+k; CF(x)]$ for some $k\ge 1, x\ge 1$, then
$$
Rz =  \frac{6p-35q}{p-6q} = [6;k,CF(x)]. 
$$
Further $R\circ R = \id$.
\end{itemize}
\NI  {\rm (iii)}    The quantity $p^2 - 6pq + q^2$ is invariant by both $S$ and $R$.  
\end{lemma}
\begin{proof}   If $z = [5+k; CF(x)] = 5+k + \frac 1x = \frac{(5+k)x + 1}{x}$ then
\begin{align*}
Sz &= \frac{(29 +6k)x + 6}{(5+k)x + 1}=  5 + \frac{(4+k)x+1}{(5+k)x+1},
 \quad\mbox{ while}\\
[5; 1,4+k,CF(x)]& = 5 + \frac 1{1+\frac{1}{4+k +  \frac 1x}} = 5 + \frac{(4+k)x+1}{(5+k)x+1}.
\end{align*}
This proves (i).  The proof of (ii) is similar, and (iii) follows by an easy calculation.
\end{proof}

\subsection{Quasi-perfect classes}\label{ss:qperf}

We first explain the action of the symmetries on the quasi-perfect classes, and then show in Proposition~\ref{prop:block0} that every 
quasi-perfect class  with center $> a_{\min} = 3+2\sqrt2$ is a center-blocking class. 

As explained in \eqref{eq:dioph0},
a quasi-perfect class $\bE= dL - mE_0 - \sum_i m_i E_i$
is determined by a tuple $(d,m,p,q)$ of positive integers
 (where $(m_1,m_2,\dots)$ is the weight expansion of $p/q$ as in \eqref{eq:WT}) that satisfies the conditions
\begin{align}\label{eq:Dioph}
3d=p+q+m,\quad d^2 - m^2 = pq-1.
\end{align}
We will continue to call these the {\bf Diophantine conditions} on $\bE$ as in \cite{ICERM}.
If we use the first equation above to express $m$  as a function of $d,p,q$, the second equation is a quadratic in $d$ 
\begin{align}\label{eq:Xdioph}
8d^2 - 6d(p+q) +p^2 + 3pq+ q^2 -1=0,
\end{align}
with solution
\begin{align*}
d = \tfrac 18 \bigl(3(p+q) \pm \sqrt{p^2-6pq + q^2 + 8}\,\bigr).
\end{align*}
Thus, if we define
\begin{align}\label{eq:tdioph}
t : =  \sqrt{p^2-6pq + q^2 + 8},\quad  \eps: = \pm 1,
\end{align}
the coefficients
$d,m$ in $\bE$ are given by the formulas
\begin{align}\label{eq:dmdioph}
 d: = \tfrac 18( 3(p + q) + \eps t), \quad  m: =  \tfrac 18((p + q) + 3 \eps t)
\end{align}
in \eqref{eq:formdm0}. In other words, modulo an appropriate choice of $\eps$, we can think of a quasi-perfect class as an  integer point on the quadratic surface $X$ defined 
by \eqref{eq:Xdioph}, where we can use either the coordinates $(d,p,q)$ or $(p,q,t)$.\footnote{
We are indebted to Peter Sarnak for explaining this point of view to us.}  Note, however, that the fact that  $p,q,t$ are integers does not imply that $d,m$ are also.
\MS

We now show that 
a quasi-perfect class  $(d,m,p,q,t, \eps)$ is uniquely determined by its center $p/q$.

\begin{lemma} \label{lem:uniqpq} 
For each integral solution $(p,q,t)$ of the equation $t^2 = p^2-6pq+q^2+8$,  there are integers $(d,m)$ satisfying
\begin{align}\label{eq:dm+-}
 d = \tfrac 18 \bigl( 3(p+q) + \eps t\bigr),\qquad  m = \tfrac 18 \bigl((p+q) + \eps 3t\bigr)
\end{align}
 for at most one value of $\eps \in \{ \pm1\}$.
 \end{lemma}
 \begin{proof}  If this is false there are positive integers $p,q,t, d_\pm, m_\pm$ such that 
 $(d_+, m_+)$
 solve \eqref{eq:dm+-} for $\eps = +1$, while $(d_-, m_-)$ solve it for $\eps = -1$.
 Then $$
 d_++d_- = \frac 34(p+q),\quad  d_+ - d_- = \tfrac 14 t
 $$ 
 are integers.  Since $p,q$ cannot both be even they 
 may written as $4a+1, 4b-1$ (in some order), and, with $t: = 4s$  the following equation holds
 $$
 s^2 = a^2 - 6ab + b^2 + 2a - 2b + 1.
 $$
 But also we need $8$ to divide $ 3(p+q) + \eps t$ which implies that $a+b+s$ is even.  It is now easy to see that there are no integer solutions.
 \end{proof}
 
 \begin{cor} There is at most one quasi-perfect class with center $p/q$.   Conversely, for given $(d,m)$ there is at most one quasi-perfect class with these degree variables and $p/q> 1$.
 \end{cor}
  \begin{proof} The first claim follows immediately from Lemma~\ref{lem:uniqpq}. To prove the second, notice that $d,m$ determine $p+q = 3d-m$ and  $pq = d^2-m^2+1$, which uniquely determines $p,q$  modulo order.
  \end{proof}

We now discuss the effect of the symmetries on these classes.
As always, we write
$$
S=\left(\begin{array}{cc}
6&-1\\1&0\end{array}\right),\quad R= \left(\begin{array}{cc}
6&-35\\1&-6\end{array}\right)
$$
where $S$ is the shift and $R$ is the reflection that fixes $7$.   
Note that the action of $S$ on $p,q$ fixes $t$ by  Lemma~\ref{lem:rat}~(ii). 
It is also easy to check that  $R$ also fixes $t$.  
We now show that the action of these transformations act on the integer points of $X$ extends to
an action on the tuples $(d,m,p,q,t,\eps)$.  It follows from Lemma~\ref{lem:Phi0} that any element of the group generated by $S,R$ can be  written $S^iR^\de, \ i\in \Z, \de\in \{0,1\}$.

\begin{definition} \label{def:symmact} 
Let $T = S^iR^\de,  \ i\in \Z, \de\in \{0,1\}$,  and suppose that $(d,m,p,q,t,\eps)$  is a tuple of integers that satisfy the identities in \eqref{eq:tdioph},~\eqref{eq:dmdioph}.  Then
we define
\begin{align}\label{eq:TGg}
T^\sharp(d,m,p,q,t,\eps): = (d',m',p',q',t,\eps') = \bigl(d',m',T(p,q),t, (-1)^{i+\de}\eps\bigr)
\end{align}
where $d',m'$ are given by the formulas 
$$
 d': = \tfrac 18( 3(p' + q') + \eps' t), \quad  m': =  \tfrac 18((p' + q') + 3 \eps' t), \quad \eps': = (-1)^{i+\de}\eps.
$$
\end{definition}

The next lemma shows that this action of $T$ preserves integrality.

\begin{lemma}\label{lem:diophT}  For each $(d,m,p,q,t,\eps)$ and $T = S^i R^\de$ as above, 
 $T^{\sharp}(d,m,p,q,t,\eps)$ is also integral and satisfies the Diophantine conditions~\eqref{eq:Dioph}.
 Moreover, for all such $T_1,T_2$,  $(T_1T_2)^{\sharp}(d,m,p,q,t,\eps) = 
 (T_1)^{\sharp}(T_2)^{\sharp}(d,m,p,q,t,\eps).$
\end{lemma}
\begin{proof}  The above construction shows that every integral point $(p,q,t)$ of $X$ can be extended to a tuple 
$(d,m,p,q,t,\eps)$ that satisfies the Diophantine conditions.  In particular,
since $t$ is invariant under the action of $S,T$, the  tuple $(d',m',p',q',t,\eps')$ satisfies these conditions; however we do need to check that it is integral. 
Because $3\bigl(3(p' + q') + \eps' t\bigr) - \bigl((p' + q') + 3\eps' t\bigr)$ is divisible by $8$, it suffices to check that
$(p' + q') + 3\eps' t $ is divisible by $8$.  Thus it suffices to check that if
$8$ divides $p+q+3\eps t$ for some $\eps \in \{\pm 1\}$, and we set $(p',q') = S(p,q) = (6p-q,p)$, then
$8$ divides $p'+q'-3\eps t = 7p-q - 3\eps t$.  But this is immediate, since   $7p-q - 3\eps t = 8p -(p +q+3\eps t)$.
A similar calculation proves that the action of $R$ preserves integrality.

This proves the first claim. The second follows immediately from the fact that the action of $S,R$ on the coordinates $(p,q)$ is compatible with composition.
\end{proof}

\begin{rmk}\rm The tuples $(d,m,p,q,t,\eps)$ that correspond to quasi-perfect classes have positive entries with $p>q>0$.  As we shall see in Example~\ref{ex:13} the classes with $t=1$ belong to the staircase at $b = 1/3$, while all other classes of interest have $t\ge 3$ and hence $p/q>a_{\min}: = 3 + 2\sqrt 2$ (so that $p^2-6pq + q^2>0$).  Therefore the full subgroup of $\PSL(2,\Z)$ generated by $S,R$ does not act on the staircases.  This is why 
in Theorem~\ref{thm:Gg} we only consider the restriction of the action of the elements $S^iR^{\de}, i\ge 0,$ to the classes with centers $p/q> 7$.\hfill$\er$
\end{rmk}

\begin{rmk}\rm \label{rmk:Ush2}  We now relate our description of the symmetries to that given by Usher in \cite[\S2.2.1]{usher}.  He denotes a quasi-perfect class $\bE$ in a blowup of $S^2\times S^2$ by the tuple $(a,b,c,d)$, where $a,b$ are the coefficients of the two lines (each with Chern class $2$) and $(c,d): = (p,q)$ are the coordinates of its center.  Thus the equations $c_1(\bE) = 1, \bE\cdot\bE = -1$ become
$$
 2(a+b) =p+q, \qquad 2ab = pq-1.
 $$
The first equation implies that there are integers $x, \de, \eps$ such that 
$$
(a,b,p,q) = \bigl(\frac{x+\eps}2, \frac{x-\eps}2, x+\de, x-\de\bigr).
$$
With these variables, the second equation is then 
\begin{align}\label{eq:ush}
x^2 - 2 \de^2 = 2-\eps^2,
\end{align}
In terms of the element  $x + \de \sqrt 2\in \Z[\sqrt 2]$,  this equation simply says that
 $x + \de \sqrt 2$ has norm $2-\eps^2$.
He now considers  symmetries of the form
$$
x + \de \sqrt 2\mapsto x' + \de' \sqrt 2: = (u+v\sqrt 2)(x + \de \sqrt 2) =  ux + 2v\de  + ( vx+ u\de)\sqrt 2,
$$ 
where $u+v\sqrt 2\in \Z[\sqrt 2]$ is an element of norm $1$, i.e.
$u^2-2v^2 = 1$.  Then $(x',\de',\eps)$ is another solution of \eqref{eq:ush}.  Usher considers the symmetries given by $u+v\sqrt 2 = H_{2k} + P_{2k}\sqrt 2$ where $H_{2k}, P_{2k}$ are respectively  half-Pell and Pell numbers. When $k = 1$, $H_{2} = 3, P_2 = 2$, and we get the unit $u+v\sqrt 2 = 3 + 2\sqrt 2$.  
Therefore $x' = 3x + 4\de, \de' = 2x+3\de$ so that $
(p',q'): = (x'+\de',x'-\de') = (5x + 7\de, x + \de)$.  Substituting $x = \frac{p+q}2, \de = \frac{p-q}2$, we obtain the transformation
$$
(p,q)\mapsto (p',q') = (6p-q,p).
$$
More generally, Usher states that the transformation has formula $(x,\de)\mapsto \bigl(H_{2k}x + 2P_{2k}\de, \
P_{2k}x + H_{2k}\de)$, which, in terms of the $(p,q) $ coordinates translates to
$$
(p,q)\mapsto \bigl( (H_{2k}+\tfrac 32 P_{2k})p - \tfrac 12 {P_{2k}}\, q,\ \tfrac12 {P_{2k}}\, p +(H_{2k}-\tfrac 32 P_{2k})q\bigr)
$$
To see  that this map is the same as $(p,q)\mapsto S^{k} (p,q)$, notice that  by Remark~\ref{rmk:Pell}, the entries $y_i$ of $S^k$ have the form $\frac12 P_{2i}$.  Therefore, we have to check a linear identity between the Pell numbers $P_n$ and  their half-companions $H_n$.  But because of the recursion, such an identity holds if and only if it holds for two distinct values of $n$, and  when $k=2$ we have $(H_4,P_4) = (17, 12)$, which gives  $(p,q)\mapsto (35p-6q, 6p-q)$ as required.    

Finally, we observe  that Usher's symmetries preserve $\eps$ which is related to the recursion variable for his staircases, and hence plays much the  same role as our variable $t$.  
\hfill$\er$
\end{rmk}

The next lemma, taken from \cite[Ex.32]{ICERM} is the key to the proof
 that every quasi-perfect class is center-blocking.

 \begin{lemma}\label{lem:volacc}  For all $b\in [0,1)$ the graph of the function $z\mapsto \frac {1+z}{3-b}$ passes through the accumulation point $\bigl(\acc(b), V_b(\acc(b))\bigr)$. Moreover, for each $b$,  the line lies above the volume curve when $z>\acc(b)$.
 \end{lemma}
 \begin{proof} 
Let $c(b)=\frac{(3-b)^2}{1-b^2}-2$. We have
\begin{align}\label{eq:volacc}
&\frac {1+\acc(b)}{3-b}= \sqrt{\frac {\acc(b)}{1-b^2}} \\ \notag
& \qquad\quad  \Longleftrightarrow\;\; \bigl(1+\acc(b)\bigr)^2 = \acc(b)\frac{(3-b)^2}{1-b^2} = \acc(b)\bigl(c(b) + 2\bigr),\\  \notag
& \qquad\quad \Longleftrightarrow\;\;  \acc(b)^2 - c(b) \acc(b) + 1 = 0
\end{align}
which holds by the definition of $\acc(b)$ in \eqref{eq:accb}.   It follows that
the two functions
\begin{align}\label{eq:same0}
b\mapsto  V_b(\acc(b)) \quad\mbox{ and} \qquad   b\mapsto \frac{1+\acc(b)}{3-b}
\end{align}
are the same.

It remains to note that for all $z\ge a_{\min}>5$, the slope $$
\frac{d}{dz}\Bigl(\sqrt{\frac {z}{1-b^2}}\Bigr) =\tfrac 12 \frac{V_b(z)}z = \frac{1+z}{2z(3-b)},\quad z: = \acc(b)
$$
 of the volume curve $V_{b}(z)$ at $z = \acc(b)$ is smaller than $1/(3-b)$ which is the slope of the line. \end{proof}

\begin{rmk}\rm (i)  
This fact is the reason why the  linear relations in \eqref{eq:linrel}  imply that the staircase converges to 
an endpoint of the interval blocked by the associated  blocking class; see the proof of \cite[Thm.52]{ICERM}.
\MS

\NI (ii)  As noticed by Tara Holm, there also seem to be classes that play a similar role for the other convex toric domains discussed in \cite{AADT}.\hfill$\er$
\end{rmk}

\begin{prop}\label{prop:block0}
Every quasi-perfect class $\bE: = (d,m,p,q,t,\eps)$ with $p/q> a_{\min} = 3+2\sqrt 2$  is center-blocking.
 \end{prop}
 \begin{proof}  Define $\acc_\eps^{-1}$ to be $ \acc_U^{-1}$ if $\eps=1$ and $ \acc_L^{-1}$ if $\eps=-1$.\footnote{
By \eqref{eq:dmdioph}  we have $m/d>1/3$ exactly if $\eps = 1$. Moreover the condition $p/q> a_{\min}$ implies that $p/q = \acc(b)$ for at least one $b$ so that it is in the domain of $\acc_\eps^{-1}$ for at least one value of $\eps$.}
 Then  we must check that
 $$
 \mu_{\bE,b}(p/q) = \frac{p}{d-mb} > V_b(p/q) = \frac{p+q}{q(3-b)}, \quad b: = \acc_\eps^{-1}(p/q).
 $$
 where we have used \eqref{eq:obstr0} and \eqref{eq:same0}.  Thus we need
 $$
 pq(3-b) > (p+q)(d-mb),  
 $$
 or equivalently
 $$
 3pq - d(p+q) > b\bigl(pq-m(p+q)\bigr).
 $$ 
By \eqref{eq:bsi} we have $b = \frac{3pq+\eps(p+q)\sqrt {\si}}{(p+q)^2 + pq}$, where $\si + 8 = t^2$.  Thus we must check that
$$
\bigl((p+q)^2 + pq\bigr)\bigl(3pq - d(p+q)\bigr) > \bigl(3pq+\eps(p+q)\sqrt {\si}\bigr)\bigl(pq-m(p+q)\bigr).
$$
By deleting the term $3p^2q^2$ from both sides, multiplying by $8$ and substituting for $d,m$, we  obtain  the equivalent inequality 
\begin{align*}
&24 pq(p+q)^2 - (p+q)\bigl((p+q)^2 + pq\bigr)\bigl(3(p+q)+\eps t\bigr) \\
&\qquad > - 3pq (p+q) (p+q+3\eps t) +\eps (p+q)\sqrt{\si}\Bigl(8pq - (p+q+3\eps t)(p+q)\Bigr)\\
&\qquad =  - 3pq (p+q) (p+q+3\eps t) +\eps (p+q)\sqrt{\si}\Bigl( - \si - 3\eps t(p+q)\Bigr),
\end{align*}
where the last equality uses the identity $8pq - (p+q)^2 = -\si$.
If we take all  the terms in this inequality that involve an even power of $\eps$, and put them on the LHS, we obtain
\begin{align*}
&  24 pq(p+q)^2 - 3(p+q)^2\Bigl((p+q)^2 + pq\Bigr) + 3pq(p+q)^2 + 3(p+q)^2 t \sqrt\si\\
  &\qquad = 3(p+q)^2\Bigl(8pq -(p+q)^2 - pq + pq +\si + (t\sqrt \si - \si)\Bigr)\\
  &\qquad =  3(t\sqrt \si - \si)(p+q)^2 =:A > 0,
  \end{align*}
  since $\si = p^2 -6pq + q^2$ and $t^2 = \si + 8$.
If we do the same with the coefficient of $\eps$,  we obtain
\begin{align*}
& -t(p+q)\bigl((p+q)^2 + pq\bigr) + 9t(p+q)pq  +(p+q)\si\sqrt\si\\
 &\qquad =  t(p+q)\bigl((-p^2-3pq - q^2 + 9pq -\si\bigr) - (t-\sqrt\si)(p+q)\si\\
 &\qquad =  (\sqrt{\si}-t)(p+q)\si= :B.
  \end{align*}
We need to check that $A> |B|$, which is equivalent to $3(p+q) >\sqrt{\si}$.  Since this holds,
 the required inequality is established. 
 \end{proof}

\begin{rmk}\label{rmk:block1}\rm 
Proposition~\ref{prop:block0} shows that every class that is defined by a tuple  $(d,m,p,q,t,\eps)$ with $p/q>a_{\min}$ as in \eqref{eq:formdm0} is in fact a center-blocking class.    Thus all our stair steps are center-blocking.  However,
we have not been able to resolve the  question of whether there is a blocking class $\bB$ of more general type  that is not center-blocking.  In this case, there would be a point $b_0\in [0,1)$ such that $\mu_{\bB,b_0}(\acc(b_0))> V_{b_0}(\acc(b_0))$. However, 
if  $I$ is the largest interval containing $\acc(b_0)$ on which  $\mu_{\bB,b_0}$ is obstructive, and if $a\in I$ is the corresponding break point (see~\cite[Lem.14]{ICERM}), then $\mu_{\bB,b}$ would not block $a$, i.e. for both elements $b\in \acc^{-1}(a)$ we would have  $\mu_{\bB,b}(a) \le V_b(a)$.  (For further discussion of blocking classes, see~\cite[\S2.3]{ICERM}.) We  bypass this question here by restricting attention to 
(quasi-perfect) center-blocking classes.\hfill$\er$
\end{rmk}

The following fact about blocking classes   was pointed out to us by Morgan Weiler.
It is somewhat surprising since we know from \cite[Lem.15(iii)]{ICERM} that  every quasi-perfect class 
$\bB = (d,m,p,q)$ is obstructive when $b=m/d, z=p/q$, i.e.  we have $\mu_{\bB,m/d}(p/q)> V_{m/d}(p/q)$.  Thus it is natural to think that $m/d$ would lie in the blocked interval $J_\bB$, which is defined to be the maximal interval containing $\acc_\eps^{-1}(p/q)$ consisting of parameters $b$ such that $\mu_{\bB,b}(\acc(b)) > V_b(\acc(b))$.  
However, we now show that this never happens.

\begin{lemma}\label{lem:block} Every quasi-perfect  class $\bB = (d,m,p,q,t\eps)$  with $\eps=1$ (resp. $\eps=-1$)
has the property that
$m/d>b$ (resp. $m/d<b$), for all $b$ in the closure of  the blocked $b$-interval $J_\bB$.
 \end{lemma}
\begin{proof}   
By \eqref{eq:obstr} and Lemma~\ref{lem:volacc}, the upper end point of the interval $I_{m/d}$ containing $p/q$ on which $\mu_{\bB,m/d}$ is obstructive
is given by the solution $z_1$ of the equation
$$
\mu_{\bB,m/d}(z_1)  = \frac p{d-m\cdot m/d} = \frac {1+z_1}{3-m/d}.
$$
Thus, using the identities $d^2-m^2 = pq-1$ and $3d-m = p+q$ we see that
$z_1 = \frac{p^2+1}{pq-1}$.  
Hence if $z_0 = \acc_\eps(m/d) \in I_{m/d}$, i.e. if $m/d\in J_{\bB}$, we must have $z_0 < z_1$.
On the other hand, by definition $z_0$ is the unique solution $>1$ of the equation
$$
z_0 + \frac 1{z_0} = \frac{(3-m/d)^2}{1-(m/d)^2} - 2,
$$
and the function $z\mapsto z+1/z$ increases when $z>1$.  
Therefore, $z_0 < z_1$ only if $z_0 + \frac 1{z_0} <z_1 + \frac 1{z_1}$.
But, because $z_0$ satisfies \eqref{eq:accb}, we have 
\begin{align*}
z_0 + \frac 1{z_0} &= \Bigl(\frac{(3-m/d)^2}{1-(m/d)^2} - 2\Bigr)\\
&= \frac {p^2+q^2 +2}{pq-1} = z_1 + \frac {q^2+1}{pq-1} \\
& > z_1 + \frac{pq-1}{p^2+1} =   z_1 + \frac 1{z_1}.
\end{align*}
Thus $z_1<z_0 = \acc_\eps(m/d).$
The result now follows because  the function $z\mapsto \acc_\eps^{-1}(z)$ preserves orientation for $\eps=1$ and reverses it for $\eps=-1$.
\end{proof}


We end this subsection with a few remarks about
the case $b=1/3$, which is the focal point of the shift $S$ and separates the two regimes $b>1/3, b<1/3$.
As far as we know, this is the unique rational value of $b$ with a staircase.\footnote
{
In particular, there should be no staircases at  the points $b\in \acc^{-1}(v_i)$. This was proved in~\cite[Thm.6]{ICERM} for the case $b=1/5\in   \acc^{-1}(6)$, but the proof seems too elaborate to be easily generalized.}
Example~\ref{ex:13} below describes the ascending staircase at $b=1/3$.  Although we have not managed to resolve the question of whether there is also a descending staircase at $b=1/3$, we can make the following observation. Note that the proof uses the same idea as in Lemma~\ref{lem:live1}.

\begin{lemma}\label{lem:13}  If there is no descending staircase when $b=1/3$, then
there is $\eps>0$ such that 
$c_{1/3}(z) = \frac{3(1+z)}8$ for $3+2\sqrt2 = a_{\min}<z<a_{\min}+\eps$. \end{lemma}
\begin{proof}    We saw in Remark~\ref{rmk:lowdeg}~(ii) that the obstruction $\mu_{\bE_1,b}$ given by the class $\bE_1: = 3L - E_0-2E_1 - E_{2\dots 6}$ is precisely $z\mapsto \frac{1+z}{3-b}$ when $5<z< 6$.  Thus for $z\in [a_{\min},6]$ we know that $c_{H_{1/3}}(z)\ge \frac{3(1+z)}8$.  If we do not have equality for $z\in (a_{\min}, a_{\min} + \eps)$ and there is no staircase, then there must be a different obstruction curve  $z\mapsto \frac{A + Cz}{d-m/3}$ 
that goes through the accumulation point $(a_{\min}, 3\sqrt{\frac {a_{\min}}8})$.  Because $a_{\min}$ is irrational, the equation
$\frac{A+Ca_{\min}}{{d-m/3}} =  \frac{3(1+a_{\min})}{8}$ can hold only if  $A=C$.  But then the graphs of the two obstructions are lines 
of the form $z\mapsto \la (1+z)$, and hence coincide.
\end{proof}

 This following observation is also relevant because, for example, 
$d_\bB-3m_\bB$ appears as the coefficient of $d$ in the linear staircase relation \eqref{eq:linrel}, so that it would be awkward if it were ever zero.

 \begin{lemma} \label{lem:13no}
There is no quasi-perfect class $\bE = (d,m,p,q)$ with $m/d=1/3$.
\end{lemma}
\begin{proof}  By \eqref{eq:dioph0}, given any such $\bE$ the integers $m,p,q$ would have to satisfy $8m = p+q, 8m^2= pq-1$, and hence also $p^2-6pq+q^2 = -8$.  But this Diophantine equation has no solution.  (Working mod $8$, we would have $p\equiv_8-q$ and hence $p^2\equiv_8  -1$, which is easily seen to be impossible.) 
\end{proof}

\subsection{The known staircases}\label{ss:recap}

The paper \cite{ICERM} found three different families of blocking classes $\bB^U$, $\bB^E$, and $\bB^L$ and their associated staircases. For our purposes, $\Ss^U$ and $\Ss^L$ are the essential families, since applying powers of $S$ to these staircases generate all the staircase families discussed here.  

We now review the theorems in 
\cite[Thm.56,58]{ICERM} that define these two staircase families.
In all cases the staircase steps  $\bE_{n,k}$ are quasi-perfect classes given by tuples $\bigl(d_{n,k}, m_{n,k}, p_{n,k}, q_{n,k}\bigr)$ that for $x=d,m,p,q$ satisfy the recursion described below.  We call $p_{n,k}/q_{n,k}$ the center of the step, and use the staircase relation\footnote
{
In \cite{ICERM} we did not yet realize the role of the variable $t$.} (together with the linear Diophantine condition $3d=m+p+q$) to determine the entries $d,m$ from knowledge of $p,q$.

 \begin{thm}\label{thm:Uu}   The classes $\bB^U_n = \bigl(n+3,n+2, 2n+6,1\bigr), n\ge 0,$
 with increasing centers, 
 are perfect blocking classes, with the following associated staircases  $\Ss^{U}_{\ell,n}, \Ss^{U}_{u,n}$, where  $\si_n=(2n+5)(2n+1)$ and $\eend_n=(2n+4)$ or $(2n+5,2n+2)$:
 \begin{itemize}  \item for each $n\ge 1$, $\Ss^{U}_{\ell,n}$ has limit point $a^{U}_{\ell,n, \infty}= 
 [2n+5;2n+1, \{2n+5,2n+1\}^\infty]$, 
and its 
\begin{align*}
\begin{array}{ll}
\mbox{\rm (Centers)}&\quad  [\{2n+5,2n+1\}^k,\eend_n], \;  \ k\ge 0,\\
& \qquad \mbox{ where } \  \eend_n = 2n+4\;\;\mbox{ or } (2n+5,2n+2),\\  \notag
\mbox{\rm (Recursion)}&\quad x_{n,k+1} = (\si_n + 2)x_{n,k} - x_{n,k-1},\\ \notag
\mbox{\rm (Relation)}&\quad  (2n+3)d_{n,k} = (n+1) p_{n,k} + (n+2) q_{n,k}.  
\end{array}
\end{align*}
 \item for each $n\ge 0$, $\Ss^{U}_{u,n}$ has limit point 
 $a^{U}_{u,n,\infty}= 
 [2n+7; \{2n+5,2n+1\}^\infty]$, and
 has 
 \begin{align*}
 \begin{array}{ll}
\mbox{\rm (Centers)} &\quad  [2n+7;\{2n+5,2n+1\}^k,\eend_n], \\ \notag
\mbox{\rm (Recursion)} &\quad x_{n,k+1} = (\si_n + 2)x_{n,k} - x_{n,k-1}, 
\\ \notag
\mbox{\rm (Relation)} &\quad (2n+3)d_{n,k} = (n+2)p_{n,k} - (n+4) q_{n,k}.
\end{array}
\end{align*}
 \end{itemize}
   The  limit points $a^U_{\bullet,n, \infty}$  form  increasing unbounded sequences in $(6,\infty)=(v_2,v_1)$, while the corresponding $b$-values lie in $ (5/{11}, 1\bigr)$, where $5/{11} = \acc_U^{-1}(6)$, and increase with limit $1$. 
\end{thm}

\begin{rmk}\label{rmk:U}\rm (i)
If $\nu$ is an integer $\ge 3$, the recursion $x_{k+1} = \nu x_k - x_{k-1}$  has a unique solution $> 1$ of the form $x_k = \al^k$
where $\al^2 - \nu \al + 1 = 0$.   
If the recursion has seeds $x_0, x_1$, the general solution can be written  $X\al^k + \ov X {\ov{\al}}^k$ where $\ov \al: = 1/\al$ is the other solution, $X \in \Q[\sqrt{\nu^2-4}]$, and for $X = a + b\sqrt{\nu^2-4}$ we define
$\ov X: = a - b\sqrt{\nu^2-4}$.
Hence, if $x_\bullet, y_\bullet$ both satisfy this recursion, the ratio $x_\bullet/y_\bullet$ converges to the quantity $X/Y$. It is thus straightforward to calculate  quantities such as $a^{U}_{u,n,\infty}$ from knowledge of the recursion plus its seeds; see Lemma~\ref{lem:recur}.
\MS

\NI (ii)  The staircases $\Ss^U_{\bullet,n}$ are described above as having two intertwined strands,  one for each $\eend_n$.  We show in Lemma~\ref{lem:recur1} that these classes may be combined into a single family with recursion variable $2n+3$.    This  simpler description of the staircases clarifies  their essential structure. 
\MS

\NI (iii)  Notice that there is no ascending staircase in this family with $n=0$ since the centers of its steps would be $< 6$,  the center of $\bB^U_0$.  Of course, there is a staircase of this kind, but we view it here as the image of $\Ss^U_{u,0}$ by the reflection $R_{v_2} = SR$, and so consider it a member of the staircase family $(SR)^\sharp(\Ss^U)$.\MS

\NI (iv)  Finally, note that staircases that are associated to a blocking class $\bB$ and labelled with the subscript $\ell$ are always ascending, and converge to the {\it lower} end of the $z$-interval blocked by $\bB$, while those labelled $u$ descend and converge to the {\it upper} endpoint of this $z$-interval. \hfill$\er$
\end{rmk}

There is a corresponding definition of the staircase family $\Ss^L$.

 \begin{thm}\label{thm:Ll}  The classes $\bB^L_n = \bigl(5n,n-1,{12n+1},{2n})\bigr), n\ge 1,$ with decreasing centers, 
 are  perfect and center-blocking, and have the following associated staircases  $\Ss^{L}_{\ell,n}, \Ss^{L}_{u,n}$ for $n\ge 1$, with $\si_n$ and $\eend_n$ as in Theorem~\ref{thm:Uu}~:
 \begin{itemize}  \item  $\Ss^{L}_{\ell,n}$ is ascending, with limit point $a^{L}_{\ell,n, \infty}= 
 [6;2n+1, \{2n+5,2n+1\}^\infty]$,
	and has
\begin{align*}
\begin{array}{ll}
\mbox{\rm (Centers)} &\quad
 [6;2n+1, \{2n+5,2n+1\}^k,\eend_n],  \\ \notag
\mbox{\rm (Recursion)} &\quad  x_{n,k+1} = (\si_n + 2)x_{n,k} - x_{n,k-1}, \;\; \si_n: = (2n+1)(2n+5)\\ \notag
\mbox{\rm (Relation)} &\quad (2n+3)d_{n,k} =   (n+1) p_{n,k} - (n-1) q_{n,k}.
\end{array}
\end{align*}
 \item $\Ss^{L}_{u,n}$ is descending, with limit point $a^{L}_{u,n, \infty}= 
 [6;2n-1,2n+1, \{2n+5,2n+1\}^\infty]$
 and has 
 \begin{align*}
 \begin{array}{ll}
\mbox{\rm (Centers)} &\quad  [6;2n-1,2n+1,\{2n+5,2n+1\}^k,\eend_n], \\ \notag
\mbox{\rm (Recursion)} &\quad x_{n,k+1} = (\si_n + 2)x_{n,k} - x_{n,k-1},
\\ \notag
\mbox{\rm (Relation)} &\quad (2n+3)d_{n,k} = -(n-1)p_{n,k} + (11n+2)q_{n,k}.
\end{array}
\end{align*}
 \end{itemize}
  The  limit points $a^L_{\bullet,n,\infty}$ (with $\bullet = \ell$ or $u$)   form a decreasing sequence in $(6,7)=(v_2,w_1)$ with limit $6$, while the corresponding $b$-values lie in $(0,1/5)$ and increase with limit $1/5$.
\end{thm}

\begin{rmk}\label{rmk:Ll}\rm (i)
It follows from Lemma~\ref{lem:symmCF}~(i),~(ii) that the symmetry $R$ takes the centers both of the blocking classes and of the staircase steps for the family $\Ss^U$ into those for $\Ss^L$.  Notice that in the case of $\Ss^U_{\ell,n}$ it takes the step with label $k$ to the step  in  $\Ss^U_{u,n}$ with label $k-1$. 
Note that, with this choice of labelling, the image by $R$ of the step with center $[2n+7; 2n+4]$ (which appears both as $\Ss^U_{\ell,n+1,0}$ and as $\Ss^U_{u,n,1}$) has no counterpart in the staircase
$\Ss^L_{u,n+1}$, though it could be added to it. 
\MS

\NI (ii) {(\bf The Fibonacci stairs)} The Fibonacci stairs are the ascending stairs that should be associated to $\bB^L_0 = R^\sharp(\bB^U_0)$.  However  such a class would have center at $R(6) = \infty$, and so it does not exist
as a geometric obstruction.  
Nevertheless, if we ignore the first few steps, the steps of the Fibonacci stairs have precisely the form predicted by putting $n=0$ in the formulas for $\Ss^L_{\ell,n}$, namely: they have
 \begin{align*}
 \begin{array}{ll}
\mbox{\rm (Centers)} &\quad  [6;1,\{5,1\}^k,\eend_0], \\ \notag
\mbox{\rm (Recursion)} &\quad x_{k+1} =  5 x_{k} - x_{k-1},
\\ \notag
\mbox{\rm (Relation)} &\quad 3 d_{k} = p_{k} + q_{k}.
\end{array}
\end{align*}
Moreover, 
although the class $\bE': = 3L (-0E_0) -  2E_1-E_2-\dots E_7$ is not perfect, its obstruction $\mu_{\bE',0}(z)$ for $z\in (6,7]$ is the function $z\mapsto \frac{1+z}{3}$, which goes through the  point $(a_{0,\infty}, V_{0}(a_{0,\infty}))$ (where $a_{0,\infty} = \tau^4$) and equals $c_{H_0}(z)$ for $z\in [\tau^4, 7]$.  Therefore this class $\bE'$ plays the geometric role of the blocking class, and
 we consider these stairs as part of the family $\Ss^L = R^\sharp(\Ss^U)$. 
As we explain before Lemma~\ref{lem:recur2},  there is a different tuple  with negative entries that plays the numeric role of the missing blocking class; below we denote this by  $\bB^L_0$.
 
    Notice that all the other families $(S^i)^\sharp(\Ss^L), i\ge 1,$  have an ascending staircase for $n=0$ that is associated with the
 blocking class $(S^i)^\sharp(\bB^L_0)$, which now has positive entries and so is geometric.\hfill$\er$
\end{rmk}

 Lemma~\ref{lem:block} shows that for every $n\ge 0$ the ratio $(n+2)/(n+3)$ does not lie in the $b$-interval $J_{\bB^U_n}$ blocked by  $\bB^U_n$.  The next lemma locates this point more  precisely.
  
\begin{lemma}\label{lem:blockU}   Let $(d_n,m_n) = (n+3,n+2)$ and $a_n = 2n+6$ be the degree variables and center of the blocking class $\bB^U_n = (n+3,n+2, 2n+6,1)$. Then for all $n\ge 0$, we have $\acc(m_n/d_n) < a_{n+1}$ and $m_n/d_n \in J_{\bB^U_{n+1}}$.
%
\end{lemma}
\begin{proof}  
To see that $z_n: = \acc(m_n/d_n) < a_{n+1} = 2n+8$, it suffices to check that
$z_n + \frac 1{z_n} < 2n+8 + \frac 1{2n+8}$.  But  \eqref{eq:accb} implies that
\begin{align*}
z_n + \frac 1{z_n} &= \frac{\bigl(3-\frac{n+2}{n+3}\bigr)^2}{1-(\frac{n+2}{n+3})^2} - 2\\
& = \frac{4n^2 +24n+39}{2n+5} \\
&
 = 2n+8 - \frac{2n+1}{2n+5} < 2n+8 +\frac{1}{2n+8}.
\end{align*}

Note that $m_n/d_n \in J_{\bB^U_{n+1}}$ exactly if  $\mu_{\bB^U_{n+1},m_n/d_n}(z_n) > V_{m_n/d_n}(z_n) = \frac{1+z_n}{3-m_n/d_n}$.
Since $z_n< a_{n+1} = 2n+8$, \eqref{eq:obstr} implies that we have
$$
\mu_{\bB^U_{n+1},m_n/d_n}(z_n)= \frac{z_n}{d_{n+1} - \frac{m_n}{d_n}\,m_{n+1}} =\frac{z_n}2.
$$
Thus we need $z_n/2 > (1+z_n)(n+3)/(2n+7)$, which holds because $z_n>2n+6$.
\end{proof}

\begin{rmk}\label{rmk:block} \rm We show in Corollary~\ref{cor:bk} that for all our staircases, whether ascending or descending, the ratios $(m_k/d_k)_{k\ge 1}$ decrease when $m_k/d_k>1/3$ and increase when $m_k/d_k<1/3$. 
 (Since the  staircase steps are defined recursively,  Corollary~\ref{cor:monot} shows that this follows from the structure of the first two steps.)  Further  Proposition~\ref{prop:Uu} shows that the blocking class $\bB^U_n$ can be considered as a step in the ascending stair $\Ss^U_{\ell,n+1}$ associated to  $\bB^U_{n+1}$. Hence $m_n/d_n$ is part of a decreasing sequence that limits on the lower endpoint of $J_{\bB_{n+1}^U}$. Thus, once we have shown that
 $m_n/d_n < a_{n+1}$, this reasoning implies  that  $m_n/d_n \in J_{\bB_{n+1}^U}$.  More generally,
 an analog of Lemma~\ref{lem:blockU} holds for any pre-staircase family.  
\hfill$\er$
\end{rmk}

\begin{example}\label{ex:13} {\bf(The staircase at $b = 1/3$)}\rm  \; 
This staircase, discovered in \cite{AADT}, behaves in a different way from all other known staircases in $H_b$. 
It has three interwoven sequences: 
$$
\bE_{k,i} = \bigl(d_{k,i},\ m_{k,i},\ p_{i,k}={g_{k,i}},\ q_{i,k}={g_{k-1,i}}\bigr),\quad i = 0,1,2
$$
where for each $i$ the numbers $g_{k,i}$ satisfy the recursion $$g_{k+1,i}=6g_{k,i}-g_{k-1,i}$$
with seeds (i.e. initial values)
$$
g_{0,0} = 1,\;\;g_{1,0} = 2,\quad g_{0,1} = 1,\;\;g_{1,1} = 4,\quad g_{0,2} = 1,\;\;g_{1,2} = 5.
$$
The centers $p_{i,k}/q_{i,k}: = g_{k,i}/g_{k-1,i}$
have  continued fractions $[5;\{1,4\}^{k},\eend_i]$, where the possible ends are $\eend_0 = 2,
\eend_1=(1,3)$ and $\eend_2 =\emptyset$.  
\MS

Normally one needs two seeds, say $x_0,x_1$, to fix an iteration of the form $x_{k+1} = \nu x_k - x_{k-1}$.  However the variables $p,q$ for the $1/3$-staircase have the form $x_k,x_{k-1}$, and so are part of a single recursive sequence. 
Thus we can consider that each of the three strands of the $1/3$-staircase has a single  seed $\bE_{seed} = (d,m,p,q)$, namely
$$
\bE_{seed,0} = (1,0,2,1),\quad \bE_{seed,1} = (2,1,4,1),\quad \bE_{seed,2} = (2,0,5,1),
$$
and that the rest of the staircase comes from the images of the $(p,q)$ components of these three classes under the action of  $S: (p,q)\mapsto (6p-q,p)$, with $(d,m)$ determined by a modification of \eqref{eq:dmdioph} as follows. We have
\begin{align} \label{eq:dm13}
    d_{k,i}=\tfrac18(3(g_{k,i}+g_{k-1,i})+\eps_{k,i} t_i, \quad m_{k,i}=\tfrac18(g_{k,i}+g_{k-1,i}+3\eps_{k,i} t_i),
    \end{align}
  where  $\eps_{k,i} = (-1)^{k+i}$, while  $t_i$ is constant w.r.t $k$ and equal to
$$
t_i=\sqrt{g_{1,i}^2+g_{0,i}^2-6g_{1,i}g_{0,i}+8}=\begin{cases} 1 & \text{if $i=0,1$} \\ 2 & \text{if $i=2$} \end{cases}.
$$
Note that  $\eps_{k,i} = 1$ if $m_{k,i}/d_{k,i}>1/3$ and $=-1$ otherwise.

Note that, in contrast with the case of $\Ss^U,\Ss^L$ discussed in Remark~\ref{rmk:U}~(ii) above,  it is {\it not} 
possible to combine these three interwoven sequences into a single recursive sequence; for example, there is no constant $c$ such that for all $k$
\[g_{k,2}=cg_{k,1}-g_{k,0}.\]  
Further the relation between the $d,p,q$ variables in this staircase is significantly different from the homogenous  linear relation satisfied by the  other staircases.
By \cite[Prop.49]{ICERM} the latter relation implies that the ratios $m_{k,i}/d_{k,i}$ are monotonic and  converge as $k\to \infty$ so quickly that the step classes are themselves center-blocking classes.  On the other hand, classes with centers $< a_{\min}$ cannot be center-blocking, and the ratios $m_{k,i}/d_{k,i}$ lie on both sides of $1/3$.
Another  important distinction between this staircase and the others is the fact that  the $t$-variable is fixed for each strand.
The point here is that in all cases the variable called $t$ is fixed by the shift $S$.  This shift implements the recursion of the $(1/3)$-staircase, while it takes every other staircase to a different one.

\MS

Finally, we remark that the three seed classes $\bE_{seed,i}, i=0,1,2$, have rather different natures.
The first two have the form of the blocking classes $$
\bB^U_n: = (n+3,n+2,2n+6,1), \quad\;  n\ge 0,
$$
 and indeed are given by taking $n = -2,-1$ in this formula.  One might consider that the $\bB^U_n, n\ge -2,$ form a single family of classes, that behave differently according as to where the center $2n+6$ lies in relation to $a_{\min}$.
If the center is $< a_{\min}$ then iteration by $S$ gives a staircase  for $b=1/3= \acc^{-1}(a_{\min})$, while if the center is $> a_{\min}$ then we get a family of blocking classes that block $b$-intervals that converge to $1/3$, but do not form a staircase since they are not live when $b=1/3$.   (To see this, note that
$t = |d-3m|$ is fixed by $S$, while by \cite[Lem.15]{ICERM}  a perfect class $\bE$ with center $p/q$ 
 is obstructive at its center for a given value of $b$ if and only if $|bd-m|<\sqrt{1-b^2}$.
Hence a class that is live at $b=1/3$ must have $|t|<3$, while $\bB^U_n$ has $t = 2n+3$.)\MS  

The third seed $ \bE_{seed,2} = (2,0,5,1)$ with $(t,\eps) = (2,-1)$ is significantly different.  Both it and $\bE_{seed,0} = (1,0,2,1)$ are steps in the Fibonacci staircase at $b=0$, but the iteration that gives this sequence is not $S$ but rather $x_{\ka+1} = 3x_\ka - x_{\ka-1}$.  All the other steps in the Fibonacci staircase have centers $> a_{\min}$, and are blocking classes by \cite[Lem.41]{ICERM}.  
On the other hand,  because $S(1,1) = (5,1)$, the strand formed by $ \bE_{seed,2} $ and its iterates under $S$ can be extended one place backwards by the tuple
\begin{align}\label{eq:seed}
\bE_{-1,2}: = (1,1,1,1),\quad \mbox{with }\ (t,\eps) = (2,1).
\end{align}
It turns out that the sequence of classes formed by $\bE_{-1,2}, \bE_{seed,2},$ and the images of $\bE_{seed,2},$ by $S^i$ have an important role to play in generating
the staircase families; see for example Corollaries~\ref{cor:Uu},~\ref{cor:Ll} and Lemma~\ref{lem:seeds}.
 \hfill$\er$
\end{example}
 
 \begin{rmk}\label{rmk:lowdeg}\rm {\bf (The role of low degree classes)} (i)   The three seeds of the $1/3$-staircase 
are given by the only exceptional curves in $H_b$ with degree at most two.  
These classes have centers $< a_{\min}  = 3+ 2\sqrt2$, while all the other classes of interest here have centers  $> a_{\min}$.
\MS

\NI (ii)
There is one exceptional divisor of degree three, that gives rise to three potentially interesting exceptional divisors in $H_b$ that we will label by the coefficient of $E_0$, namely\footnote
{here we use the shorthand of \cite{ball}, where $E_{i\dots j}: = \sum_{k=i}^jE_k$.} 
$$
\bE_0: = 3L -  2E_1 -E_{2\dots 7}, \quad  \bE_1: = 3L - E_0 - 2E_1 -E_{2\dots 6}, \quad \bE_2: = 
3L-2E_0 - E_{1\dots 6}.
$$
Each of these classes has a special role to play:  
\begin{itemlist}\item[-] 
 $\bE_0$ substitutes geometrically (but not numerically) for the blocking class for the Fibonacci stairs (see Remark~\ref{rmk:Ll}~(ii));
\item[-]
$\bE_1$ is a (non perfect) class with breakpoint $6$. This class is discussed extensively in \cite[Ex.32]{ICERM}.   It is obstructive at $z=6$ for $\acc_L^{-1}(6) = 1/5 < b <  5/11 = \acc_U^{-1}(6)$, and for this range of $b$ 
we have
$$
\mu_{\bE,b}(z) =\frac {1+z}{3-b},\;\; z\in (5,6),\qquad \mu_{\bE,b}(z)= \frac7{3-b},\;\; z\ge 6.
$$
The obstruction given by this class is discussed further in Lemma~\ref{lem:volacc}.  Its properties turn out to be a crucial element in our proofs.  

\item[-] $\bE_2$ is the perfect blocking class $\bB^U_0$ with center $6$.   \hfill$\er$
\end{itemlist} 
\end{rmk}

\section{Structure of the pre-staircase families}\label{sect:structure} 

We begin by explaining the recursion that underlies our staircases.
In \S\ref{ss:seed}, we then  describe  the staircase families $\Ss^U,\Ss^L$  in the new language, 
and show that they are indeed generated by their blocking classes together with two seeds, as claimed in
 Proposition~\ref{prop:seed}.

In \S\ref{ss:Gg}, we prove Theorem~\ref{thm:Gg}. This establishes that for each $T=S^iR^\delta$, $T^\sharp(\Ss^U)$ is a pre-staircase family. This entails examining how the seeds and blocking classes are transformed under $S$. Furthermore, this implies that each pre-staircase in $T^\sharp(\Ss^U)$  is associated to a blocking class, an essential feature of a pre-staircase family. 
In \S\ref{ss:symmact}, we compute how $S^\sharp$ and $R_{v_i}^\sharp$ act on the degree coordinates $(d,m)$ of the blocking classes. 
\MS\MS

\subsection{The single recursion}\label{ss:diopht}

The following lemma shows that we can consider each staircase in $\Ss^U,\Ss^L$ to be a single family of classes satisfying the recursion
$x_{\ka+1} = (2n+3)x_\ka - x_{\ka-1}$ with parameter $2n+3$ instead of a double family of classes
as in Theorems~\ref{thm:Uu},~\ref{thm:Ll} that each satisfy the recursion $x_{k+1} = \bigl((2n+1)(2n+5)+2\bigr)x_k - x_{k-1}$ with parameter $(2n+1)(2n+5)$.
Note that an increase of $k$ by $1$ corresponds to an increase of $\ka$ by $2$.  

\begin{lemma}\label{lem:recur00}
\begin{itemize}\item[{\rm (i)}] Let $m\ge 1$ be any integer, and consider
\begin{align*}
&\frac {p_1}{q_1}  = [m; 2n+4],\;\; \frac {p_2}{q_2}  = [m; 2n+5,2n+2],\;\;\\
&  \frac {p_3}{q_3}  = [m; 2n+5,2n+1,2n+4],\;\;
\frac {p_4}{q_4}  = [m; 2n+5,2n+1,2n+5, 2n+2].
\end{align*}
Then for $x_\bullet = p_\bullet, q_\bullet $ we have $x_{\ka+1} = (2n+3)x_\ka - x_{\ka-1}, \ka=1,2$.
\item[{\rm (ii)}]  If $x_\bullet$ is a sequence such that $x_{\ka+1} = \nu x_\ka - x_{\ka-1}$ for all $\ka$, then $$
x_{\ka+2} =(\nu^2 - 2)x_{\ka} - x_{\ka-2}, \quad  \mbox{for all } \ \ka.
$$
  In particular, if $\nu=2n+3$ then $\nu^2 - 4 = (2n+1)(2n+5)$.
\item[{\rm (iii)}]   If $y_\ka, z_\ka$ satisfy the recursion $x_{\ka+1} = \nu x_\ka - x_{\ka-1}$, then
$y_{\ka+2}z_{\ka} - z_{\ka+2}y_\ka = \nu (y_{\ka+1}z_\ka - z_{\ka+1}y_\ka)$
\end{itemize}
\end{lemma}
\begin{proof}  (i) by direct calculation.    (ii) holds because $\nu x_{\ka-1} = x_\ka + x_{\ka-2}$ so that 
$$
x_{\ka+2} = \nu x_{\ka+1} - x_{\ka} = \nu(\nu x_\ka - x_{\ka-1}) - x_\ka = (\nu^2-1)x_\ka - \nu x_{\ka-1} =
(\nu^2-2)x_\ka - x_{\ka-2}.
$$
Finally (iii) holds because $$
y_{\ka+2}z_\ka - z_{\ka+2}y_\ka  = \det\left(\begin{array}{cc} z_\ka&y_\ka\\
z_{\ka+2}&y_{\ka+2}\end{array}\right) =\det\left(\begin{array}{cc} z_\ka&y_\ka\\
\nu z_{\ka+1}&\nu y_{\ka+1}\end{array}\right).
$$
This completes the proof.  
\end{proof}

We next show that this recursion extends naturally to the triples $(p_\ka, q_\ka, t_\ka)$ that parametrize quasi-perfect classes as in Definition~\ref{def:symmact}, provided that the initial seeds satisfy the compatibility condition \eqref{eq:tcompat} that is given below.  Further, \eqref{eq:pqt} below states that the tuple $(p,q,t)$ is the coordinate of a point on the surface $X$ defined by the equation \eqref{eq:Xdioph}.  Thus  the next lemma gives compatibility conditions on the seeds $(p_0,q_0,t_0), (p_1,q_1,t_1)$ under which the recursion with parameter $\nu$ extends to the integral points of  $X$.

It is convenient to use matrix notation with
\begin{align}\label{eq:matrix}
&A : = \left(\begin{array}{ccc} -1&3&0\\3&-1&0\\0&0&1\end{array}\right),\quad  \bx: = 
 \left(\begin{array}{c} p\\q\\t\end{array}\right).
 \end{align}
 Thus the matrix $A$ is symmetric, and $X = \{\bx \ | \ \bx^T A \bx = 8\}$.
 
  \begin{lemma}\label{lem:recur0}  {\rm (i)}
 Suppose that $\bx_0, \bx_1$ are integral vectors  that satisfy the following  conditions for
some integer  $\nu>0$:
\begin{align}\label{eq:pqt}  & \bx_i^T A \bx_i = 8, 
 \quad i = 0,1,\\ \label{eq:tcompat}
&\bx_1^T A \bx_0 = 4\nu.
\end{align}
Then  the vectors $\bx_2: = \nu \bx_1 - \bx_0, \bx_1$ also satisfy these conditions for the given $\nu$.
\end{lemma}
\begin{proof}
Since  $A$ is symmetric,  
$$
(\nu \bx_1 - \bx_0)^T A (\nu \bx_1 - \bx_0) = \nu^2 \bx_1^T A \bx_1 - 2\nu \bx_1^T A \bx_0 + \bx_0^T A \bx_0 = 8
$$
exactly if  $8 \nu^2 = 2\nu \bx_1^T A \bx_0$, which holds by \eqref{eq:tcompat}.
Further \eqref{eq:tcompat} holds for $\nu \bx_1 - \bx_0, \bx_1$ because
$$
(\nu \bx_1 - \bx_0)^T A \bx_1 = 8\nu - (\bx_0)^T A \bx_1  = 4\nu,
$$
by \eqref{eq:tcompat}.
This completes the proof. \end{proof}

\begin{cor}\label{cor:recur0}  Any two integral triple $\bx_i = (p_i,q_i,t_i), i=0,1$  that satisfy 
\eqref{eq:pqt}, \eqref{eq:tcompat} for a given $\nu$ can be extended to a sequence $\bx_i, i\ge 0$, using recursion parameter $\nu$, where each successive pair satisfies these conditions.  Further, for fixed $\eps\in \{\pm1\}$, the corresponding quantities
$$
d_i = \tfrac 18\bigl(3(p_i+q_i) + \eps t_i\bigr),  \quad m_i = \tfrac 18\bigl(p_i+q_i + 3\eps t_i\bigr)
$$
of \eqref{eq:dmdioph} also satisfy this recursion and hence are integers, provided  that they are integers for $i=0,1$.
\end{cor}

We end this section with some useful formulas about recursively defined sequences. 
The following is an adaptation of \cite[Lem.47]{ICERM}.

\begin{lemma} \label{lem:recur}   Let $x_{\ka }, \ka \ge 0,$ be a sequence of integers that satisfy the recursion 
\begin{align}\label{eq:recur1}
x_{\ka +1} = \nu x_{\ka } - x_{\ka -1},\quad \nu\ge 3,
\end{align} and let $$
\la =  \frac{\nu+\sqrt \si}{2}
\in \Q[\sqrt{\si}]
$$ 
be the larger root of the equation
$x^2 - \nu x + 1=0$, where $\si = \nu^2-4>1$.  Then there is a number  $X \in \Q[\sqrt{\si}]$  such that
\begin{align}\label{eq:recurX1}
x_\ka   = X\la^\ka  + \ov X\, {\ov\la}\,\!^\ka ,
\end{align}
where $\ov{a + b\sqrt\si}: = (a - b\sqrt\si)$, so that $\la \ov\la = 1$.  
Further, if we write  $X = X'+ X''\sqrt \si$, then
\begin{align}\label{eq:recurX}
 X' = \frac{x_0}{2},\qquad X'' = \frac{2x_1 - \nu x_0}{2\si}.
 \end{align}
\end{lemma}  
 \begin{proof}  
 If the monomials $x_\ka = c^\ka $ satisfy the recursion then we must have
$c^2-\nu c+1=0$, so that
  $c  =  \frac 12\bigl(\nu \pm  \sqrt{\nu^2-4}\bigr) =  \frac 12\bigl(\nu \pm  \sqrt{\si}\bigr)$.      Let $\la$ be the larger solution, so that $\ov \la$ is the smaller one, and we have $\la \ov\la = 1$.
 Since  \eqref{eq:recur1}  has a unique solution once given the seeds $x_0, x_1$, 
  it follows that  for each choice of constants $A,B$, the numbers 
\begin{align}\label{eq:lambda}
  x_\ka : = A \la^\ka  +  B\ov\la^\ka 
\end{align}
  form the unique solution with
  $$
  x_0 = A+B,\qquad  x_1 = A \la + B \ov \la.
  $$
 Then $A,B\in  \Q[\sqrt{\si}]$.  Notice also that $\sqrt \si$ is irrational since $\si = \nu^2-4$ is never a perfect square when $\nu\ge 3$. 
It  follows easily that $x_0, x_1\in \Q$ only if 
we also have $B: = \ov A$. 

Thus we have $$
x_0 = X + \ov X = X'+ X''\sqrt \si + X'- X''\sqrt \si = 2X'
$$
which gives  $X' = {x_0}/2$, and
$$
x_1 = (X'+ X''\sqrt \si)\ \frac{\nu+\sqrt \si}{2} +  (X'- X''\sqrt \si)\ \frac{\nu-\sqrt \si}{2}
= X'\nu + X'' \si,
$$
which gives  $X'' = (2x_1 - \nu x_0)/ (2\si)$, as claimed
\end{proof}

\begin{cor}\label{cor:monot}   Let $x_{\ka }, y_\ka, \ka \ge 0,$ be  increasing sequences that both  satisfy \eqref{eq:recur1} for some $\nu\ge 3$ and are such that $x_\ka,y_\ka$ are positive for $\ka\geq 1$.  Suppose further that at most one of $x_0, y_0$ are zero. Then:\MS

\NI {\rm (i)}   The ratios $(x_\ka/y_\ka)_{\ka\ge1}$ form a monotonic sequence that is strictly increasing if
$x_1y_0 - x_0y_1>0$ and is strictly decreasing if $x_1y_0 - x_0y_1<0$.
\MS

\NI  {\rm (ii)} In all cases, $\lim_{\ka\to \infty} x_\ka/y_\ka$ exists and equals $X/Y$, where $X,Y$ are the constants defined in \eqref{eq:recurX}.
\MS

\NI  {\rm (iii)}  Provided that $x_1y_0 - x_0y_1\ne 0$, the limit $X/Y$ is irrational.
\end{cor}
\begin{proof} 
Note first that
$$
x_{\ka+1} y_{\ka} - x_{\ka}y_{\ka+1} = (6x_{\ka}-x_{\ka-1}) y_{\ka} - x_{\ka}(6y_{\ka}-y_{\ka-1}) =x_\ka y_{\ka-1} -
x_{\ka-1}y_\ka 
 $$
 is independent of $\ka$.    Hence the sequence $x_\ka/y_\ka, \ka\ge 1,$ is monotonic\footnote{If $x_0/yo \leq 0$, then the sequence is only monotonic for $\ka \ge 1$. This will be the case for the seeds of $\Ss^U_u$ computed in Lemma~\ref{lem:recur1}~(ii)}, and  the quantity $x_1y_0 - x_0y_1$ determines whether it is increasing or decreasing.   
Part (ii) is an immediate consequence of the formula in \eqref{eq:lambda}, and the fact that $\la>1$.

To prove the third claim, suppose first that $2x_1 > \nu x_0$.  Then 
$X = X' + X''\sqrt \si$ is irrational since the coefficient $X''$ of $\sqrt \si$ does not vanish, and $\si = \nu^2-4$ is not a perfect square.  Similarly, if  $2y_1 > \nu y_0$, $Y$ is irrational.   Further 
\begin{align*}\frac XY &= \frac{(X'+X''\sqrt \si)(Y'-Y''\sqrt \si)}{(Y'+Y''\sqrt \si)(Y'-Y''\sqrt \si)}\\
& = \frac{(X'Y'- X''Y''\si) + (X''Y'-X'Y'')\sqrt \si}{(Y')^2 - (\si Y'')^2}
\end{align*}
is irrational unless $X''Y'-X'Y''=0$.  But formula \eqref{eq:recurX} implies that $X''Y'-X'Y''$ is a multiple of $x_1y_0 - x_0y_1$ and so is nonzero by hypothesis. 

To deal with the possibility that $X''$ or $Y''$ equals $0$, notice that  the limit ratio $X/Y$ does not depend on which pair of terms $k_0, k_0+1$ we take as the initial terms (though the values of $X,Y$ do depend on this choice).  Further, because by hypothesis $\nu \ge 3$ and $x_0<x_1$  we have $x_2 = \nu x_1 - x_0$ so that  $2x_2 > \nu x_1$.  Similarly, $2y_2> \nu y_1$.
Therefore if we define the quantities $X, Y$ as in \eqref{eq:recurX} but starting with $\ka = 1$  then the above argument shows that the ratio $X/Y$ is irrational.  \end{proof}

\subsection{Generating the known staircase families}\label{ss:seed}
We now prove Proposition~\ref{prop:seed}; its proof is contained in Corollary~\ref{cor:Uu} and \ref{cor:Ll}.
In all cases considered in this paper, the recursion variable $\nu = 2n+3$ for the pre-staircase $\Ss^\Ff_{\bullet,n}$ 
in a staircase family $\Ff$ is a linear function of  $n$, and we will enumerate the terms $x_\ka$ in our recursively defined sequences so that the value of $x_\ka$ is a polynomial of degree $\ka$ in $n$.  In particular, $x_0$ is a constant.  In contrast the staircase sequences were enumerated in Theorems~\ref{thm:Uu},~\ref{thm:Ll}  via a number $k$
that counted the  iterations of the repeated pair $\{2n+5, 2n+1\}$ in the continued fraction expansion 
of $p_k/q_k$.   

 \begin{lemma}\label{lem:recur1} \begin{itemize}\item[{\rm (i)}]    If $(p_{0},q_{0}, t_{0}) = (1,1,2)$ and $(p_{1},q_1,t_1) = (\nu+1,1, \nu-2)$ for any $\nu\ge 3$,  then the identities in Lemma~\ref{lem:recur0} hold.
Further, with $(d_0,m_0) = (1,1)$ and $(d_1,m_1) = (1,\frac{\nu + 1}2)$ the identities in \eqref{eq:dmdioph} hold with $\eps = 1$.
 
 \item[{\rm (ii)}]   If $(p_{0},q_{0}, t_{0}) = (-5,-1,2)$ and $(p_{1},q_1,t_1) = (\nu+5,1, \nu+2)$ for any  $\nu\ge 1$,  then the identities in Lemma~\ref{lem:recur0} hold,   
Further, with $(d_0,m_0) = (-2,0)$ and $(d_1,m_1) = (\frac{\nu + 5}2,\frac{\nu + 3}2)$ the identities in \eqref{eq:dmdioph} hold with $\eps = 1$.

 \item[{\rm (iii)}]  If we define the triple $(p_\ka,q_\ka, t_\ka)$ by the recursion $x_{\ka+1} = \nu  x_\ka - x_{\ka-1}$ for $\ka\ge 1$, then the ratios $p_\ka/q_\ka$ form an increasing sequence in case (i) and a decreasing sequence in case (ii).  
 \end{itemize}
 \end{lemma}
 \begin{proof}  The claims in (i), (ii) hold by an easy computation.     To prove (iii), it suffices by Corollaries~\ref{cor:recur0} and~\ref{cor:monot} to check the first two terms.
  But  in case (i) we have $p_1/q_1 > p_{0}/q_{0}$, while in case (ii)
  $p_1/q_1 = \nu + 5\; > \; p_2/q_2 = (\nu^2 + 5\nu + 5)/(\nu + 1)$.
 \end{proof}

\begin{prop}\label{prop:Uu}
All the classes involved in the staircase family $\Ss^U$ can be extended to tuples $(d,m,p,q,t,\eps)$  with $\eps = 1$.  Moreover  the staircase $\Ss^U_{\bullet,n}$ has recursion parameter $2n+3$, which is the $t$-coefficient in $\bB^U_n$, and it can be extended to have the seeds described in Lemma~\ref{lem:recur1}.  More precisely,
\begin{itemize}\item the ascending staircase $\Ss^U_{\ell,n}, n\ge 1,$ has initial step  
given by the tuple $(d,m,p,q,t) = (1,1,1,1,2)$ and  next step (at $\ka=1$) given by the ($t$-extended) blocking class 
$$
\bB^U_{n-1} = (n+2,n+1, 2n+4,1, 2n+1).
$$
\item the descending staircase $\Ss^U_{u,n}, n\ge 1,$ has initial step  
given by the tuple $(d,m,p,q,t) =  (-2,0,-5,-1,2)$ and  next step (at $\ka=1$) given by the ($t$-extended) blocking class 
$$
\bB^U_{n+1}= (n+4,n+3, 2n+8,1, 2n+5).
$$ 
\end{itemize}
\end{prop}

\begin{cor} \label{cor:Uu}  The family  $\Ss^U$ is generated in the sense of Definition~\ref{def:prestf} by its blocking classes $\bB^U_{n} = (n+3,n+2, 2n+6,1, 2n+3)$ together with the seeds 
$$
\bE^U_{\ell,seed} = (1,1,1,1,2),\quad \bE^U_{u,seed} = (-2,0,-5,-1,2), \qquad \eps = 1.
$$
Moreover, for all staircases in this family, whether ascending or descending, the ratios $m_\ka/d_\ka$ decrease.
\end{cor}

\begin{proof}[Proof of Proposition~\ref{prop:Uu}] It is straightforward to check that
  the $d,m$ coordinates in  $\bB^U_n: = (n+3,n+2, 2n+6,1), \ n\ge 0$  are given by the formulas in
  \eqref{eq:dmdioph}  with 
  \begin{align}
(p,q,t) = (2n+6,1, 2n+3), \quad\mbox{ and }\ \eps = 1.  
 \end{align}
By Lemma~\ref{lem:recur00}, if for each $n\ge 1$ we enumerate the ascending staircase $\Ss^U_{\ell,n}$ as a single staircase that is indexed by the degree $\ka$ of $p_\bullet$ as a function of $n$, then this staircase has
\begin{itemize}\item [-] recursion parameter $\nu = 2n+3$, 
\item [-] initial steps with centers $p_1/q_1 = (2n+4)/1$, and $$
p_2/q _2= [2n+5;2n+2] = \frac{(2n+5)(2n+2) + 1}{2n+2};$$ 
\item [-]  linear relation $(2n+3) d_\ka = (n+1)p_\ka + (n+2)q_\ka$.
\end{itemize}
The linear relation implies  that the  $d$ values of the first two steps are $d_1 = n+2$, $d_2 = (n+1)(2n+5)$.
Note also that the class with center at $p_1/q_1$ is precisely $\bB^U_{n+1}$.

By the definition of $t$ in \eqref{eq:tdioph}, we have that $t_1 = 2n+1 = \nu -2$ and that $d_1$ is as predicted 
by \eqref{eq:dmdioph} with $\eps = 1$.
Further  if we take $(p_0,q_0,t_0)= (1,1,2)$ as in Lemma~\ref{lem:recur1}~(i), then
$$
(p_2,q_2) = \bigl(\nu p_1 - p_0, \nu q_1 - q_0\bigr).
$$
Therefore we may think of the tuple $(d_2,m_2,p_2,q_2, t_2)$ as given by a recursion with $\nu = 2n+3, \eps = 1$,
with initial terms 
\begin{align}
(p_{0},q_{0}, t_{0}) = (1,1,2), \quad\mbox{ and }\ \  (d_{0},m_{0}) = (1,1).
\end{align} 
Note that the entries $d_{0},p_{0},q_{0}$ also satisfy the linear relation $(2n+3) d_\ka = (n+1)p_\ka + (n+2)q_\ka$. 
Therefore, as in Corollary~\ref{cor:recur0}, all subsequent terms in this staircase must have degree coefficients $(d,m)$ given by  
\eqref{eq:dmdioph}.   Thus, as claimed, for each $n$  this sequence is generated by the tuple $(d,m,p,q,t) = (1,1,1,1,2)$ together with the appropriate blocking class.
\MS

Similarly, by Lemma~\ref{lem:recur00},  the descending staircase classes for $\Ss^U_{u,n}, n\ge 0, $ when indexed by $\ka$ have 
\begin{itemize}\item [-] recursion parameter $\nu = 2n+3$, 
\item [-] initial steps with centers $p_2/q_2 = [2n+7;2n+4]=\frac{(2n+7)(2n+4) + 1}{2n+4}$, and $$
\frac{p_3}{q _3}= [2n+7; 2n+5,2n+2] = \frac{(2n+7)\bigl((2n+5)(2n+2)+1\bigr) + 2n+2}{(2n+5)(2n+2)+1};
$$ 
\item [-]  linear relation $(2n+3) d_\ka = (n+2)p_\ka - (n+4)q_\ka$.
\end{itemize}
The linear relation implies  that the corresponding $d$ values are $d_2 = 2n^2 + 11n+14$, $d_3 = 
4n^3+28n^2 + 60n + 38$.    Note that we can add the blocking class $\bB^U_{n+1} = (n+4,n+3, 2n+8,1)$ to the staircase as the step for $\ka = 1$, because
\begin{align*}
& p_1: =  (2n+3)p_2 - p_3 =    2n+ 8,\quad q_1: = (2n+3)q_2-q_3 = 1,
\end{align*}
and the appropriate linear relation $(2n+3)(n+4) = (n+2)(2n+8) - (n+4)$ holds.

In fact, 
this staircase  has the form of Lemma~\ref{lem:recur1}~(ii), with initial tuples
\begin{align}
(p_{0},q_{0}, t_{0}) = (-5,-1,2), \quad (p_{1},q_{1},t_{1}) = (2n+8,1, 2n+5).
\end{align}
The next entry in the recursive sequence is then
 $$
 (p_2,q_2, t_2) = \bigl((2n+8)(2n+3)+5, 2n+4, (2n+7)(2n+3)+1\bigr),
$$
which has the same values for $p_2,q_2$ as does
the first staircase step given above.   
The formulas in \eqref{eq:dmdioph} give 
\begin{align}
(d_{0},m_{0}) = (-2,0), \quad 
(d_{1},m_{1}) = (n+4,n+3) 
\end{align}
 with $\eps = 1$.  Further, the linear relation $(2n+3) d_\ka = (n+2)p_\ka - (n+4)q_\ka$ holds in both cases.
Thus, again, this staircase is generated by the initial tuple 
 $(d,m,p,q,t) = (-2,0,-5,-1,2)$ together with the appropriate blocking class.
\end{proof}

\begin{proof}[Proof of Corollary~\ref{cor:Uu}]
The first claim is an immediate consequence of Proposition~\ref{prop:Uu}.  By Corollary~\ref{cor:monot} to check the second claim we must check that $m_1d_0-m_0d_1 <0$ for all   staircases.  For the descending staircases, this is immediate because $m_0 = 0, d_0 < 0$ while $m_1>0$.  For the ascending staircases, this holds because $m_0=d_0 = 1$ while $m_1<d_1$.
\end{proof}

\begin{rmk}\rm\label{rmk:negparam}  (i)  The parameters $(d,m,p,q) = (1,1,1,1)$ of the initial step in the ascending staircases 
 do correspond to those of an exceptional class in $H_b$, albeit one that is not live for the relevant $z$ values (which are all $> 3 + 2\sqrt2=a_{\min}$).   Thus its parameters are  both geometrically and numerically meaningful. (See also the discussion concerning \eqref{eq:seed}.)    On the other hand,  
 the parameters $(-2,0,-5,-1)$ of the initial step in the descending staircases are negative and so, though numerically meaningful, have no immediate interpretation in terms of a point $p/q\in (1,\infty)$.  Instead, they parametrize the ray  $\{(-5,-1)\la\ | \ \la>0\}$ in $\R^2$.  Though we do not do this here, one could think of the symmetries 
 in terms of their action on these rays; see Remark~\ref{rmk:seeds}.
  \MS
 
 \NI (ii)  For each $n\ge 1$, the ascending staircase $\Ss^U_{u,n+1}$ with recursion parameter $t_{\bB}= 2n+5$ has the same second step as the descending
  staircase $\Ss^U_{\ell,n}$ with $t_{\bB}=2n+3$.  Indeed, the formulas given above show that
  this step has center $(d,m,p,q,t,1)$ with $p,q$ given by
   $$
   p/q= [2n+5; 2n+2] = [2(n-1)+7;2(n-1)+4]
  $$
   and  with 
  $t = (2n+3)(2n+5) -2$. \hfill$\er$
 \end{rmk}
 \MS
 
 There is a similar story for the staircase family $\Ss^L$, except that now there is no geometric blocking class for $n=0$, and the tuple $(d,m,p,q,t) = (0,-1,1,0,3)$ that replaces it  has no obvious geometric meaning. 
 Nevertheless, we define $\bB^L_0: = (0,-1,1,0,3)$ for the current purposes. (See Remark~\ref{rmk:Ll}~(ii).)

  Here is the appropriate numerical lemma.
 
  \begin{lemma}\label{lem:recur2} \begin{itemize}\item[{\rm (i)}]    If $(p_{0},q_{0}, t_{0}) = (5,1,2)$ and $(p_{1},q_1,t_1) = (6\nu-5,\nu-1, \nu+2)$ for any $\nu>1$,  then the identities in Lemma~\ref{lem:recur0} hold.
Further, with $(d_0,m_0) = (2,0)$ and $(d_1,m_1) =  (\frac 52 (\nu-1),\frac{\nu - 3}2)$ the identities in \eqref{eq:dmdioph} hold with $\eps = -1$.
 
 \item[{\rm (ii)}]   If $(p_{0},q_{0}, t_{0}) = (-29,-5,2)$ and $(p_{1},q_1,t_1) = (6\nu-29,\nu-5, \nu-2)$ for any  $\nu\ge 3$,  then the identities in Lemma~\ref{lem:recur0} hold,   
Further, with $(d_0,m_0) = (-13,-5)$ and $(d_1,m_1) = (\frac 52 (\nu-5),\frac{\nu - 7}2)$ the identities in \eqref{eq:dmdioph} hold with $\eps = -1$.

 \item[{\rm (iii)}]  If we define the triple $(p_\ka,q_\ka, t_\ka)$ by the recursion $x_{\ka+1} = \nu  x_\ka - x_{\ka-1}$ for $\ka\ge 1$, then the ratios $p_\ka/q_\ka$ form an increasing sequence in case (i) and a decreasing sequence (for $\ka\ge 2$) in case (ii).  
 \end{itemize}
 \end{lemma}
\begin{proof}  This is left to the  reader.
\end{proof}

\begin{prop}\label{prop:Ll}
All the classes involved in the staircase family $\Ss^L$ can be extended to tuples $(d,m,p,q,t,\eps)$ 
 with $\eps = -1$.  Moreover,  the staircase $\Ss^L_{\bullet,n}$ has recursion parameter $2n+3$, which is the $t$-coefficient in $\bB^L_n$, and it can be extended
 to have the seeds described in Lemma~\ref{lem:recur2}.  More precisely,
\begin{itemize}\item the ascending staircase $\Ss^L_{\ell,n}, n\ge 0,$ has initial step  
given by $(d,m,p,q,t) = ( 2,0,5,1,2)$ and  with the next step (at $\ka=1$) given by the blocking class 
$\bB^L_{n+1}$ with $t = 2n + 5$;
\item the descending staircase $\Ss^L_{u,n}, n\ge 2,$ has initial step  
given by  $(d,m,p,q,t) =  (-13,-5, -29,-5,2)$ and  with the next step (at $\ka=1$) given by the blocking class 
$\bB^L_{n-1}$ with $t = 2n + 1$.  When $n=1$ there is the same initial step, and the step at $\ka=1$ is given by the tuple $\bB^L_{0}: = (0,-1,1,0,3)$.
\end{itemize}
\end{prop}

\begin{cor} \label{cor:Ll}  The family  $\Ss^L$ is generated in the sense of Definition~\ref{def:prestf} by its blocking classes $\bB^L_n: = (5n, n-1, 12n+1,2n, 2n+3), n\ge 1,$ together with the seeds 
$$
\bE^L_{\ell,seed} = (2,0,5,1,2),\quad \bE^L_{u,seed} = (-13,-5, -29,-5,2), \quad \eps = -1.
$$
where in the definition of $\Ss^L_{u,1}$ we take the tuple $\bB^L_{0}$ to be $(0,-1,1,0,3)$.
Further, for all staircases in the family with $n>0$, both ascending and descending, the ratios $m_\ka/d_\ka$ increase.
\end{cor}

\begin{proof}[Proof of Proposition~\ref{prop:Ll}]
  The blocking classes $\bB^L_n: = (5n, n-1, 12n+1,2n), \ n\ge 1$ are given by the formulas in \eqref{eq:dmdioph}, with 
  \begin{align}\label{eq:Llblock}
(p,q,t) = (12n+1,2n, 2n+3), \quad\mbox{ and }\ \eps = -1.  
 \end{align}

By Lemma~\ref{lem:recur00}, if for each $n\ge 0$ we enumerate the ascending staircase $\Ss^L_{\ell,n}$ as a single staircase that is indexed by the degree $\ka$ of $p_\bullet$ as a function of $n$, then this staircase has
\begin{itemize}\item [-] recursion parameter $\nu = 2n+3$, 
\item [-] initial steps $p_\ka/q_\ka$  with centers \begin{align*}
p_2/q_2 &= [6;2n+1,2n+4] = \frac{24n^2 + 62 n + 34}{4n^2 + 10n + 5} ,\\
p_3/q _3&= [6;2n+1,2n+5,2n+2] = \frac{(24n^2 + 98 n + 121)2n + 89}{(4n^2 + 16n + 19)2n + 13}.
\end{align*}
\item [-]  linear relation $(2n+3) d_\ka = (n+1)p_\ka - (n-1)q_\ka$.
\end{itemize}
The linear relation implies  that   $(d_2, m_2) = (10 n^2+ 25 n + 13, 2n^2 + 3n)$.

One can easily check that the values $(p_2,q_2)$ and $(p_3,q_3)$ given above agree with those in the recursive sequence with 
$\nu = 2n+3$ and initial terms
\begin{align*}
&(d_0,m_0,p_0,q_0,t_0) = (2,0,5,1,2), \\
& (d_1,m_1,p_1,q_1,t_1)= (5(n+1), n, 12n+13, 2n + 2, 2n+3).
\end{align*}
Hence because the tuple for $\ka = 0,1$ satisfies \eqref{eq:dmdioph} with $\eps = -1$, all classes in the staircase have this form.   This establishes the claims concerning $\Ss^L_{\ell,n}$.

By Lemma~\ref{lem:recur00}, if for each $n\ge 1$ we enumerate the descending staircase $\Ss^L_{u,n}$ as a single staircase that is indexed by the degree $\ka$ of $p_\bullet$ as a function of $n$, then this staircase has
\begin{itemize}\item [-] recursion parameter $\nu = 2n+3$, 
\item [-] initial steps with centers 
\begin{align*}
p_3/q_3 &= [6;2n-1,2n+1,2n+4] = \frac{48n^3 + 100 n^2 + 22n -1}{8n^3 + 16 n^2 + 2n-1} ,\\
p_4/q _4&= [6;2n-1, 2n+1,2n+5,2n+2]=\frac{96n^4+344n^3+320n^2+50n+1}{16n^4+56n^3+48n^2+2n-2}
\end{align*}
\item [-]  linear relation $(2n+3) d_\ka = -(n-1)p_\ka + (11n + 2)q_\ka$.
\end{itemize}
Again one can check by direct computation
that this sequence is generated by the tuples
\begin{align*}
&(p_0,q_0,t_0) =(-29,-5,2),\quad (p_1,q_1,t_1) =( 12n-11,2n-2,2n+1), \qquad \mbox{ with} \\
& (p_2,q_2,t_2) = \bigl(24n^2 + 14 n - 4, 4n^2 + 2n -1, 4n^2 + 8n + 1\bigr),
\end{align*}
which have the form described in Lemma~\ref{lem:recur2}~(ii) with $\nu = 2n+3, \eps = -1$. Moreover, formula \eqref{eq:dmdioph} gives the following values for $d,m$ with $\eps = -1$:
$$
(d_0, m_0) = (-13,-5),\quad (d_1,m_1) = (5n-5, n-2).
$$
This completes the proof.
\end{proof}

\begin{proof}[Proof of Corollary~\ref{cor:Ll}]
The first claim is an immediate consequence of Proposition~\ref{prop:Ll}.  By Corollary~\ref{cor:monot} to check the second claim we must check that $m_1d_0-m_0d_1 > 0$.  For the descending staircases, 
we have $m_0=-5, d_0=-13, m_1 = n-2, d_1 = 5n-5$ and $n\ge 1$, so that $m_1d_0 =  -13(n-2) > -5(5n-5)=-25n+5 = m_0d_1$.
  For the ascending staircases with $n\ge1$, this holds because $m_0=0$ while $m_1,d_0>0$.
\end{proof}

\subsection{Proof of Theorem~\ref{thm:Gg}}\label{ss:Gg}
We now prove Theorem~\ref{thm:Gg} stating that each symmetry  $T\in \Gg$ transforms the staircase family $\Ss^U$ into another pre-staircase family $T^\sharp(\Ss^U)$; in particular  $T=R$ interchanges the families $\Ss^U, \Ss^L$.  

We know from Definition~\ref{def:symmact} and Lemma~\ref{lem:diophT} how 
 $T = S^iR^\de$ acts on quasi-perfect classes. The corresponding definition for staircase families is as follows.

 \MS

\begin{definition} \label{def:symmactT} 
Given $T = S^iR^\de\in \Gg$,  we define
 $T^{\sharp}(\Ss^U)$  to be the collection of pre-blocking classes 
$\bigl(T^{\sharp}(\bB_n^U)\bigr)_{n\ge0}$ together with the seeds
$T^{\sharp}(\bE^U_{seed,\ell}), \ T^{\sharp}(\bE^U_{seed,u})$.
\end{definition}

We first prove our earlier claim that the reflection $R$ takes the family $\Ss^U$ to the family $\Ss^L$.
Recall from Propositions~\ref{prop:Uu},~\ref{prop:Ll} that all the classes in $\Ss^U$ have $\eps =1$ while all those in $\Ss^L$ have $\eps = -1$.

\begin{lemma}\label{lem:Ract}  {\rm (i)} The map $R^\sharp$ takes the blocking classes and seeds of  the staircase family $\Ss^U$ together with all the associated staircase steps  to the corresponding elements in the family $\Ss^L$.  
Further $\bE^U_{u,seed} = -\bE^L_{\ell,seed}$, where for a class $\bE = (d,m,p,q,t,\eps)$ we define
$-\bE : = (-d,-m,-p,-q,t,-\eps)$.
\MS

\NI {\rm (ii)}  Moreover, $S^\sharp(\bE^U_{\ell,seed}) = \bE^L_{\ell,seed}, \ S^\sharp(\bE^U_{u,seed}) = \bE^L_{u,seed}$.
\end{lemma}
\begin{proof}
We already noted in \cite[Cor.60]{ICERM} that the reflection $R$ takes the step classes in $\Ss^U$ to those of $\Ss^L$.
By Corollaries~\ref{cor:Uu},~\ref{cor:Ll}  the seeds of these families are:
\begin{align*}
\bE^U_{\ell,seed},\ \bE^U_{u,seed}:&\qquad (d,m,p,q,t,\eps) = \ (1,1,1,1,2,1),\ (-2,0,-5,-1,2,1)\\
\bE^L_{\ell,seed},  \bE^L_{u,seed}:&\qquad (d,m,p,q,t, \eps)= \ (2,0, 5,1,2,-1),\ (-13,-5,-29,-5,2, -1)
\end{align*}
Therefore  $\bE^U_{u,seed} = -\bE^L_{\ell,seed}$ as claimed.  Moreover, because
 $R(1,1) =  (-29,-5)$ and $R(5,1) = (-5,-1)$, we find that 
$$
R^\sharp(\bE^U_{\ell,seed}) = \bE^L_{u,seed}, \qquad R^\sharp(\bE^U_{u,seed}) = \bE^L_{\ell,seed}.
$$
Since $R$ is a reflection that interchanges the ascending and descending staircases,  this reversal of seeds is to be expected.  Note also that the centers of the blocking classes $\bB^L_n$ descend rather than ascend, so that this reversal is also consistent with Definition~\ref{def:prestf} that explains how the blocking classes  and seeds 
generate a staircase.   Thus $R$ takes the full structure of the family $\Ss^U$ to that of $\Ss^L$.

The proof of (ii) is straightforward, and is left to the reader.
\end{proof}
\MS

In order to show that for arbitrary $T\in \Gg$ the seeds and pre-blocking classes $T^{\sharp}(\Ss^U)$ define  a staircase family,
we must show that the appropriate linear relations \eqref{eq:linrel} hold.  This proof is largely based on analyzing the seed classes.
\MS

As already noted, modulo sign, the seeds $\bE^U_{\bullet,seed}, \bE^L_{\bullet,seed}$ for the staircase families $\Ss^U,\Ss^L$ are  classes that appear in the third strand (the one with $i=2$) of the staircase at $b = 1/3$; see Example~\ref{ex:13}.    This strand is generated by the recursion $S$ with action on $(d,m)$ given by equation~\eqref{eq:dm13}, and hence consists of the classes 
\begin{align}\label{eq:seed0}
&\qquad \bE_i: = \bigl(d_i,m_i, g_i, g_{i-1}, 2, (-1)^i\bigr), \ i\ge 0 \;\;\mbox{  where }\\ \notag
&d_i = \tfrac18 \bigl(3(g_i + g_{i-1}) + 2(-1)^i\bigr), \quad 
m_i = \tfrac1{24}\bigl(3(g_i + g_{i-1}) + 18(-1)^i\bigr),
\end{align}
with   values given in the following table
\begin{align} \label{table:gi}
\begin{array}{c|c|c|c|c|c|c|c|c}\hline
i&-1&0&1&2&3&4&5&6\\ \hline
g_{i}&1&1&5&29&169&985&5741&33461\\ \hline
m_i&&1&0&5&24&145&840&4901\\\hline
d_i&&1&2&13&74&433&2522&14701\\ \hline
\end{array}
\end{align}
Note that $g_{i} = y_{i+1} - y_{i}$ for all $i$. Further the $\eps$-coordinate alternates
between the value $+1,-1$, while $t=2$ for all $i$.

\begin{lemma}\label{lem:seeds}  
For all $i\ge0$, the  families $(S^{i+1})^\sharp(\Ss^U)$ and $(S^iR)^\sharp(\Ss^U)$ have the same lower seeds, and the same upper seeds, namely 
\begin{align*}
\bE^{(S^{i+1})^\sharp(\Ss^U_\ell)}_{\ell,seed} = \bE^{(S^iR)^\sharp(\Ss^U_u)}_{\ell,seed} =\bE_{i+1},
\quad
\bE^{(S^{i+1})^\sharp(\Ss^U_u)}_{u,seed} = \bE^{(S^iR)^\sharp(\Ss^U_\ell)}_{u,seed} = -\bE_{i+2},
\end{align*}
where $\bE_i$ is as in \eqref{eq:seed0}. Furthermore, the extended action $R^\sharp_{v_{i+2}}$ of the reflection $R_{v_{i+2}}$
that is defined  in \eqref{eq:TGg}  interchanges the lower and upper seeds of these families.
\end{lemma}
\begin{proof}  The first statement is an easy consequence of Lemma~\ref{lem:Ract}~(ii) and the fact that $S^\sharp(\bE_i) = \bE_{i+1}$ for all $i$. 

For the second statement, note that by Corollary~\ref{cor:ident}~(i)
$(R_{v_{i+2}})^\sharp=(S^{2i+1}R)^\sharp = S^{i}R S^{-(i+1)}$. The upper seed of $\Ss^U_u$ is $-\bE_1$. 
Hence, to compute the lower seed of $(R_{v_{i+2}}^\sharp)((S^{i+1})^\sharp(\Ss^U_u))$, we are considering  
\[
(S^{i}R S^{-(i+1)})^\sharp((S^{i+1})^\sharp(-\bE_1))=(S^i)^\sharp R^{\sharp}(-\bE_1)=(S^i)^\sharp(\bE_1) = \bE_{i+1}
\]
where the first equality follows from  Lemma~\ref{lem:diophT}.
The upper seed can be computed similarly. 
\end{proof}

Recall from the dicussion concerning \eqref{eq:linrel}, that in order for a staircase $\Ss$ with steps $\bE: = (d,m,p,q,t,\eps)$ to be associated to the pre-blocking class 
$\bB = \bigl(d_\bB,m_\bB,p_\bB,q_\bB,t_\bB,\eps\bigr)$ we require that for each step the following  linear relation holds:
\begin{align}\label{eq:linrel1}
(3m_\bB-d_\bB)\, d& = (m_\bB - q_\bB)\, p +m_\bB\, q & \quad \mbox{ if }\; \Ss \; \mbox{ ascends}\\ \notag
(3m_\bB-d_\bB)\, d& =   m_\bB\, p - (p_\bB-m_\bB)\, q & \quad \mbox{ if }\; \Ss \; \mbox{ descends.}
\end{align}

\begin{lemma}\label{lem:linrel2}  {\rm (i)}  The identities in \eqref{eq:linrel1} may be rewritten as
\begin{align}\label{eq:linrel2}
m_\bB m& = d_\bB \, d -q_\bB\, p & \quad \mbox{ if }\; \Ss \; \mbox{ ascends}\\ \label{eq:linrel3}
m_\bB m& =  d_\bB \, d - p_\bB\, q& \quad \mbox{ if }\; \Ss \; \mbox{ descends.}
\end{align}

\NI {\rm (ii)}   Given any two classes $\bE', \bE''$,
 \eqref{eq:linrel2} holds for 
the pair $(\bE, \bB) = (\bE', \bE'')$  if and only if  \eqref{eq:linrel3} holds for 
the pair $(\bE, \bB) = (\bE'', \bE')$. 

\MS

\NI {\rm (iii)}  
 \eqref{eq:linrel2} holds for  the pair $(\bE, \bB)$ if and only if
\eqref{eq:linrel3} holds for the pair
$\bigl(S^\sharp(\bE), \bB\bigr)$.  Further
\eqref{eq:linrel3} holds for  the pair $(\bE, \bB)$ if and only if
\eqref{eq:linrel2} holds for the pair
 $\bigl(\bE, S^\sharp(\bB)\bigr)$.
 \MS
 
\NI {\rm (iv)}    If both \eqref{eq:linrel2} and \eqref{eq:linrel3} hold for $(\bE, \bB)$ then they both hold for
$\bigl(S^\sharp(\bE), S^\sharp(\bB)\bigr)$.
 \end{lemma}
 
 \begin{proof} To prove (i), rewrite the term $ 3m_\bB d$ on the left hand side as $m_\bB(p+q+m)$ and simplify.
 \MS
 
 The proof of (ii) is also straightforward:  if  $\bE': = (d',m',p',q'), \bE'': = (d'',m'',p'',q'')$, then both equations
 reduce to the claim that $m'm'' = d'd''- p'q''$.
\MS
  
 Now consider (iii).  Let $\bE = (d,m,p,q,t,\eps)$, so that   
 \begin{align*} S^\sharp(\bE)& = \bigl(d',m',(6p-q),p,t,-\eps\bigr),\quad\mbox{  where}\\
 d' &= \tfrac 18\bigl(3(7p-q) -\eps t\bigr),\quad  m' = \tfrac 18\bigl((7p-q) -3\eps t\bigr).
\end{align*}
To prove the claim about the pair $\bigl(S^\sharp(\bE), \bB\bigr)$, we want to show  that  the equation $m_\bB (p+q +3t\eps)= d_\bB \, (3(p+q) +t\eps) -8q_\bB\, p$ holds if and only if
 $$
 m_\bB \bigl((7p-q) -3\eps t\bigr) = d_\bB \bigl(3(7p-q) -\eps t\bigr) - 8q_\bB p
 $$
 also holds.  But
by adding the two  equations, we obtain the linear Diophantine identity
 $$
 m_\bB 8p =  d_\bB 24p  - (q_\bB + p_\bB)8p
 $$
which always holds.   
 This proves the first claim in (iii).  
 
 The second claim in (iii) now follows from the symmetry relation in (ii).
 We have
 \begin{align*}  
 \eqref{eq:linrel3} \mbox{ holds for } (\bE, \bB)\;& \Longleftrightarrow \;
 \eqref{eq:linrel2} \mbox{ holds for } (\bB, \bE)\;\quad\mbox{ by (ii)}\\
&\Longleftrightarrow
\eqref{eq:linrel3} \mbox{ holds for } \bigl(S^\sharp(\bB), \bE\bigr) \\
& \Longleftrightarrow \; \eqref{eq:linrel2} \mbox{ holds for } \bigl(\bE,S^\sharp(\bB)\bigr)\quad\mbox{ by (ii)}.
\end{align*}
Thus (iii) holds, and a similar argument proves (iv).
  \end{proof}
 
 \begin{prop}\label{prop:Gg}  Theorem~\ref{thm:Gg} holds.
 \end{prop}
  
 \begin{proof}  
 We must show that for each $T\in \Gg$ the seeds and pre-blocking classes in $T^\sharp(\Ss^U)$  form a pre-staircase family in the sense of Definition~\ref{def:prestf}.
Here, we define the pre-blocking classes and seeds of $T^\sharp(\Ss^U)$ to be the images of
 the pre-blocking classes and seeds in $\Ss^U$ by the formula in \eqref{eq:TGg}.  Lemma~\ref{lem:diophT}
 shows that these are quasi-perfect classes.
 Further, Lemma~\ref{lem:seeds} shows that the lower and upper seeds of both families
 $(S^{i+1})^\sharp(\Ss^U)$ and $(S^iR)^\sharp(\Ss^U)$ are $\bE_i, - \bE_{i+1}$.

Therefore, it remains to check that each 
 pre-blocking class $\bB_n$ is associated both to the ascending pre-staircase with seeds $\bE_{\ell,seed}, \bB_{n-1}$ and recursion parameter $t_{\bB_n}$,
 and to the descending pre-staircase with seeds $\bE_{u,seed}, \bB_{n+1}$ and recursion parameter $t_{\bB_n}$.  Thus each pre-staircase must satisfy the  appropriate linear relation~\eqref{eq:linrel}.
Because the steps in the staircases of a staircase family $\Ff$ with seeds $\bE^\Ff_{\ell,seed}$, $\bE^\Ff_{u,seed}$ and pre-blocking classes $\bB^\Ff_n$ are defined recursively by Corollary~\ref{cor:recur0}, it  suffices to check this 
for the pairs
\begin{align*}
(\bE,\bB)& =  (\bE^\Ff_{\ell,seed}, \bB^\Ff_n), (\bB^\Ff_{n-1},\bB^\Ff_n), \quad n\ge 1 \quad \mbox{ in the ascending case},\\
&=  (\bE^\Ff_{u,seed}, \bB^\Ff_n), (\bB^\Ff_{n+1},\bB^\Ff_n), \quad n\ge 0 \quad \mbox{ in the descending case}.
\end{align*}

We saw in Corollaries~\ref{cor:Uu},~\ref{cor:Ll}. that these relations do hold when $\Ff= \Ss^U, \Ss^L$.  
Thus it suffices to show that if $\Ff$ satisfies these identities, then its  image $S^\sharp(\Ff)$ does as well.  
To this end, we denote  seeds  of $\Ff$ by $\bE^\Ff_{\ell,seed}, \bE^\Ff_{u,seed} = -S^\sharp(\bE^\Ff_{\ell,seed})$ and its pre-blocking classes by $\bB_n^\Ff,n\ge 0.$  Then we know that for all $n$
\begin{align*}
(\bE^\Ff_{\ell,seed},\bB_n^\Ff), (\bB_{n-1}^\Ff,\bB_n^\Ff)\;\;& \mbox{  satisfy }\;\; \eqref{eq:linrel2},\\ 
(\bE^\Ff_{u,seed},\bB_n^\Ff), (\bB_{n+1}^\Ff,\bB_n^\Ff)\;\;& \mbox{  satisfy }\;\; \eqref{eq:linrel3}.
\end{align*}
In particular,  adjacent blocking classes $\bB^\Ff_n, \bB^\Ff_{n+1}$ satisfy both relations 
\eqref{eq:linrel2} and \eqref{eq:linrel3}.  Hence  Lemma~\ref{lem:linrel2}~(iv) shows that their images by 
$S^\sharp$ also satisfy both relations.  
Further, 
since $\bE_\ell: = \bE^{S^\sharp(\Ff)}_{\ell,seed}= - \bE^\Ff_{u,seed}$,  Lemma~\ref{lem:linrel2}~(iii) shows that
the pair
$(\bE_\ell, S^\sharp(\bB_n^\Ff))$ satisfies \eqref{eq:linrel2} because $(\bB_n^\Ff, \bE^\Ff_{u,seed})$ satisfies
 \eqref{eq:linrel3}.  In turn, this implies  that
 $\bigl(S^\sharp(\bB_n^\Ff), S^{\sharp}(\bE_\ell)\bigr) = \bigl(S^\sharp(\bB_n^\Ff), \bE^{S^\sharp(\Ff)}_{u,seed}\bigr) $ 
 satisfies  \eqref{eq:linrel3}.  
 
 Finally note that claims (i) and (ii) in the theorem follow immediately from the known behavior of $\Ss^U, \Ss^L$ and the fact that $S(v_i) = v_{i+1}, S(w_i) = w_{i+1}$.  
 This completes the proof.
 \end{proof}

 \begin{rmk}\label{rmk:E} 
 \rm (i) As the elements $T \in \Gg$ bring the seeds to seeds and the blocking classes to blocking classes by formula \ref{eq:TGg}, Lemma~\ref{lem:diophT}~(ii) implies composition behaves nicely on the staircase families, namely $T_1^\sharp T_2^\sharp(\Ss^U)=(T_1T_2)^\sharp(\Ss^U)$. This is simply because the formulas for $T^\sharp(p,q,t)$ are
  compatible with composition, and the degree coordinates $(d,m)$ are determined by the values of $p,q$. We will see in Section~\ref{ss:symmact} that while $T^\sharp$ acts linearly on $(p,q)$ (namely, via products of $S$ and $R$), there is no general linear map on the degree coordinates $(d,m)$ that respects composition. \MS
 
  \NI \rm  (ii)  The paper \cite{ICERM} established the existence of a third set of staircases, there called $\Ss^E$, and showed that the centers of its blocking classes and staircase steps are the images via $S$ of the centers of the corresponding classes in $\Ss^U$.   The interested reader can check that, just as in the proofs of Propositions~\ref{prop:Uu},~\ref{prop:Ll}, $\Ss^E$ forms a staircase family with seeds 
  $S^\sharp(\bE^U_{\ell,seed}), S^\sharp(\bE^U_{u,seed})$ and blocking classes $S^\sharp(\bE^U_n)$.  Thus
 $\Ss^E = S^\sharp(\Ss^U)$, and so
 has the same seeds as $\Ss^L$ by Lemma~\ref{lem:Ract}.
 
 Now the $E$-family has an ascending family of blocking classes  in the interval $(w_2,v_2) = (\frac{41}7, 6)$, while the $L$-family has a descending family of blocking classes  in the interval $(v_2,w_1)$.  We showed in \cite[Cor.60]{ICERM} that the centers of the blocking classes and steps in $\Ss^E$ are mapped to those in $\Ss^L$ by the reflection $
R_{v_2} = RS^{-1}=SR$  that fixes $v_2$ and interchanges $w_2=\frac{41}7$ and $w_1= 7$. 
Moreover, by Lemma~\ref{lem:seeds}, $R_{v_2}^\sharp$ interchanges the seeds of these two staircase families $\Ss^E$ and $\Ss^L$.    We show in Lemma~\ref{lem:degreflect} below that the action of $R_{v_2}^\sharp$ on the 
 $(d,m,p,q)$ components of those blocking classes of $\Ss^E, \Ss^L$ with centers in $(w_2,w_1)$ decomposes as the product $(R_{v_2})_\bB^*\times R_{v_2}: \Z^2\times \Z^2\to 
 \Z^2\times \Z^2$ of two matrices of order $2$.  This is a rather special situation that, we explain in  the discussion before Proposition~\ref{prop:degree}, does have a nice geometric interpretation.\hfill$\er$
 \end{rmk}

\begin{rmk}\label{rmk:seeds}\rm  The sequence $g_i$ is recursively generated by $g_{i+1}=6g_i-g_{i-1}$ giving us terms in the sequence as $i$ increases. Rewriting this as  
$g_{i-1}=6g_i-g_{i+1}$, we can solve for terms in the sequence as $i$ decreases. In particular, from the data in Table~\ref{table:gi}, we can solve for terms in this sequence with negative indices. Geometrically, we want to consider the sequence $\{\tfrac{g_{i}}{g_{i-1}}\},$ the center of these classes.
Computing a few terms in these sequences, we get:
$$
\begin{array}{c|c|c|c|c|c|c|c|c}\hline
i&-4&-3&-2&-1&0&1&2&3\\ \hline
g_{i}&169 & 29 & 5 & 1 & 1 & 5 & 29 & 169\\ \hline
\tfrac{g_{i}}{g_{i-1}}&169/985&29/169 & 5/29 & 1/5 & 1 & 5 & 29/5 & 169/29 \\ \hline
\end{array}
$$
Thus the sequence $\{g_i\}$ reflects on itself, namely $g_i=g_{-(i+1)}$. Note, that both the centers of $\bE_i$ and $-\bE_i$ are elements of this sequence $\{\tfrac{g_{i}}{g_{i-1}}\}$ when
$i\ge 0$. 
In particular, only terms with $i \ge 0$ appear as centers of the seed classes.  

As $S$ implements the recursion by $6$, $S(g_{i}/g_{i-1})=g_{i+1}/g_i$, so $S$ always shifts this sequence one step to the right.  Now, $R$ is the reflection of this sequence about $5=g_{1}/g_0$; indeed 
we can compute that $R(g_{i}/g_{i-1})=g_{2-i}/g_{1-i}$.  (Note: $R$ is not the reflection about $1$ which would map each element to its reciprocal.)  Hence, $R$ takes centers of seeds to terms in the sequence, but not necessarily to valid centers, since  $g_{2-i}/g_{1-i} < 1$  when $i>2$. But, applying $S$ will shift $g_{2-i}/g_{1-i}$ to the right.  In particular, the reflection $R_{v_i}=S^{2i-3}R$ has just the right number of shifts to move the image under $R$ of the  seeds of $S^{i-1}(\Ss^U)$ to the seeds of $S^{i-2}(\Ss^L)$.   
In other words, this reflection interchanges the lower and upper seeds of these families. As we show in Lemma~\ref{lem:degreflect} it also interchanges their blocking classes. 
  \hfill$\er$
 \end{rmk}

 \subsection{Action of the symmetries on blocking class degree}\label{ss:symmact}

 Although the results in this section are not needed for  the proof of our main results,
 they throw light on the geometric nature of the symmetries.   To simplify the language needed we will assume 
 already known that 
 all classes in $T^\sharp(\Ss^U)$ are perfect, a result that is proved in Proposition~\ref{prop:perf} below. 
 In particular, this means that all the pre-blocking classes in $T^\sharp(\Ss^U)$ are in fact blocking classes, and so we will  talk about blocking classes rather than pre-blocking classes.

The first step is to derive a formula for the action of $T$ on the degree coordinates $(d,m)$ of the blocking classes in $\Ss^U$.   
Because the formula \eqref{eq:dmdioph}  for $d,m$ in terms of $p,q$ is affine, we might expect this action to be affine
and to depend on the $t$-variable (or equivalently on $n$).  However, it turns out to be linear and independent of $t,n$ 
 and in fact is given by a $2\times 2$  matrix (with integer entries) that we denote $T_\bB^*$. Note, $T_\bB^*$ 
 gives 
 the action on the degree coordinates of blocking classes, and in particular does not describe the action on
the degrees of the staircase steps.
For convenience,   we denote the vector with components $d,m$ as  $(d,m)^{\vee}$.

\begin{lemma}\label{lem:degree} {\rm (i)} For each $T = S^iR^\de\in \Gg$ there is a $2\times 2$ matrix  $T^*_{\bB}$ such that for each blocking class $\bB^U_n
= (d_n,m_n,p_n,q_n, t,1), n\ge 0,$ in  $\Ss^U$, the corresponding blocking class in $T^\sharp(\bB^U_n)$ has degree coefficients $T^*_{\bB}\bigl((d_n,m_n)^{\vee}\bigr)$.\MS

\NI {\rm (ii)}  In all cases the matrix $T^*_{\bB}$ has eigenvector $(3,1)^{\vee}$
with eigenvalue $(-1)^{i+\de} \det(T^*_{\bB})$. 
\end{lemma}
\begin{proof}
Denote $T^\sharp(d,m,p,q,t,1) =  (d',m',p',q',t,\eps')$, where $\bB^U_n = (d,m,p,q,t,1)$.
By \eqref{eq:dmdioph}, 
 $3m'-d' = t\eps'$ where $t$ is constant under $T^\sharp$ and
 \begin{align}\label{eq:TdmB}
\eps' t= (-1)^{i+\de}t = (-1)^{i+\de}(3m-d) = (-1)^{i+\de}(p-3q),
\end{align}
where the last equality holds because $\bB^U_n=(n+3,n+2,2n+6,1,2n+3,1)$.
Therefore,  we have
\begin{align}\label{eq:Tdm'}
d' &= \tfrac18\bigl(3(p'+q') + \eps' t\bigr) 
= \tfrac18\bigl(3(p'+q') + (-1)^{i+\de}(p-3q)\bigr),\\ \notag
m' &= \tfrac18\bigl((p'+q') + 3\eps' t\bigr)  
= \tfrac18\bigl((p'+q') + 3(-1)^{i+\de}(p-3q)\bigr).
\end{align}
Since $p',q'$  are linear functions of $p,q$, so are $d',m'$.  This proves (i).
\MS

To prove (ii), observe that a $2\times 2$ matrix $A=\left(\begin{array}{cc} a&b\\c&d\end{array}\right)$ 
with nonzero eigenvalues $\la_1,\la_2$ has eigenvector $(3,1)^{\vee}$ with eigenvalue $\la_1$ if and only if its entries satisfy the equation  $3a+b = 3(3c+d)$, which is the same as the condition for the transpose matrix $A^T$ to have   eigenvector $(1,-3)^{\vee}$. One can check that in this situation the corresponding eigenvalue of $A^T$ is $\la_2$, which means that the transformation $(x,y)^{\vee}\mapsto A(x,y)^{\vee}$ preserves the linear function $x-3y$  modulo the factor $\la_2$.  In our case we know that the function $d-3m$ is invariant by $T^*_{\bB}$ modulo the sign factor $(-1)^{i+\de}$.
Hence the other eigenvalue is $(-1)^{i+\de}\det(T^*_{\bB})$, with corresponding eigenvector $(3,1)^{\vee}$.
\end{proof}

\begin{rmk}\label{rmk:degree}\rm  It would be more correct (though also more cumbersome) to denote the map $T_\bB^*$ by
$T_{U,\bB}^*$ since its formula depends on the  domain staircase $\Ss^U$ 
via the identity $3m-d=p-3q$ in \eqref{eq:TdmB} used to express $\eps't$ as a function of $p,q$.
Each staircase family has an analogous, but different, identity of this kind.  
For example, the blocking classes $(5n,n-1,12n+1,2n)$ satisfy $3m-d = -3p + 17q$.  
Correspondingly, the assignment $T\to T^*_\bB$ is not compatible with composition in general, though it is in a few special cases.
 For example, even though the reflection $R$ in $w_1 = 7$ has order $2$, $R_\bB^*$ does not have determinant $\pm 1$ and $(R_\bB^*)^2 \ne {\rm id}$; see Proposition~\ref{prop:degree}~(ii). Further, we show in Example~\ref{ex:blockS} that $(S^4)_\bB^* \ne \bigl((S^2)_\bB^*\bigr)^2$.
\hfill$\er$
\end{rmk}

We now show that the formula for $T^*_\bB$ does have understandable features.
 Observe that 
the blocking classes $\bB^U_n$ have centers at $(p_n,q_n) = (2n+6,1)$ and degree components $(d,m)$ 
$$
(d_n,m_n) = (3+n, 2+n) = (3,2) + (1,1) n,\qquad \lim_{n\to \infty} \frac{m_n}{d_n}  = 1.  
$$
This is no accident:  we should expect  the limits of
the  sequences $\bigl(\frac{m_n}{d_n}\bigr)$ and $\bigl(\frac{p_n}{q_n}\bigr)$ to correspond via the function  
$\acc_U^{-1}$.  However, this is a degenerate case  since both limits lie on the boundary of  the  domain of this function.
Let us apply the same analysis to the staircase families $\Ss^L=R^\sharp(\Ss^U)$ and $\Ss^E = S^\sharp(\Ss^U)$.  As illustrated in
Fig.~\ref{fig:symm} and Fig.~\ref{fig:reflect} both these families have steps in the interval $(w_2,w_1) = (41/7, 7)$.
 Further, the centers of the blocking classes in $\Ss^E$ increase with limit $v_2 = 6$, while those of $\Ss^L$ decrease, also with limit $v_2$.  Thus, in both cases, the ratios $\bigl(\frac{m_n}{d_n}\bigr)$ of the degree coefficients of the blocking classes  converge to $\acc_L^{-1}(6) = 1/5$.  
 Correspondingly, we show below that in both  cases $T = R, S$
 the matrix $T_\bB^*$ takes the vector $(1,1)^{\vee}$ that gives the coefficient of $n$ in $(d,m)$   to $(5,1)^{\vee}$.
 This observation generalizes as follows.

\begin{prop}\label{prop:degree} \NI {\rm (i)}  The matrix $(S^i)^*_{\bB}$ has integer entries and is determined by the following properties:
\begin{itemize}\item[-]
$(S^i)^*_{\bB}\bigl((3,2)^{\vee}\bigr) = \frac18\bigl(3(y_{i+2} + y_{i+1}) + 3(-1)^{i},\
y_{i+2} + y_{i+1} + 9(-1)^{i})^{\vee}\bigr)$;

\item[-]   $(S^i)^*_{\bB}\bigl((1,1)^{\vee}\bigr) = \bigl((s_i,r_i)^{\vee}\bigr)$, where 
 \begin{align}\label{eq:Tdm} 
 s_i  = \tfrac 14\bigl(3y_{i+1}+3y_{i} + (-1)^{i}\bigr),\quad
  r_i &= \tfrac 14 \bigl(y_{i+1}+y_{i} + 3(-1)^{i}\bigr).
 \end{align}
 In particular $\frac{r_i}{s_i} = \acc_{\eps}^{-1}(v_{i+1})$, where $\eps = (-1)^{i}$ and we define $$
\acc_{\eps}^{-1}: = \acc_U^{-1}, \quad \mbox{ if } \eps = +1, \quad \acc_{\eps}^{-1}: = \acc_L^{-1}, \quad \mbox{ if } \eps = -1.
$$
\end{itemize}
Further $\det((S^i)^*_{\bB}) = (-1)^i(y_{i+1}-y_{i})$. 

\NI {\rm (ii)}  When $T=R$ we have $R^*_{\bB} = \left(\begin{array}{cc} -10&15\\-3 &4   \end{array}\right)$. \MS

\NI {\rm (iii)} The matrix $(S^{i}R)^*_{\bB}$ has integer entries and is determined by the following properties:
\begin{itemize}\item[-]
$(S^iR)^*_{\bB}\bigl((3,2)^{\vee}\bigr) = \frac18\bigl(3(y_{i+1} + y_{i}) - 3(-1)^{i},\
y_{i+1} + y_{i} - 9(-1)^{i})^{\vee}\bigr)$;

\item[-]   $(S^{i}R)^*_{\bB}\bigl((1,1)^{\vee}\bigr) = \bigl((s_{i+1},r_{i+1})^{\vee}\bigr)$, where 
$s_{i+1}, r_{i+1}$ are as in \eqref{eq:Tdm}.
\end{itemize}
Further $\det((S^iR)^*_{\bB}) = (-1)^i(y_{i+2}-y_{i+1})$. 
\end{prop}
 
\begin{proof}  The matrix $(S^i)^*_\bB$ is obviously determined by the images of the vectors $(3,2)^{\vee}$ and $(1,1)^{\vee}$, and we first check that these images are as stated.   Recall the sequence
$$
y_0,y_1,y_2, y_3\dots = 0,1,6,35,\dots, \quad\mbox{ where }\  S(y_i,y_{i-1}) = (y_{i+1},y_i),
$$
and
write $$
(p_n, q_n)^{\vee} = (6,1)^{\vee} + 2n(1,0)^{\vee} = (y_2, y_1)^{\vee}  + 2n(y_1, y_0)^{\vee} .
$$
Then $$
S^i\bigl((p_n, q_n)^{\vee}\bigr) = (y_{i+2}, y_{i+1})^{\vee}  + 2n(y_{i+1}, y_i)^{\vee}
$$
 and because
 $t=2n+3$ is fixed by $S$,  we can derive the formula for
 $(S^i)^*_\bB\bigl((3,2)^{\vee})\bigr)$ and  $(S^i)^*_\bB\bigl((1,1)^{\vee})\bigr)$ by looking at the constant term and coefficient of $n$ in the equation
 \eqref{eq:TGg}.
 Thus, if $X$ is the matrix with columns $(3,2)^{\vee}, (1,1)^{\vee}$, we have
 $(S^i)^*_\bB\, X = \frac 1{32}A$, where
 $$
   A =\left(\begin{array}{cc} 3w_{i+1} + 3\eps& 3w_i+ \eps\\ w_{i+1} + 9\eps& w_i + 3\eps\end{array}\right),\quad w_i = y_{i+1}+y_i, \;\; \eps: = (-1)^i.
  $$
It follows from Lemma~\ref{lem:diophT} that the entries of  $(S^i)^*_\bB\, X$ are integers.  
 Hence because
  the matrix $X$ with columns $(3,2)^{\vee}, (1,1)^{\vee}$ has determinant $+1$,  the entries of $(S^i)^*_\bB$ are also integers.
  Further,
  \begin{align*}
  \det A & = \eps\bigl(9w_{i+1}+3w_i - 27 w_i - w_{i+1})\\
  &  = \eps(8(y_{i+2}+y_{i+1}) - 24(y_{i+1}+y_{i}) = 32 \eps (y_{i+1}-y_i),
\end{align*}
  where the last equality holds because $y_{i+1} = 6 y_i - y_{i-1}$.  
Therefore  $\det\bigl((S^i)^*_\bB\bigr)$ is a claimed.  This proves (i).
\MS

Claim (ii) follows from the fact that $R^*_\bB$ takes the degree components  $(4,3), (5,4)$ of $\bB^U_1, \bB^U_2$ to 
the corresponding components $(5,0), (10,1)$  of $\bB^L_1,\bB^L_2$; see \eqref{eq:Llblock}.  Note that
$$
R^*_{\bB}\left(\begin{array}{cc} 3&2\\1 &1   \end{array}\right)  = \left(\begin{array}{cc} -10&15\\-3 &4   \end{array}\right)\left(\begin{array}{cc} 3&2\\1 &1   \end{array}\right) = \left(\begin{array}{cc}0 &5\\-1 &1  \end{array}\right)
$$

Finally, (iii) follows by arguing as in (i), noting that $R$ interchanges the pairs $(6,1) = (y_2,y_1)$ and $(1,0) = (y_1,y_0)$.  
Therefore $$
S^iR\bigl((6,1)^{\vee}\bigr))=  \bigl((y_{i+1},y_i)^{\vee}\bigr), \quad S^iR\bigl((1,0)^{\vee}\bigr)
=  \bigl((y_{i+2},y_{i+1})^{\vee}\bigr),
$$
so that, as before,
 \eqref{eq:TGg} implies that the action on the corresponding degree coordinates $(3,2)^{\vee}, (1,1)^{\vee}$ is
\begin{align*}
(S^iR)^*_\bB\bigl((3,2)^{\vee}\bigr) &=\tfrac18\bigl((3(y_{i+1} + y_{i}) + 3(-1)^{i+1},\
y_{i+1} + y_{i} + 9(-1)^{i+1}))^{\vee}\bigr)\\
(S^iR)^*_\bB\bigl((1,1)^{\vee}\bigr) &=\tfrac14\bigl((3(y_{i+2} + y_{i+1}) + (-1)^{i+1},\
y_{i+2} + y_{i+1} + 3(-1)^{i+1}))^{\vee}\bigr).
\end{align*}
These formulas are  consistent with the fact that by Proposition~\ref{prop:Ll} the blocking classes $\bB^L_0 = R^\#(\bB^U_0), \bB^L_1 = R^\#(\bB^U_1),$ are $(0,-1,1,0)$, $(5,0,13,2)$. The determinant calculation can be checked as before.
This completes the proof.
\end{proof}

\begin{example}\label{ex:blockS}\rm  As examples of the above formulas we have:
\begin{align*}
S^*_\bB = \left(\begin{array}{cc} 5& 0\\ 2& -1\end{array}\right),\quad 
&(S^2)^*_\bB = \left(\begin{array}{cc} 28& 3\\ 9& 2\end{array}\right),\\
(S^3)^*_\bB = \left(\begin{array}{cc} 164& 15\\ 55& 4\end{array}\right),\quad 
&(S^4)^*_\bB = \left(\begin{array}{cc} 955& 90\\ 318& 31\end{array}\right)
\end{align*}
Note that the second column of these matrices  coincides 
with the degree components of a corresponding principal\footnote
{
See Lemma~\ref{lem:princ}; these are well defined for $i-2\ge 0$ and need appropriate interpretation when $i=1$.}     blocking class $(S^{i-2})^*(\bB^U_0)$.
However, these matrices do {\it not} give the action of $S^i$ on the seed classes, even when $i$ is even.  For example, the lower seeds of $\Ss^U$ and $(S^2)^\sharp(\Ss^U)$ are
$$
\bE^U_{\ell,seed}=(1,1,1,1,2,1), \qquad \bE^{(S^2)^\sharp(\Ss^U)}_{\ell,seed}=(13,5,29,5,2,1).
$$
Because these matrices $(S^i)^*_\bB$ do not give the action on the seeds, they also do  not act on the degrees of the  staircase steps.   

We could compute similar matrices $(T)^*_{\bullet,seed}$ that would take the degrees of $\bE^U_{\bullet,seed}$ to the degrees of $T^\sharp(\bE^U_{\bullet,seed})$. 
However, as above, $(T)^*_{\bullet,seed}$ would also not respect composition because there is no analog of the identity 
 $t=3m-d = ap + bq$ in \eqref{eq:TdmB} for the
seeds. (This is easy to check using the fact that $t=2$ by \eqref{eq:dm13}.)
Thus we only present the  formulas for the degree coordinates generated by the recursion $S$ by \eqref{eq:seed0}; also see Lemma~\ref{lem:seeds} and Remark~\ref{rmk:seeds}. 
\hfill$\er$
\end{example}

By  Lemma~\ref{lem:Phi0} and Corollary~\ref{cor:ident}  we can also write 
 $$
 S^{i-1} R = (S^{i-1} R S^{-i})  S^{i} = R_{v_{i+1}}\ S^{i},
 $$
where $R_{v_{i+1}}$ is a reflection that fixes the point $v_{i+1}$ and interchanges the steps of the two staircase families
$(S^i)^\sharp(\Ss^U)$ and $(S^{i-1}R)^\sharp(\Ss^U)$  with steps in the short interval $(w_{i+2},w_{i+1})$ around $v_{i+1}$; see Fig~\ref{fig:symm}.  
  Let us denote by $(R_{v_{i+1}})^*_\bB$ the matrix that takes the blocking class degrees in
 $(S^{i-1}R)^\sharp(\Ss^U)$ to those of $(S^i)^\sharp(\Ss^U)$; see Fig~\ref{fig:reflect}. 
 Note that to be consistent with our interpretation of $R$ that takes $\Ss^U$ to $\Ss^L$ we take the $p/q$ coordinates of the domain staircase family of this reflection to lie {\it above} the fixed point.\footnote
 {
 This choice is relevant because in general the degree component $T^*_\bB$  of a reflection $T^\sharp$ does not have order two: see Lemma~\ref{lem:princ}.} 
 Further, when $i$ is even (resp. odd), $(R_{v_{i+1}})^*_\bB$ acts on blocking classes of staircases with $b_\infty>1/3$ (resp. $b_\infty<1/3$).  For example, 
  $(R_{v_{2}})^*_\bB$ takes the blocking classes of $\Ss^L$ to those of $\Ss^E=S(\Ss^U))$.

  \begin{lemma}\label{lem:degreflect}  The matrix 
  $$
  (R_{v_{i+1}})^*_\bB: = \bigl((S^{i})_\bB^*)\bigr)\circ (S^{i-1}R)_\bB^*\bigr)^{-1}
  $$ 
  has eigenvectors $(s_i,r_i)^{\vee}, (3,1)^{\vee}$ with corresponding eigenvalues $1,-1$.  Thus it has order two, and hence also takes the degree components of the blocking classes in $(S^{i})^\sharp(\Ss^U)$ to those of $(S^{i-1}R)^\sharp(\Ss^U)$. 
  \end{lemma}
  \begin{proof}    Proposition~\ref{prop:degree} shows that $$
  (S^{i-1}R)^*_\bB\bigl((1,1)^{\vee}\bigr) = (s_i,r_i) =   (S^{i})^*_\bB\bigl((1,1)^{\vee}\bigr).
  $$
  and that $\det( (S^{i-1}R)^*) = -\det ((S^{i})^*)$.  The result now follows from Lemma~\ref{lem:degree}~(ii).
  \end{proof}

There turns out to be a second natural action of the reflections $R_{v_{j}}, j\ge 0,$ on blocking class degree corresponding to its action on the other component of $[0,1/3)\cup (1/3,\infty)$.   
Indeed, as illustrated in Figure~\ref{fig:reflect} and Remark~\ref{rmk:Gg}~(i), this reflection acts in two ways on our staircases depending on whether one takes $b>1/3$ or $b<1/3$.  The action discussed above (with $j = i+2\ge 2$) interchanges  the steps of the two staircase families $(S^{i+1})^\sharp(\Ss^U)$ and 
$(S^{i}R)^\sharp(\Ss^U) = (S^{i})^\sharp(\Ss^L)$ the centers of whose blocking classes converge to
$v_{i+2}$, while, as we will see,  the second action of $R_{v_{i+2}}$
 fixes the center $v_{i+2}$ of the  blocking class 
$(S^{i})^\sharp(\bB^U_0)$ and  takes the ascending blocking classes in $(S^{i})^\sharp(\Ss^U)$ with centers in $[v_{i+3}, w_{i+2}]$
to the descending blocking classes in $(S^{i+1}R)^\sharp(\Ss^U) = (S^{i+1})^\sharp(\Ss^L)$
with centers in $[w_{i+1}, v_{i+1}]$.  Note, when $i$ is odd (resp. even), the principal blocking class $(S^i)^\sharp(\bB^U_0)$ blocks a $b$-region with $b<1/3$ (resp. $b>1/3$).

 We will denote the matrix that gives this second action on degree by $(R_{v_{i+1}})^*_\bP$.  It is again defined by its action on relevant blocking classes, but the family used is different  than before.
 
  Notice that the matrix $(R_{v_{i+1}})^*_\bP$ obtained this way does not have order two, so that it matters that we define it via~\eqref{eq:TGg} applied to its action on
 the blocking classes in $(S^{i})^\sharp(\Ss^U)$.
 Besides the eigenvector 
$(3,1)^{\vee}$, its  second eigenvector (with eigenvalue $1$) must be given by the  degree components of the class
$(S^{i})^\sharp(\bB^U_0)$.  (We do need to check that all these conditions are compatible.)
Note also that  the resulting matrix $(R_{v_{i+2}})^*_\bP$ is not in general  integral.

\begin{figure}\label{fig:reflect}
\vspace{-1 in}
\centerline{\includegraphics[width=7in]{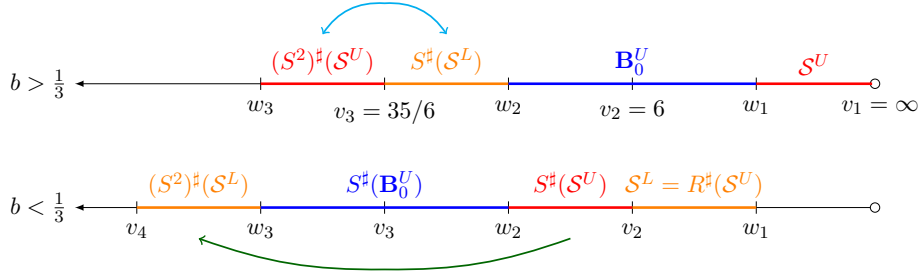}}
\vspace{-6.5 in}
\caption{Here we illustrate the different actions of the reflection $R_{v_3}$ that acts on the $z$ axis by fixing
fixing $v_3$.
The reflection $(R_{v_3})^\sharp$, represented by the light blue arrow, interchanges the upper and lower staircase families of $(S^2)^\sharp(\Ss^U)$ and $S^\sharp(\Ss^L)$, and the corresponding matrix $(R_{v_3})^*_{\bB}$ acts as a reflection on the degrees of the blocking classes associated to these staircases.
On the other hand, on the lower number line, we have the principal blocking class $S^\sharp(\bB^U_0)$ with center $v_3$ blocking the corresponding blue interval with $b<1/3$. On either side of this blocked region live the centers of the blocking classes of $(S^2)^\sharp(\Ss^L)$ and $S^\sharp(\Ss^U)$. The matrix $(R_{v_3})^*_{\bP}$ maps the degree coordinates of $S^\sharp(\Ss^U)$ to those of $(S^2)^\sharp(\Ss^L)$. This is represented by the green arrow. Note: this is not a reflection, so the arrow only goes in one direction.}
\end{figure}

We call the blocking classes $(S^{i})^\sharp(\bB^U_0), i\ge 0,$ {\bf principal blocking classes}.
The next result spells out their main properties.

\begin{lemma}\label{lem:princ}  {\rm (i)}  The principal  blocking class 
$(S^{i})^\sharp(\bB^U_0)$ has  components
$$
\bigl(\tfrac 38(y_{i+2}+y_{i+1}+ (-1)^{i})  , \tfrac 18(y_{i+2}+y_{i+1}+ 9(-1)^{i}), y_{i+2},y_{i+1},3, (-1)^i\bigr).
$$

\NI {\rm (ii)}   The  transformation $(R_{v_{i+2}})_\bP^\sharp: (S^{i})^\sharp(\Ss^U) \to (S^{i+1})^\sharp(\Ss^L)$ 
acts on the degrees of the blocking classes  by the matrix $$
(R_{v_{i+2}})^*_\bP: = (S^{i+1}R)^*_\bB\ \bigl((S^{i})^*_\bB)\bigr)^{-1}
$$ 
with determinant $\frac{y_{i+3}-y_{i+2}}{y_{i+2}-y_{i+1}}$.  Its eigenvalues are $\det\bigl((R_{v_{i+2}})^*_\bP\bigr)$ and $ 1$ with corresponding eigenvectors given by  $(3,1)^{\vee}$ and the degree components of $(S^{i})^\sharp(\bB^U_0)$. 
\end{lemma}

\begin{proof} The formula for the degree components of the principal blocking class $\bB: = (S^{i})^\sharp(\bB^U_0)$ follows from \eqref{eq:TGg}.   
The claim in   (ii)  follows immediately from Proposition~\ref{prop:degree}.
\end{proof}

\begin{example}\label{ex:v2}\rm The reflection $R_{v_2}$ that fixes $6$ has two extensions to an action on blocking classes:
\begin{align}
(R_{v_2})^*_\bB = \left(\begin{array}{cc} 4 & -15\\ 1& -4\end{array}\right),\qquad (R_{v_2})^*_\bP= \left(\begin{array}{cc} -59& 90\\-20& 31\end{array}\right).
\end{align}
Here $(R_{v_2})^*_\bB$, with determinant $-1$, interchanges the blocking classes of $\Ss^E = S(\Ss^U)$ with those of  $\Ss^L$. Both of these familes have $b_\infty<1/3$ and their accumulation points limit to $v_2$. On the other hand, $(R_{v_2})^*_\bP$, with determinant $-29$, fixes the degree components $(3,2)$ of $\bB^U_0$ and takes the blocking classes of $\Ss^U$ to those of $(S^2R)^\sharp (\Ss^U)$.  Both of these staircases have $b_\infty>1/3$ and lie on different sides of the interval blocked by the principal blocking class $B^U_0$, see Fig~\ref{fig:reflect}. Thus $(R_{v_2})^*_\bB= (S)^*_\bB\circ ((R)^*_\bB)^{-1}$, and one can check that $(R_{v_2})^*_\bP = (SR)^*_\bB$.
\hfill$\er$
 \end{example}

\section{The pre-staircases are live}\label{sec:live}

This section completes the proof of Theorem~\ref{thm:Gglive}.  We first show that the pre-staircase families define in Section 3 are perfect pre-staircase families. Establishing that all staircase classes are perfect implies that
each such class $\bE_\ka = (d_\ka,m_\ka,p_\ka,q_\ka,t_\ka,\eps)$  is live at its center $p_\ka/q_\ka$ when 
$b = m_\ka/d_\ka$.    
Therefore, by \eqref{eq:obstr},  the capacity function $c_{H_{m_\ka/d_\ka}}$ takes the value $\frac {p_\ka d_\ka}{d^2_\ka - m^2_\ka}$  at the point $p_\ka/q_\ka$, and this implies by continuity\footnote
{
This holds because for any quasi-perfect class $\bE = (d,m,p,q)$ we have $\mu_{\bE,b}(p/q) \le V_b(p/q)\bigl(\sqrt{1+\frac1{d^2-m^2}}\bigr)$ and here $d_\ka\to \infty$; see \cite[Lem.15]{ICERM}.}
 that  the limiting value $b_\infty$  is unobstructed, i.e.  $c_{H_{b_\infty}}(a_\infty) =  V_{b_\infty}(a_\infty)$, where $a_\infty=\lim(p_\ka/q_\ka)$.

What we have to show is that, at least for sufficiently large $\ka$, the class $\bE_\ka$ remains live at the limiting $b$ value $b_\infty: = \lim m_k/d_k$.    In \cite{ICERM} we established this  in two steps, first showing that if the ratios $m_\ka/d_\ka$     satisfy a bound such as that in \eqref{eq:DMineq} below then, by the positivity of the intersections of exceptional classes,  the degree of any class $\bE$ such that 
$\mu_{\bE, b_\infty}(p_\ka/q_\ka) \ge \mu_{\bE_\ka, b_\infty}(p_\ka/q_\ka)$ is bounded above by a constant that is independent of $\ka$.  This means that there can be only finitely many such classes.  In particular, there must be one class $\bE_{ov}$  that  dominates infinitely many of the steps. This is possible only if  the obstruction $\mu_{\bE_{ov}, b_\infty}(z)$ 
given by this class  goes through the accumulation point $(a_\infty, V_{b_\infty}(a_\infty))$ of the 
staircase.\footnote{   
Indeed, we know that the accumulation point is unobstructed so that $\bE_{ov}$ cannot be obstructive at this point, and if (for a descending staircase) the obstruction $\mu_{\bE_{ov}, b_\infty}$ crossed the volume curve to the right of the accumulation point then it could at best overshadow only a finite number of steps.}
 We call such a class an {\bf overshadowing class} because its obstruction $\mu_{\bE_{ov}, b_\infty}(z)$ overshadows 
the staircase steps so that they cease to be visible at $b = b_\infty$.

In \cite{ICERM}  we were able to find rather good bounds for the degree of a potentially overshadowing class, and hence could show that they do not exist by a case by case analysis. 
 This method is not feasible here since the bounds on the degree of any overshadowing class of a staircase in the family $(S^iR^\de)^\sharp(\Ss^U)$ increase too rapidly with $i$.    However it turns out that we can exploit the fact that the obstruction $\mu_{\bE_{ov}, b_\infty}(z)$
goes through the accumulation point to obtain powerful arithmetic information about the degree components $d_{ov},m_{ov}$ of $\bE_{ov}$, which is enough to rule out the existence of such a class.  
This argument hinges on the results of Lemma~\ref{lem:volacc}, namely, that the two functions
\begin{align}\label{eq:same}
b\mapsto  V_b(\acc(b)) \quad\mbox{ and} \qquad   b\mapsto \frac{1+\acc(b)}{3-b}
\end{align}
are the same.

\subsection{The pre-staircase classes are perfect}
We first prove that all the classes in the families $T^\sharp(\Ss^U)$ are perfect, that is, they are exceptional classes,
and then use this fact in Corollary~\ref{cor:block}
to gain information about the $z$-intervals that are blocked by the blocking classes.

  We use the 
 following recognition principle, which is explained  for example in \cite[Prop.1.2.12]{ball}.

\begin{lemma}\label{lem:Cr1}  An integral class $\bE: = dL-\sum_{i=1}^N n_i E_i$ in the $N$ fold blowup $\C P^2\# N \ov{\C P}^2$ represents an exceptional divisor if and only if it may be reduced to $E_1$ by repeated application of  Cremona transformations.  \end{lemma}

Here, a {\bf Cremona transformation} is a composition of the transformation
\[c_{xyz}(dL-\sum_{i=1}n_iE_i)=(d+\delta_{xyz})L-\sum_{i\in \{x,y,z\}}(n_i+\delta_{x,y,z})E_i-\sum_{i \notin \{x,y,z\}} n_iE_i\]
where 
$\delta_{xyz}=d-n_x-n_y-n_z$
and a reordering operation. Writing $\bE$ in coordinates $(d;n_1,\hdots,n_N),$ $c_{xyz}$ adds $\delta_{xyz}$ to the coordinates $d,n_x,n_y,n_z$ and the reordering can reorder any of the $n_i$.
Because Cremona moves are reversible, to verify $\bE_k$ is exceptional, we just need to show it reduces to some other $\bE_j$ that we know to be exceptional. We say $\bE_k$ and $\bE_j$ are {\bf Cremona equivalent} if one can be reduced to the other.

Our staircase classes are quasi-perfect and hence have the form $(d,m, W(p/q))$ where $W(p/q) = (W_1,\dots,W_N)$ is the integral weight expansion of $p/q.$   We  denote such a  class by the tuple $(d,m,p,q)$, and consider the Cremona moves as acting on the corresponding sequence    $(d,m, W_1,W_2,W_3,\hdots)$.
Thus, for example,  we have
\begin{align*}& c_{012}((d,m,W_1,W_2,W_3,\hdots)) = \\
& \qquad\quad (2d-m-W_1
-W_2,m-(W_1+W_2),W_1-(m+W_2),W_2-(m+W_1),W_3,\hdots).
\end{align*}

Here is the key lemma.

\begin{lemma}\label{lem:CrS}
Suppose $(d,m,p,q)$ satisfies  $3d=m+p+q$, and that $d,m$ are defined by \eqref{eq:dmdioph}. 
Then, $(d,m,p,q)$ is Cremona equivalent to $S^{\sharp}(d,m,p,q)$.
\end{lemma}
\begin{proof} 
Let $S^\sharp(d,m,p,q)=(D,M,P,Q)$. By Lemma~\ref{lem:symmCF}, when $p/q=[5+k,CF(x)]$ we have $S(p/q)=[5;1,4+k,CF(x)]$. Thus, the integral weight expansion of $S(p/q)$ will always have five more integers than the integral weight expansion of $p/q.$ 
As $S(p/q)=P/Q=(6p-q)/p,$ we have that
\begin{align*}
    W(P/Q)&=W((6p-q)/p)=(p^{\times 5},p-q) \sqcup W((p-q)/q)\\
    &=(Q^{\times 5},P-5Q) \sqcup W((p-q)/q)
\end{align*}
where $W((p-q)/q)$ by definition is $W(p/q)$ with the first entry $q$ removed. 
Thus, to reduce the class $(D,M,P,Q)$ to the class $(d,m,p,q)$, it suffices to  show that we can reduce
\[(D,M,Q^{\times 5},P-5Q) \sqcup W((p-q)/q) \quad \text{to} \quad 
(d,m,q) \sqcup W((p-q)/q).\]
Note that in this reduction we must get rid of five terms because as mentioned 
$W(S(p/q))$ has five more terms than $W(p/q).$

Next, observe that
\begin{align*}& c_{256}c_{234}c_{016}c_{345}c_{012}(D;Q^{\times 5},M,P-5Q)=
\\ &\qquad\quad (8D-3(M+P);0^{\times 2},3D-M-2P+5Q,0^{\times 3},3D-2M-P).
\end{align*}
This can be seen by direct computation, where the zeros come from the linear Diophantine condition $3D-M-P-Q=0.$ The first three steps give the two zeros in positions 1 and 2, the fourth step results in the two zeros in positions 4 and 5, and the fifth step results in one zero in position 6.
Furthermore, since $P=6p-q$ and $Q=p$, it follows from \eqref{eq:dmdioph} that
$$
D=\tfrac{1}{8}(3(7p-q)+\varepsilon t), \qquad M=\tfrac{1}{8}((7p-q)+3\varepsilon t).
$$
Performing these substitutions, we get
\begin{align*}
 & 8D-3(M+P)\;=\;\tfrac{1}{8}(3(p+q)-\varepsilon t)=d\\
 & 3D-M-2P+5Q\;=\;q \\
 & 3D-2M-P\;=\; \tfrac{1}{8}(p+q-\varepsilon t)=m
\end{align*}
We conclude that
\[c_{367}c_{345}c_{127}c_{456}c_{123}(D;Q^{\times 5},M,P-5Q)=(d;0^{\times 2},q,0^{\times 3},m).\]
Thus these five Cremona moves and an appropriate reordering reduces
\[(D,M,Q^{\times 5},P-5Q) \sqcup W((p-q)/q) \quad \text{to} \quad 
(d,m,q) \sqcup W((p-q)/q).\]
Hence, the class $(D,M,P,Q)$ is Cremona equivalent to  $(d,m,p,q)$, as claimed.
\end{proof}

In \cite{ICERM}, it was shown that both $\Ss^U$ and $\Ss^L$ are perfect, so it is enough to show that $S^\sharp$ preserves Cremona equivalence, but there is an equally nice argument that $R^\sharp$ preserves Cremona equivalence.
\begin{lemma}\label{lem:CrR}
 Suppose $(d,m,p,q)$ are such that $d,m$ are defined by \eqref{eq:dmdioph}, and $p/q>7.$ Then, $(d,m,p,q)$ is Cremona equivalent to $R^\sharp(d,m,p,q).$
\end{lemma}
 
 \begin{proof}
 Let $R^\sharp(d,m,p,q)=(D,M,P,Q)$. 
Assume $p/q=[6+k;CF(x)]$ for some $k \geq 1$ and $x \geq 1$. Then, we have
\begin{align*} 
W(p/q)=(q^{6+k},p-(6+k)q,\hdots).
\end{align*} 
By Lemma~\ref{lem:symmCF}~(ii), $R(p/q)=(6p-35q)/(p-6q)= [6,k,CF(x)],$ and thus 
\begin{align*}
    W(P/Q)&=(Q^{\times 6},P-6Q^{\times k},Q-k(P-6Q),\hdots)\\
    &=(p-6q^{\times 6},q^{\times k},p-(6+k)q,\hdots)
\end{align*}
Only the first 6 terms of the weight expansions $W(P/Q)$ and $W(p/q)$ differ. Thus, to show that $(D,M,P,Q)$ is Cremona equivalent to $(d,m,p,q)$, we need to consider the degree coordinates and the first 6 terms of the weight sequence for each. 

We use the notation $\mapsto_{ijk}$ to represent applying $c_{ijk}$ to the previous tuple.
Applying two Cremona moves to $(D,M,P,Q)$ gives:
\begin{align}\label{eq:RPQ} (D,M,Q^{\times 6}) &\mapsto_{123} (2D-3Q,M,D-2Q^{\times 3}, Q^{\times 3})\nonumber \\ &\mapsto_{456} (4D-9Q,M,D-2Q^{\times 3},2D-5Q^{\times 3}).
\end{align} 
Applying three Cremona moves to $(d,m,p,q)$ gives:
\begin{align*}
    (d,m,q^{\times 6}) & \mapsto_{012} (2d-m-2q,d-2q,d-m-q^{\times 2},q^{\times 4}) \\
    &\mapsto_{034} (3d-2(m+2q),2d-m-4q,d-m-q^{\times 4},q^{\times 2})\\
    &\mapsto_{156} (5d-3(m+3q),2d-m-4q,3d-2(m+3q),d-m-q^{\times 3},2d-m-4q^{\times 2}), 
\end{align*}
which we can reorder to get 
\begin{align} \label{eq:Rpq}
    (5d-3(m+3q),3d-2(m+3q),2d-m-4q^{\times 3},d-m-q^{\times 3}).
\end{align}

We claim this reordered tuple is precisely \eqref{eq:RPQ}. To see this, we write each term in each of the tuples in terms of $p$ and $q$.
We are going to assume that for $(d,m,p,q)$, $\eps=1.$ By by \eqref{eq:dmdioph}, Lemma~\ref{lem:symmCF}~(iii), and the definition of $R$, we can write $(d,m,p,q)$ and $(D,M,P,Q)$ completely in terms of $p$ and $q$ as follows:
\begin{align*}
(d,m,p,q)&=(\frac18(3(p+q)+t),\frac18(p+q+3t),p,q) \\
(D,M,P,Q)&=(\frac18(21p-123q-t),\frac18(7p-41q-3t),6p-35q,p-6q).
\end{align*}

Now, we expand all of the entries in \eqref{eq:Rpq} and \eqref{eq:RPQ} in terms of $p$ and $q$ to get the following equalities:
\begin{align*}
    5d-3(m+3q)&=\frac12(3p-15q-t)=4D-9Q \\
    3d-2(m+3q)&=\frac18(7p-41q-3t)=M\\
    2d-m-4q&=\frac18(5p-27q-t)=D-2Q \\
    d-m-q&=\frac14(p-3q-t)=2D-5Q
\end{align*}
This completes the proof. 
 \end{proof}

\begin{prop}\label{prop:perf}  
For each $T \in \Gg$, the classes in the pre-staircase family $T^\sharp(\Ss^U)$ are perfect, that is, they are exceptional classes. 
\end{prop}
\begin{proof} 
By \cite[\S3.4]{ICERM}, this holds for $\Ss^U$ and $\Ss^L$.   Hence this is an immediate consequence of 
Lemma~\ref{lem:CrS}.
\end{proof}

Here is a typical corollary.  For simplicity we only consider the principal blocking classes mentioned in Lemma~\ref{lem:princ}, but a similar argument applies to all blocking classes that have associated perfect staircases.

\begin{cor}\label{cor:block} For each $i\ge 0$, the $z$-interval blocked by 
the principal blocking class $(S^i)^\sharp(\bB^U_0)$ contains $[w_{i+2},w_{i+1}]$.
\end{cor}
\begin{proof} Because the pre-staircases $(S^{i+1}R)^\sharp(\Ss^U_0), (S^{i})^\sharp(\Ss^U_0), $
that are associated to $\bB$ consist of perfect classes, we know from \cite[Lem.27]{ICERM} that their $z$-limit points 
$\al_{\bB,\ell}, \al_{\bB,u}$ are unobstructed and that the interval $(\al_{\bB,\ell}, \al_{\bB,\ell})$ lying between them  
 is precisely the $z$-interval blocked by $\bB$.
 Thus it suffices to show that 
$$
\al_{\bB,\ell}< w_{i+2} < w_{i+1} <  \al_{\bB,u}.
$$
Since $S$ preserves order and takes $w_i$ to $w_{i+1}$ for all $i$,  we only need  check this for $i=0$, and this can be done either by direct evaluation or by comparing the continued fraction expansions of these quantities.
\end{proof}

\subsection{Recognizing staircases}

We proved the following staircase recognition theorem in \cite[Thm.51]{ICERM}.

  \begin{thm}\label{thm:stairrecog}  Let  $\Ss = (\bE_\ka )$  be a perfect pre-staircase, let $\la$ be as in Lemma~\ref{lem:recur}, and denote by
 $D,M,P,Q$ the constants $X$ defined by \eqref{eq:recurX}, where $x_\ka  = d_\ka , m_{\ka }, p_\ka , q_\ka $ respectively.
     Suppose in addition 
that at least one of the following conditions holds:
\MS

 \NI {\rm (i)} There is $r/s>0$ such that $ M/D<r/s$,
 $$
  \frac{m_\ka ^2-1}{d_\ka m_\ka } <  \frac MD<  \frac{s + m_\ka (rd_\ka -sm_\ka )}{r+d_\ka (rd_\ka -sm_\ka )},\qquad \forall \ 
\ka \ge \ka _0,
$$
and there is no 
overshadowing class 
at $(z,b)=(P/Q,M/D)$ of degree $d'< s/(r-sb_\infty)$ and with $m'/d' >r/s$. 
\MS

 \NI {\rm (ii)}  There is $r/s>0$ such that $M/D >r/s$,
  $$
   \frac{m_\ka (sm_\ka  - rd_\ka )-s}{d_\ka (sm_\ka -rd_\ka ) - r} <  \frac MD< \frac{m_\ka }{d_\ka } \qquad \forall \ 
\ka \ge \ka _0,
 $$ 
and there is no 
overshadowing class 
at $(z,b)=(P/Q,M/D)$ of degree $d'< s/(sb_\infty - r)$ and with $m'_\ka /d' <r/s$. 

Then  $\Ss$ is  live, and it is a staircase for $H_{M/D}$ that accumulates at $P/Q$. 
\end{thm}

The next result states the estimates that we must establish in order to apply the above theorem.

\begin{lemma}\label{lem:DMineq} 
	 Consider a pre-staircase with classes $\bigl(d_\ka ,m_\ka ;q_\ka \bw(p_\ka /q_\ka )\bigr)$, where the ratios $b_\ka : =m_\ka /d_\ka $ 
have  limit $b_\infty$, and let the constants $D,D',D'',M,M',M'', \si$ be as in \eqref{eq:recurX}, with $x_\ka  = d_\ka , m_\ka $ respectively. 
\begin{itemize}\item[{\rm (i)}]  Suppose that $M\ov D - \ov M D\ne 0$.  Then the $b_\ka $ are strictly increasing iff
$$
M\ov D - \ov M D = 2\sqrt{\si}(M''D'-M'D'') = \frac{m_1d_0 -m_0 d_1}{\sqrt{\si}}> 0,
$$
and otherwise they are strictly decreasing.
\item[{\rm (ii)}]  if $m_1d_0 - d_1 m_0 > 0$, $b_\infty<r/s \le 1$,  and
\begin{align}\label{eq:DMineq}
|m_1d_0 - d_1 m_0| \le \sqrt {\si} \frac{sD-rM}{|rD-sM|},
\end{align}
then
 there is $\ka _0$ such that
$$
\frac{m_\ka }{d_\ka } \le b_\infty = \frac MD  \le 
\frac{s+ m_\ka (rd_\ka -sm_\ka )}{r+d_{\ka }(rd_\ka -sm_\ka )},\quad \mbox { for } \;\; \ka \ge \ka _0.
$$
\item[{\rm (iii)}]  
if $M''D'-M'D'' < 0$, $b_\infty>r/s > 0$, and   \eqref{eq:DMineq} holds,
then  there is $\ka _0$ such that 
$$
 \frac{m_\ka (sm_\ka -rd_\ka )-s}{d_\ka (sm_\ka -rd_\ka ) -r}\le  b_\infty = \frac MD  \le \frac{m_\ka }{d_\ka }, \quad  \mbox { for } \;\; \ka \ge \ka _0, 
$$
\end{itemize}
 \end{lemma}
\begin{proof}  This is a reformulation of \cite[Lem.67]{ICERM} in which (i) incorporates the calculation 
$$
M''D'-M'D'' = \frac{m_1d_0 - d_1 m_0}{2\si}
$$
 that follows from \eqref{eq:recurX}. \end{proof} 
 
 In \cite[\S3]{ICERM}, an $r/s$ was carefully chosen in order to reduce the number of potential overshadowing classes that needed to be ruled out. For our purposes, we simply need to know an $r/s$ exists since we use an arithmetic argument to rule out overshadowing classes. The following corollary shows that some $r/s$ does indeed exist.

 \begin{cor} \label{cor:rs}
 There exists an $r/s$ such that either the condition (ii) or (iii) in Lemma~\ref{lem:DMineq} is satisfied. 
 \end{cor}
 \begin{proof}
 For condition (ii), assume $m_1d_0-d_1m_0$ is positive. As $r/s$ approaches $b_\infty$ from the right $\frac{sD-rM}{|rD-sM|}=\frac{1-\tfrac{r}{s}b_\infty}{r/s-b_{\infty}}$ approaches infinity. Hence, there always exists some $b_\infty<r/s$ such that condition (ii) is satisfied. A similar argument applies for condition (iii). 
 \end{proof}

\begin{rmk} \rm  The statement in Lemma~\ref{lem:DMineq} is not identical to that in  \cite[Lem.67]{ICERM}
because we have changed notation, now indexing by $\ka$ rather than by $k$.  We now explain 
the relation between the two statements.
In \cite[Rmk.68(i)]{ICERM}, we calculated that
$$
M''D'-M'D'' = \frac 1{2(2n+3)\sqrt {\si}}(m_1d_0 - m_0 d_1).
$$  
However, 
there the staircases were divided into two parts according to the different endings, and (as in Lemma~\ref{lem:recur00}) the recursion parameter was $\si + 2$ where $\si = (2n+1)(2n+5) = \nu^2-4$.
 This reformulation should not change the limiting values $M,D$ and hence $M',M'', D', D''$.  However  in our current notation $m_{k=1},d_{k=1}$ are denoted $m_{\ka=2}, d_{\ka=2}$, since the staircases in \cite{ICERM} are indexed by $k$ while in the current paper we combine the two strands and  index via $\ka$, which is (approximately) $2k$.  Thus if we index by $\ka$ and take
 $2n+3 = \nu$ we have
\begin{align*}
& m_2d_0 - m_0 d_2 = (\nu m_1 - m_0)d_0 - m_0(\nu d_1 - d_0) = \nu(m_1d_0 - m_0 d_1),\\
& \frac 1{2(2n+3)\sqrt {\si}}(m_2d_0 - m_0 d_2) = \frac 1{2\sqrt {\si}} (m_1d_0 - m_0 d_1).
\end{align*}
which is consistent with the identities in Lemma~\ref{lem:DMineq}. \hfill$\er$
\end{rmk}

Corollary~\ref{cor:rs} establishes that we do not need to compute $m_1d_0-d_1m_0$ for estimating $r/s$. By Lemma~\ref{lem:DMineq}~(i), the sign of this quantity determines whether the terms $b_\ka=m_\ka/d_\ka$ strictly increase or decrease for a pre-staircase, which is important for the overshadowing argument. 

While the quantities $m_\kappa,d_\kappa$ do depend on both $i$ and $n$, we suppress those indices for simplicity.

\begin{lemma} \label{lem:det}

\NI {\rm (i)} For $S^i(\Ss^U_{\ell,n})$, $m_1d_0-m_0d_1= -\eps (2ny_i+y_{i+1})$, where $\eps = (-1)^i$.

\NI {\rm (ii)} For $S^i(\Ss^U_{u,n})$, $m_1d_0-m_0d_1=  -\eps(2ny_{i+1}+y_{i+2})$, where $\eps = (-1)^i$.

\NI {\rm (iii)} For $S^i(\Ss^L_{\ell,n})$, $m_1d_0-m_0d_1= -\eps(2ny_{i+1}+y_i)$, where $\eps = (-1)^{i+1}$.

\NI {\rm (iv)}  For $S^i(\Ss^L_{u,n})$, $m_1d_0-m_0d_1= -\eps(2ny_{i+2}+y_{i+1})$, where $\eps = (-1)^{i+1}$.
\end{lemma}
\begin{proof} Let  $\bullet$ be one of  $(U,\ell),(U,u),(L,\ell),(L,u)$.
For two seeds $(p_{0},q_{0},t_{0},\eps)$ and $(p_{1},q_{1},t_{1},\eps)$
 of $S^i(\Ss_{\bullet,n})$, the degree formulas from  \eqref{eq:dmdioph} imply
 that 
\begin{align} \label{eq:det1}
m_1d_0-m_0d_1=\frac{\eps}{8}\bigl((p_{0}+q_{0})t_1-(p_{1}+q_{1})t_0\bigr).
\end{align}
Denote the seeds of  the initial staircase
in  $\Ss^U$ or $\Ss^L$ by $(p_0^\bullet,q_0^\bullet,t_0^\bullet)$ and $(p_1^\bullet,q_1^\bullet,t_1^\bullet)$, and 
note that $S^i=\begin{pmatrix} y_{i+1} & -y_{i} \\ y_i & -y_{i-1} \end{pmatrix}$, where $y_{i-1} =6y_i-y_{i+1}$.   Using  the invariance of $t_0^\bullet,t_1^\bullet$ under $S$ we obtain that
\begin{align*} m_1d_0-m_0d_1&=\frac{\eps}{8}\Bigl(\bigl(t_1^\bullet(p_0^\bullet+q_0^\bullet)-t_0^\bullet(p_1^\bullet+q_1^\bullet)\bigr)y_{i+1}\\
&\qquad +\bigl(t_1^\bullet(p_0^\bullet-q_0^\bullet)-t_0^\bullet(p_1^\bullet-q_1^\bullet)+6(t_0^\bullet q_1^\bullet-t_1^\bullet q_0^\bullet)\bigr)y_i\Bigr).
\end{align*}
Then, by substituting the relevant formulas for the seeds from Lemma~\ref{lem:recur1} and Lemma ~\ref{lem:recur2}, we get the desired results. 

For example, for (ii), we have 
\begin{align}\label{eq:seedU0}
(p^U_{0,u},q^U_{0,u},t^U_{0,u})=(-5,-1,2), \qquad (p^U_{1,u},q^U_{1,u},
t^U_{1,u}) =(2n+8,1,2n+5)
\end{align} 
 giving 
\begin{align*}
m_1d_0-m_0d_1
& =-\eps\bigl((2n+6)y_{i+1} - y_{i}\bigr)\ =\ 2n y_{i+1} + y_{i+2} 
\end{align*}  as claimed.
\end{proof}

\begin{cor} \label{cor:bk}
For a pre-staircase $T^\sharp(\Ss^U)$, if $b_\infty>1/3$, then $m_\ka/d_\ka$ strictly decrease. If $b_\infty<1/3$, then $m_\ka/d_\ka$ strictly increase.
\end{cor}
\begin{proof} Lemma~\ref{lem:det} implies that  $m_1d_0-m_0d_1$ always has the opposite sign to $\eps$.
As $\eps$ is positive if and only if $b_\infty>1/3$, it follows that  $m_1d_0-m_0d_1<0$ if and only if $b_\infty>1/3$. Now use Lemma~\ref{lem:DMineq}~(i). 
\end{proof}



Finally, we prove an estimate that will be useful below.

\begin{lemma}\label{lem:slope}  Let $\Ss$ be one of the descending stairs $S^i(\Ss^U_{u,n})$
for all  $i,n\ge 0$ except $(i,n) = (0,0)$ or one of
the stairs $S^i(\Ss^L_{u,n})$ for all  $i\ge 0, n\ge 1$.  Denote the steps of $\Ss$ by $(d_\ka ,m_\ka ,p_\ka ,q_\ka ,t_\ka ,\eps), \ka \ge 0,$ and 
let $b_\infty = \lim m_\ka /d_\ka $.
Then there is $\ka _0$ such that
\begin{align}\label{eq:MDTineq}
\frac{p_\ka  q_\ka }{p_\ka  + q_\ka } > \frac{d_\ka  - m_\ka b_\infty}{3-b_\infty},\quad \ka \ge \ka _0.
\end{align}
\end{lemma}
\begin{proof} 
Since 
$(3d_\ka -m_\ka )p_\ka q_\ka  = (p_\ka +q_\ka )(d^2_\ka  - m^2_\ka  + 1)$, we have
\begin{align*}
\frac{p_\ka +q_\ka }{d_\ka } & = \Bigl(3- \frac{m_\ka }{d_\ka }\Bigr) p_\ka q_\ka  - (p_\ka +q_\ka )\bigl(d_\ka  - \frac{m^2_\ka }{d_\ka }\bigr)\\
& = (3-b_\infty)p_\ka q_\ka  - (p_\ka  + q_\ka )(d- m_\ka  b_\infty) - \bigl( \frac{m_\ka }{d_\ka } - b_\infty\bigr)\bigl(p_\ka q_\ka -m_\ka  (p_\ka +q_\ka )\bigr).
\end{align*}
Therefore \eqref{eq:MDTineq} will hold provided that
 \begin{align*}
\frac{p_\ka +q_\ka }{d_\ka }> \bigl( \frac{m_\ka }{d_\ka } - b_\infty\bigr)\bigl(m_\ka  (p_\ka +q_\ka )-p_\ka q_\ka \bigr), \quad \ka \ge \ka _0.
\end{align*}
Since $$
m_\ka  (p_\ka +q_\ka )-p_\ka q_\ka  = m_\ka  (3d_\ka -m_\ka )- (d^2_\ka  - m^2_\ka  + 1)  < m_\ka  (3d_\ka -m_\ka )- (d^2_\ka  - m^2_\ka )
$$
it suffices to show that
 \begin{align*}
\frac{3d_\ka -m_\ka }{d_\ka }> \bigl( \frac{m_\ka }{d_\ka } - b_\infty\bigr)\bigl(m_\ka  (3d_\ka -m_\ka )- (d^2_\ka  - m^2_\ka )\bigr), \quad \ka \ge \ka _0.
\end{align*}
Since by Lemma~\ref{lem:recur} we have $m_\ka  = M\la^\ka  + \ov M \la^{-\ka }, \ d_\ka  = D\la^\ka  + \ov D \la^{-\ka }$ for some $\la>1$, one easily checks that
this will hold if
 \begin{align}\label{eq:MDT}
 \frac{3D-M}{3M-D} >  |\ov M D - \ov D M| = \frac{|m_1d_0 - d_1 m_0|}{\sqrt \si}, \quad \si = \nu^2-4,
  \end{align}
  where $\nu$ is the recursion parameter of $\Ss$.
%
  \MS
  
    \NI {\bf Claim 1:} {\it \eqref{eq:MDT} holds when $\Ss  = (S^i)^{\sharp}\bigl( \Ss^U_{u,n}\bigr)$ for all $ i, n \ge 0$ except $(i,n) = (0,0)$.}
    \begin{proof}  
   First consider the case $\Ss = \Ss^U_{u,n}$.   When $X = P, Q ,T$, Lemma~\ref{lem:recur} implies that  $X = \frac 1{2\sqrt \si}\bigl( 2x_1 - x_0(\nu-\sqrt\si)\bigr)$, 
    where $\nu = 2n+3$ is the recursion variable, and $\si = \nu^2 - 4 = (2n+1)(2n+5)$. 
    The seeds for this staircase are given in \eqref{eq:seedU0}.
    Since  $\nu-\sqrt\si > 0$  and the seed class has $p_0+q_0 < 0$ while $t_0 > 0$, we have
     \begin{align*}
 \frac{3D-M}{3M-D} & = \frac{P+Q}{T} = \frac{ 2(p_1+q_1) - (p_0+q_0)(\nu-\sqrt\si)}{ 2t_1 - t_0(\nu-\sqrt\si)} \\
 & > \frac{ (p_1+q_1)}{ t_1}  = \frac{2n+9}{2n+5} \\
 & > \frac {m_0d_1 - m_1d_0}{\sqrt \si} = \frac{2n + 6}{\sqrt{(2n+1)(2n+5)}}
  \end{align*}
  provided that $(2n+9)^2(2n+1) > (2n+6)^2(2n+5)$.  But this  holds unless $n = 0,1,2,3$.  If 
  $n=1,2,3$ then 
 one can check by direct calculation that \eqref{eq:MDT} holds.  Moreover, although it does not hold when $n=0$ one can check that in this case we have
      \begin{align}\label{eq:case0}
 \frac{3D-M}{3M-D} >  \frac{34}{35}   \ \frac {m_0d_1 - m_1d_0}{\sqrt \si} =  \frac{y_4}{y_2 y_3}\ 
  \frac {6}{\sqrt \si}.
 \end{align}
  \MS
  
  Now suppose that $\Ss  = (S^i)^{\sharp}\bigl( \Ss^U_{u,n}\bigr)$.
     Notice that $\si$ and $\nu$ are invariant under the shift.  The quantity  $3M-D$
     is also invariant under the shift since it is the limit of 
  $(3m_\ka  - d_\ka) \la^{-\ka}  = \eps t_\ka \la^{-\ka} $.  To consider how much the right hand side of \eqref{eq:MDT} increases under iterations of the shift, we again use Lemma~\ref{lem:det}.  Because 
   $y_{i+2} < 6y_{i+1}$ for all $i > 0$ 
   we can estimate $$
   \frac{ 2ny_{i+1} + y_{i+2}}{2ny_{1} + y_{2}} =    \frac{ 2ny_{i+1} + y_{i+2}}{2n + 6}  < y_{i+1}, \quad \forall \  i>0, \ n \ge 0
   $$
   Therefore the result will hold in the case $n>0$ if we show that when we apply $S^i$ the quantity $3D-M$ increases by a factor of at least $y_{i+1}$.
But by Lemma~\ref{lem:Phi0}, $3D-M = P+Q$ is taken by $S^i$ to $(y_{i+1} + y_i)P- (y_i+ y_{i-1}) Q$. 
 Therefore we need $(y_{i+1} + y_i)P- (y_i+ y_{i-1}) Q > y_{i+1}(P+Q)$, or equivalently
$$
y_iP  > (y_{i+1} + y_i+ y_{i-1})Q = 7 y_iQ.
$$
But this holds because $P/Q$ is the accumulation point of a staircase in $\Ss^U$ and so satisfies  $P/Q> 7$.  Indeed $w_1 = 7$ is blocked by $\bB^U_0$ by Corollary~\ref{cor:block}.

In the case $n = 0$, the the right hand side of \eqref{eq:MDT} increases by the factor $y_{i+2}/6$.
Therefore
 it suffices to show that when we apply $S^i$ the quantity $3D-M$ increases by a factor of at least $\frac{35}{34}\ \frac{y_{i+2}}{6} $.  Since $\frac{35}{34}\ \frac{y_{i+2}}{6} \le y_{i+1}$ when $i\ge 2$, the previous argument applies to show that this holds.   In the case $i = 1$, we must check that
 $$
 7P - Q>  \frac{35}{34}\ \frac{35}{6} (P+Q),
 $$
 which holds because 
 $P/Q = [7;\{5,1\}^\infty]  > [7;6] = 43/6.$    
This completes the proof.  \end{proof}
  \MS
  
   \NI {\bf Claim 2:} {\it \eqref{eq:MDT} holds when $\Ss  = (S^i)^{\sharp}\bigl( \Ss^L_{u,n}\bigr)$, $n\ge 1$ and $i\ge 0$}
   \begin{proof}  Let $\Ss =  \Ss^L_{u,n}$.  
   The initial seeds are now  
   $(p^L_{0,u},q^L_{0,u},t^L_{0,u})=(-29,-5,2 )$ and $(p^L_{1,u},q^L_{1,u},
t^L_{1,u}) = (12n-1, 2n-2, 2n+1)$.  Further $ \nu =  2n+3$ so that $\si = \sqrt{(2n+1)(2n+5)}$ as before.
If we simplify the inequality as before and use the result in Lemma~\ref{lem:det}~(iv), it follows that it suffices to check that
$$
 \frac{ (p_1+q_1)}{ t_1}  > \frac {m_1d_0 - m_0d_1}{\sqrt \si} =\frac{ 2n y_2 + y_1}{\sqrt \si} =
 \frac{ 12n + 1}{\sqrt \si}.
 $$
   But this holds for all $n\ge1$. 
\MS

Next consider the case $i>0$. We argue as before, noting that for each $n$ the accumulation point  $P/Q$ is larger than the center of the blocking class $\bB^L_n$, which is $\frac {12n+1}{2n}$ by Theorem~\ref{thm:Ll}.   
By Lemma~\ref{lem:det}~(iv), when we apply $S^i$  the right hand side grows by the factor $$
\frac{2n y_{i+2} + y_{i+1}}{2n y_2 + y_1} = \frac{ 2ny_{i+2} + y_{i+1}}{12n+1} , \qquad n\ge 1.
$$
Therefore it suffices to check that
\begin{align*}
(y_{i+1} + y_i)P- (y_i+ y_{i-1}) Q& > \frac{2n y_{i+2} + y_{i+1}}{12n+1}(P+Q)\\&
 = \Bigl(y_{i+1} - \frac{2n}{12n+1}y_i \Bigr)(P+Q).
\end{align*}
Because $y_{i+1} + y_i + y_{i-1} = 7 y_i$, this  simplifies to
$$
 \frac{14n+1}{12n+1}y_i P > y_i\Bigl(7-\frac {2n}{12n+1}\Bigr)Q.
 $$
   Thus we need $P/Q> \frac{82n+7}{14n+1}$.  But this holds for all $n\ge 1$ since $P/Q> \frac {12n+1}{2n}$.   
   \end{proof}
%
%
%

The proof of the lemma is now complete.
\end{proof}

 \subsection{Overshadowing classes} \label{ss:os}
 
To utilize the Staircase Recognition Theorem~\ref{thm:stairrecog}, it remains to show there are no overshadowing classes.   
Here, we not only show that there are no overshadowing classes for $S^iR^\delta(\Ss^U)$, but prove a more general result about overshadowing classes. Namely, given a general perfect pre-staircase family with recursion parameter $\ge 3$, if $b_\infty<1/3$ and $m_\ka/d_\ka$ strictly increase (resp. $b_\infty>1/3$ and $m_\ka/d_\ka$ strictly decrease), then the staircase family is live provided only that the
staircase is not overshadowed by the obstruction from the exceptional class $\bE_1:=3L-E_0-2E_1-E_{2\ldots 6}$.   Recall from
 Remark~\ref{rmk:lowdeg}~(ii) and Lemma~\ref{lem:volacc} that when  $b \in (1/5,5/11)$ this obstruction is given by the formula $z\mapsto  \frac{1+z}{3-b}$, and hence always passes through the accumulation point  
$\bigl(\acc(b),V_b(\acc(b))\bigr)$. Hence, it could overshadow the descending staircases for these $b$. 
Since the $\Ss^U$ staircases have $b> 5/11$ while the $\Ss^L$ staircases have $b< 1/5$, 
we only need be concerned about  their images under a shift $S^i$.   
The next lemma shows that these staircases are not overshadowed in this way, 
 because 
the slope of any overshadowing class must be greater than the slope of the obstruction  $\mu_{\bE_1,b}$.

\begin{lemma}\label{lem:slope1}  Let $\Ss$ be any descending pre-staircase in one of the families $(S^iR^\de)^{\sharp} (\Ss^U),$ where $i>0, \de\in \{0,1\}$. 
Then the slope of an overshadowing class 
 must be larger than $\frac{1}{3-b_\infty}$. 
\end{lemma}
\begin{proof}
The slope of an overshadowing class must be larger than the slope of the line segments from the accumulation point $(z_\infty,b_\infty)$ to the outer corners $\bigl(\frac{p_\ka}{q_\ka}, \frac{p_{\ka}}{d_{\ka}-m_{\ka}b_{\infty}}\bigr)$ of the staircase.
Therefore we must check that
 \[ 
 \frac{\frac{p_{\ka}}{d_{\ka}-m_{\ka}b_{\infty}}-\frac{1+z_{\infty}}{3-b_{\infty}}}{\frac{p_{\ka}}{q_{\ka}}-z_{\infty}} > \frac{1}{3-b_\infty}.
 \]
When we simplify $z_\infty$ cancels from the inequality, and we get
\[ 
(3-b_\infty)p_\ka q_\ka>(p_{\ka}+q_{\ka})(d_\ka-m_\ka b_\infty),
\] 
which was proved in Lemma~\ref{lem:slope}.
\end{proof}

We now use an arithmetic argument to rule out the existence of any other overshadowing classes. The following lemma establishes properties about the common divisors of $(d,m,p,q)$ needed for the argument. 

\begin{lemma} \label{lem:relPr}
Given a quasi-perfect class $(d,m,p,q)$, 
assume there is an integer $k$ such that $\frac{km}p$ and $\frac{kd}p$ (resp. $\frac{km}q$ and  $\frac{kd}q$) are both integers.  Then $p | k$ (resp. $q|k$). 
\end{lemma}
\begin{proof}
Assume $\frac{km}p$ and  $\frac{kd}p$ are both integers. Let $c_m=\gcd(p,m)$ and $c_{d}=\gcd(p,d)$. As $q=3d-p-m$ and $\gcd(p,q)=1$, we must have $\gcd(p,d,m)=1$. Thus, $\gcd(c_m,d)=\gcd(c_d,m) = \gcd(c_m,c_d)= 1$. 

Let $c_m'$ (resp. $c_d'$) be the largest product of primes occurring in $c_m$ (resp. $c_d$) that divides $p$. Then, by the above, we can write $p=c_m'c_d'p'$ where $\gcd(c'_m,c'_d)=1$ and $p'$ involves primes not dividing either $c_m$ or $c_d$, so that $\gcd(p',m)=gcd(p',d)=1$. 

Then, by unique factorization $p | k$ if and only if $c_m', c_d',$ and $p'$ all divide $k$. But 
because $\frac{km}p \in \Z$, both $c_d'$ and $p'$ divide $k$.  Similarly, $c_m'|k$.  
Hence  $c_m', c_d',$ and $p'$ all divide $k$ as required.
The case for $q$ follows similarly.
\end{proof}

Suppose that at $(z_\infty,b_\infty)$ there is an overshadowing class $\bE_{ov}= (d_{ov}, m_{ov}, \bbm)$ for some pre-staircase.  Recall from Corollary~\ref{cor:monot} that  the limit points $z_\infty, b_\infty$ are irrational.
By \cite[Prop.2.3.2]{ball}, there are integers $A\ge 0, C>0$ such that
$$
\mu_{\bE_{ov}, b} (z) = \frac{A +Cz}{d_{ov} - m_{ov}b}, \qquad z\approx z_\infty.
$$
Further, if $\bE_{ov}$ has break point $a_{ov} = p_{ov}/q_{ov}$ then\footnote
{
In fact, the argument given below does NOT use the estimates for $A,C$, but rather arithmetic bounds obtained 
because the obstruction $\mu_{\bE_{ov}, b}$   goes through the accumulation point.}
\begin{align}\label{eq:obstr}
  A\ge 0,\;\; C\ge q_{ov} &\mbox{ when } z_\infty < a_{ov},\\ \notag
A\ge p_{ov},\;\; C\ge 0 &\mbox{ when } z_\infty > a_{ov}.
\end{align}

We denote by $\ell_b$ the line $z\mapsto \ell_b(z) = (1+z)/(3-b)$. By \eqref{eq:same}, this has the property
that $\ell_b(\acc(b)) = V_b(\acc(b))$, i.e. for each $b$  its graph crosses the volume curve at the accumulation point.

\begin{lemma}\label{lem:live1}   Let $\Ss$ be a descending pre-staircase with irrational accumulation point $z_\infty$ that is associated to
a perfect blocking class
$\bB = (d_{\bB},m_{\bB},p_{\bB},q_{\bB},t_{\bB})$.  Suppose that  
 $\bE_{ov}= (d_{ov}, m_{ov}, \bbm)$ is an overshadowing class,
and denote by  
$$
\mu_{\bE_{ov}, b} (z) = \frac{A +Cz}{d_{ov} - m_{ov}b}.
$$
the corresponding obstruction.  
Assume the slope $C/(d_{ov} - m_{ov}b_\infty)$ of $\mu_{\bE_{ov}, b_\infty} (z) $ is $> 1/(3-b_\infty)$. 
Then there is a positive integer $k$ such that
\MS

\NI
{\rm (i)}  If $b_\infty>1/3$ and $m_{\bB}/d_{\bB}>1/3$, then $m_{ov}/d_{ov}<1/3$ and $d_{ov}-3m_{ov}=kt_{\bB}$
\MS

\NI {\rm (ii)} If $b_\infty<1/3$ and $m_{\bB}/d_{\bB}<1/3$, $m_{ov}/d_{ov}>1/3$ and $3m_{ov}-d_{ov}=kt_{\bB}$. 
\end{lemma}
\begin{proof}   Note first that the function $\mu_{\bE_{ov},b_\infty}$ must obstruct an interval $(z_\infty, z_\infty + \eps)$  to the right of the limit point and hence have break point $a_{ov} > z_\infty$.  Hence \eqref{eq:obstr} implies that $A\ge 0$ and $C\ge q_{ov}$.  
 Further, the condition on the slope implies that  $C>A$.

Next note that
$$
\frac {1+z_\infty}{3-b_\infty} = \frac {p_{\bB}}{d_{\bB}-m_{\bB}b_\infty} =  \frac{A +Cz_\infty}{d_{ov} - m_{ov}b_\infty}.
$$

This implies
\begin{align*}
&(1+z_\infty)(d_{\bB}-m_{\bB}b_\infty) =  p_{\bB}(3-b_\infty), \\
&\qquad \Longrightarrow z_\infty =  \frac{p_{\bB}(3-b_\infty) - (d_{\bB}-m_{\bB}b_\infty)}{d_{\bB}-m_{\bB}b_\infty} = \frac{b_\infty(m_{\bB}-p_{\bB}) + 3p_{\bB}-d_{\bB}}{d_{\bB}-m_{\bB}b_\infty}\\
&\mbox{ and }\\
&(A+Cz_\infty)(d_{\bB}-m_{\bB}b_\infty) =  p_{\bB}(d_{ov} -  m_{ov} b_\infty )\\
&\qquad \Longrightarrow z_\infty =  \frac{p_{\bB}(d_{ov} - b_\infty m_{ov}) - A(d_{\bB}-m_{\bB}b_\infty)}{C(d_{\bB}-m_{\bB}b_\infty)} = 
\frac{b_\infty(Am_{\bB}-p m_{ov}) + p_{\bB} d_{ov}- Ad_{\bB}}{C(d_{\bB}-m_{\bB}b_\infty)}
\end{align*}
Therefore
$$
b_\infty(Am_{\bB}-p_{\bB} m_{ov}) + p_{\bB} d_{ov}- Ad_{\bB} =C\bigl(b_\infty(m_{\bB}-p_{\bB}) + 3p_{\bB}-d_{\bB}\bigr).
$$
All quantities here are integers except for $b_\infty$ which is irrational because $z_\infty$ is.  Therefore the coefficients of $b_\infty$ must be equal.  Thus we have
$$
Am_{\bB}-p_{\bB} m_{ov} = C(m_{\bB}-p_{\bB}), \qquad p_{\bB} d_{ov}- Ad_{\bB} = C(3p_{\bB}-d_{\bB}).
$$
We can solve these equations for $m_{ov}, d_{ov}$ to get
\begin{align}\label{eq:mdC}
&m_{ov} = \frac{(A-C)m_{\bB} + Cp_{\bB}}{p_{\bB}} = C + (A - C)m_{\bB}/p_{\bB},\\ \notag
    & d_{ov} = \frac{3p_{\bB}C + (A-C)d_{\bB}}{p_{\bB}} = 3C +  (A - C)d_{\bB}/p_{\bB}.
\end{align}
As $m_{ov},d_{ov}$ are integers, $ (A - C)m_{\bB}/p_{\bB}$ and $(A - C)d_{\bB}/p_{\bB}$ are both integers.
 Then, by Lemma~\ref{lem:relPr}, $p_{\bB}|(A-C)$.  This proves $k=(C-A)/p_{\bB}$ is an integer.

 To see $k > 0$, we have an equation of the form $(A+Cz)/\la = (1 + z)/\la'$ and assume that $C/\la > 1/\la'$, i.e. $C\la'>\la$.
Then  $$
\la + \la z = A\la' + Cz\la' > A\la' + \la z,
$$
 so that $A \la'< \la < C\la'$.  In other words, $A<C$.
 
The formulas in \eqref{eq:mdC} then imply that if $b_\infty>1/3$, so that we also have $m_{\bB}/d_{\bB}>1/3$  by Corollary~\ref{cor:bk}, then $3m_{ov}<  d_{ov}$. 
We can compute
\[d_{ov}-3m_{ov}=3C-kd_{\bB}-3C+3km_{\bB}=k(3m_{\bB}-d_{\bB})=kt_{\bB}.\] This proves (i). 
The proof for (ii) follows similarly. 
\end{proof}

We prove a similar result for ascending staircases.

\begin{lemma}\label{lem:live2}   Let $\Ss$ be an ascending pre-staircase with irrational accumularion point $z_\infty$  that is associated to
a perfect blocking class
$\bB = (d_{\bB},m_{\bB},p_{\bB},q_{\bB},t_{\bB})$. Suppose that  
 $\bE_{ov}= (d_{ov}, m_{ov}, \bbm)$ is an overshadowing class, 
and denote by  
$$
\mu_{\bE_{ov}, b} (z) = \frac{A +Cz}{d_{ov} - m_{ov}b}.
$$
the corresponding obstruction. Then there is a positive integer $k$ such that
\MS

\NI
{\rm (i)}  If $b_\infty>1/3$ and $m_{\bB}/d_{\bB}>1/3$, then $m_{ov}/d_{ov}<1/3$ and $d_{ov}-3m_{ov}=kt_{\bB}$
\MS

\NI {\rm (ii)} If $b_\infty<1/3$ and $m_{\bB}/d_{\bB}>1/3$, then $m_{ov}/d_{ov}>1/3$ and $3m_{ov}-d_{ov}=kt_{\bB}$. 
\end{lemma}

\begin{proof} 
Notice first that we must have $A-C>0$ since if $A\le C$ the slope of the obstruction $\mu_{E_{ov},b_\infty}$ is at most that of the line $z\mapsto \frac{1+z}{3-b_\infty}$ which is $< V_{b_\infty}(z)$ for $z< z_\infty$, and so is not obstructive for $z< z_\infty$,  while the overshadowing class is obstructive. 
\MS

Next, as 
 in Lemma~\ref{lem:live1} we consider the equalities
$$
\frac {1+z_\infty}{3-b_\infty} = \frac {q_{\bB}z_\infty}{d_{\bB}-m_{\bB}b_\infty} =  \frac{A +Cz_\infty}{d_{ov} - m_{ov}b_\infty}.
$$
This implies
\begin{align*}
&(1+z_\infty)(d_{\bB}-m_{\bB}b_\infty) =  q_{\bB} z_\infty(3-b_\infty), \\
&\qquad \Longrightarrow {z_\infty}  =  \frac{d_{\bB}-m_{\bB}b_\infty} {q_{\bB} (3-b_\infty) - (d_{\bB}-m_{\bB}b_\infty)}
&\mbox{ and }\\
&(A+Cz_\infty)(d_{\bB}-m_{\bB}b_\infty) =  q_{\bB} z_\infty(d_{ov} -  m_{ov} b_\infty )\\
&\qquad \Longrightarrow {z_\infty}  = \frac {A(d_{\bB}-m_{\bB}b_\infty)}{ q_{\bB}(d_{ov} -  m_{ov} b_\infty )- C(d_{\bB}-m_{\bB}b_\infty)}
\end{align*}
Hence we must have
$$
A (q_{\bB} (3-b_\infty) - (d_{\bB}-m_{\bB}b_\infty)) = q_{\bB}(d_{ov} -  m_{ov} b_\infty )- C(d_{\bB}-m_{\bB}b_\infty).
$$
As before, the coefficients of $b_\infty$  on both sides must agree, which gives
\begin{align}\label{eq:dov}
 d_{ov} = 3A - (A-C)d_{\bB}/q_{\bB}, \qquad m_{ov} = A - (A-C)m_{\bB}/q_{\bB}.
\end{align}
Since $A-C>0$ and $d_{ov},m_{ov}$ are integers, Lemma~\ref{lem:relPr}~ imply $k=(A-C)/q_{\bB}$ is a positive integer. 

Furthermore, these formulas prove that if $b_\infty>1/3$ and hence
$m_{\bB}/d_{\bB}> 1/3$ then $d_{ov} > 3 m_{ov}$. We can again compute
\[ d_{ov}-3m_{ov}=k\,t_{\bB}.\]
This proves (i). The proof for (ii) follows similarly.
\end{proof}

We use the results of Lemma~\ref{lem:live1} and Lemma~\ref{lem:live2} to rule out overshadowing classes for a large set of pre-staircases, which implies that none of the pre-staircases considered in this paper have overshadowing classes.

\begin{lemma} \label{lem:OS}
Let $\Ss$ be a pre-staircase with irrational accumulation point $z_\infty$ that is associated to a perfect blocking class $\bB=(d_{\bB},m_{\bB},p_{\bB},q_{\bB},t_{\bB})$ such that $t_{\bB} \ge 3$. Suppose that either $b_\infty>1/3$ and the $m_\ka/d_\ka$ strictly decrease or $b_\infty<1/3$ and the $m_{\ka}/d_{\ka}$ strictly increase. Assume further that if $\Ss$ is descending any overshadowing class must have slope $> 1/(3-b_\infty)$.
Then, the pre-staircase $\Ss$ has no overshadowing classes at all. 
\end{lemma}
\begin{proof}
These conditions and Lemma~\ref{lem:slope1} ensure that the conditions in either Lemma~\ref{lem:live1} or Lemma~\ref{lem:live2} hold for $\Ss$. For a class $\bE_{ov}=(d_{ov},m_{ov},\bbm_{ov})$ to be overshadowing, it must be obstructive for $b= b_\infty$ at its break point $a_{ov}$. 
By \cite[Lem.15]{ICERM}
this is possible only if $| d_{ov}b_\infty -m_{ov}|<1$. 
This estimate and the results in Lemma~\ref{lem:live1} and Lemma~\ref{lem:live2} imply that for each staircase there is a positive integer $k$ such that
\[ 3>3|b_\infty d_{ov}-m_{ov}|>|d_{ov}-3m_{ov}|=t_{\bB}k\geq 3k.\]
As $k$ must be positive, no such $d_{ov}$ and $m_{ov}$ can exist. 
\end{proof}

\begin{cor} \label{cor:OS} 
For each $T \in \Gg$, the pre-staircase family $T^\sharp(\Ss^U)$ has no overshadowing classes. 
\end{cor}
\begin{proof}
Because $t_{\bB}=2n+3$ in all cases, the results of Corollary~\ref{cor:bk} and Lemma~\ref{lem:slope1} show that all the staircases in $T^\sharp(\Ss^U)$ except for $\Ss^U_{u,0}$ satisfy the conditions in Lemma~\ref{lem:OS} and hence have no overshadowing classes.    The proof for  $\Ss^U_{u,0}$ is given in \cite[Ex.70]{ICERM}.
\end{proof}

Now, we establish a straightforward  way to check if a perfect pre-staircase is live based on the above overshadowing arguments. 

\begin{prop}\label{prop:liveGen}
Let $\Ss$ be a perfect pre-staircase with irrational accumulation point $z_\infty$ associated to a blocking class $\bB$ with  recursion parameter $t_\bB\ge 3$. Suppose that either $b_\infty>1/3$ and the $m_\ka/d_\ka$ strictly decrease or $b_\infty<1/3$ and the $m_{\ka}/d_{\ka}$ strictly increase.  Assume further that if $\Ss$ is descending the slope of any overshadowing class must be $> 1/(3-b_\infty)$.   Then, $\Ss$ is a staircase, namely it is perfect and live. 
\end{prop}
\begin{proof}
By Corollary~\ref{cor:rs} and Lemma~\ref{lem:OS}, the perfect pre-staircase 
satisfies the conditions of Theorem~\ref{thm:stairrecog} implying that $\Ss$ is live. 
\end{proof} 

A consequence of this proposition is Theorem~\ref{thm:Gglive}.
\begin{cor}\label{cor:live} 
For each $T \in \Gg$, $T^\sharp(\Ss^U)$ is live.
\end{cor}
\begin{proof} For all staircases except $\Ss^U_{u,0}$ (which was shown to be live in \cite{ICERM}), the
conditions of Proposition~\ref{prop:liveGen} are satisfied since $t_{\bB}=2n+3$, and Corollary~\ref{cor:bk} and Lemma~\ref{lem:slope1} hold.
\end{proof}

\end{document}